\documentclass[10pt, a4paper]{amsart}

\usepackage{amsmath,amssymb,amsthm}
\usepackage{xcolor,textcomp}
\usepackage{graphicx, epstopdf}
\usepackage{braket,amsfonts}
\usepackage{dsfont}
\usepackage{mathtools}
\usepackage{mwe}
\usepackage[top=2.7cm,bottom=2.5cm,left=2.5cm,right=2.7cm]{geometry}
\usepackage{hyperref}
\hypersetup{colorlinks=true,linkcolor=red,citecolor=red,filecolor=magenta,urlcolor=blue}
\usepackage[capitalise]{cleveref}
\usepackage{comment}
\usepackage{tikz}
\usepackage{pgfplots}
\usepackage{stmaryrd}
\usepackage[mathscr]{euscript}
\addtocontents{toc}{\protect\setcounter{tocdepth}{1}}

\usepackage[sort,nocompress]{cite}

\renewcommand{\mod}[1]{\left\lvert#1\right\rvert}

\newcommand{\norm}[1]{\left\|#1\right\|}
\newcommand{\mb}[1]{\mathbb{#1}}
\newcommand{\mc}[1]{\mathcal{#1}}
\newcommand{\ip}[2]{\left\langle#1,#2\right\rangle}
 
\renewcommand{\vector}[2]{\begin{pmatrix}#1\\[4pt]#2\end{pmatrix}}
\newcommand{\B}{\mathcal B}
\newcommand{\E}{\mathcal E}
\DeclareMathOperator{\sgn}{sgn}

\DeclareMathOperator{\Span}{span}

\newtheorem{theorem}{Theorem}[section]
\newtheorem{lemma}[theorem]{Lemma}
\newtheorem{remark}[theorem]{Remark}
\newtheorem{corollary}[theorem]{Corollary}
\newtheorem{definition}[theorem]{Definition}
\newtheorem{proposition}[theorem]{Proposition}

\newtheorem{hypo}[theorem]{Hypothesis} 

\title[Boundary null-controllability of 1d Navier-Stokes system]{Boundary null-controllability of 1d linearized compressible Navier-Stokes System by one control force}
\author[K. Bhandari, S. Chowdhury, R. Dutta, J. Kumbhakar]{Kuntal Bhandari$^*$ \and Shirshendu Chowdhury$^{\dagger,1}$ \and Rajib Dutta$^{\dagger,2}$ \and Jiten Kumbhakar$^{\dagger,3}$}
\thanks{$^*$Laboratoire de Math\'ematiques Blaise Pascal, UMR 6620, Universit\'e Clermont Auvergne, CNRS,  63178 Aubi\`ere, France;   kuntal.bhandari@uca.fr.}
\thanks{$^\dagger$Indian Institute of Science Education and Research Kolkata, Campus road, Mohanpur, West Bengal 741246, India;
$^1$shirshendu@iiserkol.ac.in,  $^2$rajib.dutta@iiserkol.ac.in,  
$^3$jk17ip021@iiserkol.ac.in.}

\begin{document}
\begin{abstract}
In this article, we study the boundary null-controllability properties of the one-dimensional linearized (around $(Q_0,V_0)$ with constants $Q_0>0, V_0>0$) compressible Navier–Stokes equations in the interval $(0,1)$ when a control function is acting either on the density or velocity component at one end of the interval. We first prove that the linearized system, with a Dirichlet boundary control on the density component and homogeneous Dirichlet boundary conditions on the  velocity component, is null-controllable in ${H}^s_{per}(0,1)\times L^2(0,1)$ for any $s>\frac{1}{2}$ provided the time $T>1$, where ${H}_{per}^s(0,1)$ denotes the Sobolev space of periodic functions. The proof is based on solving  a mixed parabolic-hyperbolic  moments problem and to do so, we perform a spectral analysis for the associated adjoint operator which is the main involved part of this work. As a corollary, we have also shown that the  system is approximately controllable in $L^2(0,1)\times L^2(0,1)$ when $T>1$.

On the other hand, assuming that the density is equal on the two boundary points w.r.t. time,  when a control is applied on the velocity part through a Dirichlet condition, we can only able to prove that the system is  null-controllable in a strict subspace of finite codimension $\mathcal{H}\subset {H}_{per}^s(0,1)\times L^2(0,1)$ for $s>\frac{1}{2}$ when $T>1$.
More precisely, in this case we are able to prove that all the eigenfunctions of the associated adjoint operator are observable for high frequencies  whereas for the lower frequencies we are unable to conclude anything.  A parabolic-hyperbolic joint Ingham-type inequality which we prove in this article, leads to an observability inequality in the space $\mathcal{H}^*$ and the controllability result follows. 
The significant point is that the moments method  does not yield a better space for the null-controllability when a control acts on the velocity part.

\end{abstract}

\keywords{Linearized compressible Navier-Stokes system, boundary null-controllability, spectral analysis, moments method, Ingham-type inequality.}

\subjclass[2010]{76N06, 93B05,  93B60, 93C20.}

\maketitle

\tableofcontents

\section{Introduction and main results}

\smallskip 
\subsection{The system under study}
The Navier-Stokes system  for a viscous compressible isentropic fluid in $(0,L)$ is
\begin{align*}
\begin{dcases} 
\rho_t  +(\rho u)_x=0   \ &\text{in } (0,+\infty)\times (0,L)  ,  \\
\rho \big(u_t +u u_x \big) + (p(\rho))_x-\nu u_{xx} =0 \ & \text{in } (0,+\infty)\times (0,L) ,
\end{dcases} 
\end{align*}
where $L>0$ denotes the finite length of the interval, $\rho$ is the fluid density and $u$ is the velocity. The viscosity of the fluid is denoted by  $\nu>0$   and we assume the pressure $p$ satisfies the constitutive law $p(\rho)=a\rho^{\gamma}$ for $a>0$ and $\gamma\geq1$. The Navier-Stokes equations linearized around some constant steady state $(Q_0,V_0)$ (with $Q_0>0,V_0>0$) are
\begin{align}\label{lcnse1}
\begin{dcases}
\rho_t  + V_0\rho_x + Q_0u_x =0   \ &\text{in } (0,+\infty)\times (0,L) , \\
u_t - \frac{\nu}{Q_0}u_{xx} + V_0 u_x +a\gamma Q_0^{\gamma-2}\rho_x =0 \ &\text{in } (0,+\infty)\times (0,L) .
\end{dcases}
\end{align}
Now, if we consider the change of variables 
\begin{align*}
\rho(t,x) \to  \alpha \rho(\beta t, \delta x) , \quad u(t,x) \to   u(\beta t, \delta x), \ \ \forall (t,x) \in (0,+\infty) \times (0,L), 
\end{align*}
with the following choices of $\alpha, \beta, \delta >0$, 
\begin{align*}
\alpha: = \left( a\gamma Q^{\gamma-3}_0  \right)^{-1/2}, \ \ \ \beta:  = \frac{Q_0 V^2_0}{\nu}, \ \ \ \delta := \frac{Q_0 V_0}{\nu}, 
\end{align*}
the system \eqref{lcnse1} simplifies to 
\begin{align}\label{lcnse2}
\begin{dcases} 
\rho_t + \rho_x + bu_x = 0  &\text{in } (0,+\infty) \times  (0,\delta L), \\
u_t - u_{xx} + u_x + b\rho_x = 0  \  &\text{in } (0,+\infty) \times  (0,\delta L),
\end{dcases}
\end{align}
with $b = \frac{Q_0}{V_0}\left(a\gamma Q_0^{\gamma-3} \right)^{1/2}$. 

\medskip 

In this article, we study boundary null-controllability properties of the above system \eqref{lcnse2} using the spectral methods. This is the reason we consider  $b=1$ which makes our spectral computations a bit easier and more clear.
 
Let $T>0$  be a finite time (larger than $1$)  and we will work in the domain $(0,T)\times (0,1)$. We work with the space domain $(0,1)$, again for performing  slightly comfortable spectral analysis for the associated adjoint operator.  The same can be done in the interval $(0,\delta L)$. 
   
   \smallskip 
   
\paragraph{\bf Control on density}  We now  write the first  problem below.
\begin{align}\label{main-control-prbm}
\begin{dcases}
\rho_t+\rho_x +u_{x} =0   &\text{in } (0,T)\times (0,1),\\
u_{t} -u_{xx} +u_x +\rho_x =0 &\text{in } (0,T)\times (0,1),\\
\rho(t,0)= h(t)    &\text{for } t\in (0,T) , \\
u(t,0)=0, \  u(t,1)=0  &\text{for } t\in (0,T) ,\\
\rho(0,x)=\rho_0(x),\  u(0,x)=u_0(x)  &\text{for } x\in(0,1),
\end{dcases}
\end{align}
where  $h$ is a control acting on the boundary point $x=0$ only on the density part through a Dirichlet condition, $(\rho_0,u_0)$ is given initial data from some suitable Hilbert space. 
    
The main goal is to study the null-controllability property of the system \eqref{main-control-prbm} at given time $T>0$ which is sufficiently large. 

\begin{definition}  Let us prescribe the notions of null- and approximate controllability for the system \eqref{main-control-prbm}.
	\begin{itemize}
 \item The system \eqref{main-control-prbm}  is  said to be boundary null-controllable  at a finite time $T > 0$, if for any given initial state $(\rho_0, u_0)$ in some suitable Hilbert space, there exists a control $h$  such that the solution $(\rho,u)$ to \eqref{main-control-prbm} can be driven to $0$ at the time $T$, that is, 
$$(\rho(T, x),u(T,x))=(0,0), \quad \text{ for all } x\in (0,1).$$

\item  The system \eqref{main-control-prbm}  is  said to be boundary approximately controllable  at a finite time $T > 0$, if for any given initial state $(\rho_0, u_0)$ and final state $(\rho_T, u_T)$ in some suitable Hilbert space,  and $\epsilon>0$, there exists a control $h \in L^2(0,T)$  such that the solution $(\rho,u)$ to \eqref{main-control-prbm} satisfies 
$$ \|(\rho(T), u(T)) - (\rho_T, u_T)\| \leq \epsilon.$$
\end{itemize}
\end{definition}

To study the main control problem \eqref{main-control-prbm}, we shall  first consider the following auxiliary control system 
\begin{align}\label{lcnse3}
\begin{dcases}
\rho_t+\rho_x +u_{x} =0   &\text{in } (0,T)\times (0,1),\\
u_{t} -u_{xx} +u_x +\rho_x =0 &\text{in } (0,T)\times (0,1),\\
\rho(t,0)=\rho(t,1)+ p(t)      &\text{for } t\in (0,T) , \\
u(t,0)=0, \ u(t,1)=0  &\text{for } t\in (0,T) ,\\
\rho(0,x)=\rho_0(x),\  u(0,x)=u_0(x)  &\text{for } x\in(0,1).
\end{dcases}
\end{align}
where $p$ is a control exerted through the condition $\rho(t,0)=\rho(t,1)+p(t)$. 

We shall prove that there is a control $p\in L^2(0,T)$ which drives the solution $(\rho,u)$ to our new system \eqref{lcnse3} to $(0,0)$. Then, by showing $\rho(t,1)\in L^2(0,T)$ we shall consider $h(t)=\rho(t,1)+p(t)$ and this $h\in L^2(0,T)$ will act as a boundary control for the system \eqref{main-control-prbm}. This idea has already been used in several  works, see for instance  \cite{Cerpa-Crepeau}, \cite{Coron16}. Therefore, our goal is to prove  the null-controllability of the new system \eqref{lcnse3}. The reason behind introducing such an auxiliary system is to obtain  good spectral properties of the associated adjoint operator. The spectral behavior of the operator with purely Dirichlet conditions are really obscure and thus we introduce different boundary conditions as in \eqref{lcnse3} which is helpful to understand the eigenvalues-eigenfunctions of the operator  in more concrete way.

 \smallskip 
 
 \paragraph{\bf Control on velocity} We also consider the following system 
 \begin{align}\label{lcnse4}
 	\begin{cases}
 		\rho_t+\rho_x +u_{x} =0   &\text{in } (0,T)\times (0,1),\\
 		u_{t} -u_{xx} +u_x +\rho_x =0 &\text{in } (0,T)\times (0,1),\\
 		\rho(t,0)=\rho(t,1)  &\text{for } t\in (0,T) , \\
 		u(t,0)=0, \ \ u(t,1)=q(t)  &\text{for } t\in (0,T) ,\\
 		\rho(0,x)=\rho_0(x),\ \ u(0,x)=u_0(x)  &\text{for } x\in(0,1).
 	\end{cases}
 \end{align}
 with a control $q$ acting on one boundary point through the velocity component $u$. Here  we assume  that  the density of the fluid is equal on the boundary points $x=0$ and $x=1$ w.r.t. time.   In this case, the notion of null-controllability can be written as:
 \begin{definition}
 	The system \eqref{lcnse4}  is  said to be boundary null-controllable at a finite time $T > 0$, if for any given initial state $(\rho_0, u_0)$ from some suitable Hilbert space, there exists a control $q$ such that the solution $(\rho,u)$ to \eqref{lcnse4} can be driven to $0$ at the time $T$, that is, $(\rho(T, x),u(T,x))=(0,0)$ for all  $x\in (0,1)$. 
 \end{definition}

Before going to present the main controllability  results associated to the problem \eqref{main-control-prbm} or \eqref{lcnse4}, let us describe the functional framework for our systems.   

\subsection{Functional spaces}
Let $I\subset\mathbb{R}$ be any non-empty open interval and  $L^2(I)$ be the space of all measurable functions which are square integrable on $I$.
For any $s>0$, we also introduce the following Sobolev space
\begin{equation*}
H^s_{{per}}(I) = \left\{\phi \in L^2(I): \phi = \sum_{m \in \mathbb{Z}} c_me^{\frac{2\pi imx}{\mod{I}}} \text{ and }   \sum_{m\in\mathbb{Z}}\left(1+\frac{4\pi^2m^2}{\mod{I}^2}\right)^s\mod{c_m}^2< +\infty \right\},
\end{equation*}
where $\mod{I}$ denotes the length of the interval $I$, with the norm defined on it by 
\begin{equation}\label{norm-H_s}
\norm{\phi}_{H^s_{{per}}(I)} = \left(\sum_{m\in\mathbb{Z}}\left(1+\frac{4\pi^2m^{2}}{\mod{I}^2}\right)^s\mod{c_m}^2\right)^\frac{1}{2}.
\end{equation}
Each  coefficient $c_m$ is given by
\begin{equation*}
c_m= \int_{I} e^{\frac{2\pi i m x}{|I|}}\phi(x)dx,\ \ \forall m \in\mathbb{Z}.
\end{equation*}
Next,  by  ${H}^{-s}_{{per}}(I)$, we  denote  the dual of ${H}^s_{{per}}(I)$ (with respect to the pivot space $L^2(I)$) and the norm of any element $\psi\in {H}^{-s}_{{per}}(I)$ is   defined by 
\begin{equation}\label{norm-H_(-s)}
\norm{\psi}_{{H}^{-s}_{{per}}(I)} = \left(\sum_{m\in\mathbb Z}\left(1+\frac{4\pi^2m^2}{\mod{I}^2}\right)^{-s}\mod{c_m}^2\right)^{\frac{1}{2}},
\end{equation}
where 
\begin{align*}
c_m= \left \langle \psi(\cdot), e^{\frac{2\pi i m \cdot}{|I|}} \right\rangle_{H^{-s}_{per}(I), H^s_{per}(I)} , \ \  \ \forall m \in\mathbb{Z}.
\end{align*}

\smallskip 

Let us write the underlying operator associated with the control system \eqref{lcnse3}, given by
\begin{equation}\label{op_A}
A=\begin{pmatrix}-\partial_x & -\partial_x\\[4pt]-\partial_x & \partial_{xx}-\partial_x\end{pmatrix},
\end{equation}
with its domain 
\begin{align}\label{domain_A}
D(A)= \Big\{\Phi=(\xi, \eta) \in {H}^1(0,1)\times H^2(0,1) : \xi(0)=\xi(1)  , \ \eta(0)=\eta(1)=0  \Big\}.
\end{align}
The adjoint of the operator $A$ has the following formal expression
\begin{equation}\label{op_A*}
A^*=\begin{pmatrix}\partial_x&\partial_x\\[4pt]\partial_x&\partial_{xx}+\partial_x\end{pmatrix},
\end{equation}
yet with the same  domain $D(A^*)=D(A)$, given by \eqref{domain_A}.

\subsection{Notations} Throughout the paper, $C>0$ denotes the generic constant that may vary from line to line and may depend on $T$. For the Hilbert space $H^s(0,1)$ with $s\geq 0$, $\dot H^s(0,1)$ contains the elements with mean zero, i.e., $\phi \in \dot H^s(0,1)$ means $\phi \in H^s(0,1)$ with $\int_0^1 \phi =0$. 

\subsection{Main results}

This section is devoted to  announce the main results of our present work.

\begin{theorem}[Dirichlet control on density]\label{Thm-control-dens-Diri}
Let be any $s> \frac{1}{2}$. Then,  for any time $T>1$ and any given  initial state $(\rho_0, u_0)\in {{H}}_{\text{per}}^{s}(0,1)\times L^2(0,1)$, there exists a boundary control $h \in L^2(0,T)$ acting through the density component such that the system \eqref{main-control-prbm} is null-controllable at time $T$.  
\end{theorem}

As a corollary of the null-controllability result, we shall prove an approximate controllability of our system when a Dirichlet control is acting on the density part. More precisely, we have the following.
\begin{corollary}\label{Coro-approx}
	The system \eqref{main-control-prbm} is approximately controllable at time $T>1$ in the space $L^2(0,1)\times L^2(0,1)$.
\end{corollary}

As indicated earlier, to prove \Cref{Thm-control-dens-Diri}, we first prove the null-controllability of the system \eqref{lcnse3} with a modified boundary condition.  The associated result is given below.
\begin{theorem}[Controllability of the auxiliary system]\label{Thm-control-dens}
Let be any $s> \frac{1}{2}$. Then,  for any time $T>1$ and any given  initial state $(\rho_0, u_0)\in {{H}}_{\text{per}}^{s}(0,1)\times L^2(0,1)$, there exists a boundary control $p \in L^2(0,T)$ acting on the density component such that the system \eqref{lcnse3} is null-controllable at time $T$.  
\end{theorem}

To prove the above theorem, we shall use the so-called moments technique and the heart of which is the detailed spectral analysis of the adjoint operator associated with the system \eqref{lcnse3}.  We split this spectral study into two sections for lesser difficulty to read.   Sections \ref{Sec-Spectral-short} and \ref{Spectral-Details} are devoted for the  spectral analysis of the concerned operator.  The significant point  is that there are three sets of eigenvalues: a parabolic part that contains the eigenvalues $\lambda$ such that asymptotically  $\Re(\lambda)$ behaves like $-k^2\pi^2$ ($k\in \mathbb N$) while $\Im(\lambda)$ is bounded; a hyperbolic part that contains eigenvalues  $\lambda$ such that asymptotically $\Im(\lambda)$ behaves like $2k\pi$ ($k\in \mb Z$) while $\Re(\lambda)$ is a convergent sequence;  and a finite set of lower frequencies.

With this spectrum  of $A^*$,  we solve  a parabolic-hyperbolic mixed moments problem adapting a result of \cite{Hansen94} and this will prove the null-controllability of the system \eqref{lcnse3} when a control is acting on density.

Indeed, thanks to the mentioned properties of the spectrum of the adjoint operator, we prove a parabolic-hyperbolic combined Ingham-type inequality in our paper, which has its  independent interest. As a consequence, we shall prove a null-controllability result (partially)  when a control function is being applied on the velocity part. 

We have the following result.   

\begin{proposition}[A combined Ingham-type inequality]\label{prop_ingham}
	Let $\{\lambda_k\}_{k\in \mathbb N^*}$ and $\{\gamma_k\}_{k\in \mathbb Z}$ be two sequences in $\mathbb C$ with the following properties: there is $N\in \mathbb N^*$, such that 
	\begin{itemize}
		\item[(i)] for all $k, j\in \mathbb Z$,  $\gamma_k\neq\gamma_j$ unless $j=k$ ;
		\item[(ii)] $\gamma_k = \beta + 2k\pi i + \nu_k$  for all   $|k|\geq N$; 
	\end{itemize}
	where $\beta \in \mathbb C$ and $\{\nu_k\}_{|k|\geq N}\in\ell^2$.  
	
	Also, there exists  constants  $A_0\geq 0$, $B_0\geq \delta$ with $\delta>0$ and some $\epsilon>0$ for which $\{\lambda_k\}_{k\in \mathbb N^*}$ satisfies
	\begin{itemize}
		\item[(i)] for all $k, j\in \mb N^*$, $\lambda_k\neq\lambda_j$ unless $j=k$;
		\item[(ii)] $\frac{-\Re(\lambda_k)}{|\Im(\lambda_k)|}\geq \widehat c$ for some $\widehat c>0$ and $k\geq N$;	 
		\item[(iii)] there exists some $r>1$ such that $\mod{\lambda_k- \lambda_j}\geq\delta\mod{k^r-j^r}$ for all $k\neq j$ with $k,j\geq N$ and
		\item[(iv)] $\epsilon(A_0+B_0 k^r) \leq \mod{\lambda_k}\leq A_0 + B_0 k^r$ for all  $k\geq N$.
	\end{itemize}
	We also assume that	the families are disjoint, i.e.,  $$ \left\{\gamma_k, k\in \mathbb{Z}\right\} \cap \left\{\lambda_k, k\in \mathbb{N}^*\right\}=\emptyset. $$ 
	
	Then,	for any time $T>1$, there exists a positive constant $C$ depending only on $T$ such that
	\begin{align}\label{Ingham-ineq}
		\int_0^T \left|\sum_{k\in \mathbb N^*}  a_k e^{\lambda_k t} + \sum_{k\in \mathbb Z} b_k e^{\gamma_k t} \right|^2 dt \geq C\left(\sum_{k\in \mathbb N^*}| a_k|^2 e^{2\Re(\lambda_k) T} + \sum_{k\in \mathbb Z}| b_k|^2 \right),
	\end{align}
	for all sequences $\{a_k\}_{k\in \mathbb N^*}$ and $\{b_k\}_{k\in \mathbb Z}$ in $\ell^2$.
\end{proposition}

\begin{remark}
	The Ingham-type inequality \eqref{Ingham-ineq} is similar to the one in  \cite[Lemma 4.1]{Zhang-Zuazua-1} or \cite[Lemma 4.1]{Zhang-Zuazua-2} and \cite[Theorem 1.1]{Komornik-Ingham} but for different sequences $\lambda_k, \gamma_k$. In our case, we assume some general assumptions on the sequences and  as far as we know,  no direct proof of the inequality \eqref{Ingham-ineq} is available  in the literature.
\end{remark}

Although the above Ingham-type inequality \eqref{Ingham-ineq} is not really helpful to determine the controllability of the system when a control exerts on the density part but it gives some partial result when a control input is acting on the velocity part. 

The main result concerning the control on velocity is the following.
\begin{theorem}[Dirichlet Control on velocity]\label{thm_control_veloct}
	Let be any $s> \frac{1}{2}$. Then, there exists a subspace $\mathcal H \subset H^s_{per}(0,1)\times L^2(0,1)$ of finite codimension such that for any time $T>1$ and any given  initial state $(\rho_0, u_0)\in \mathcal H$, there exists a boundary null-control $q \in L^2(0,T)$ acting on the velocity component  such that the system \eqref{lcnse4} is null-controllable at time $T$.  
\end{theorem} 
The space $\mathcal H$ is defined later in \eqref{Space-V}.  
To prove the above result, we use the Ingham-type inequality \eqref{Ingham-ineq}. In this case, the moments technique does not give a sharper result in the sense that we need larger values of $s$ to prove the subspace controllability.  

The reason behind such restriction to the subspace $\mathcal H$ is the following. In this case, we can only able to show that for higher frequencies of eigenvalues, the eigenfunctions are observable and the observation terms have the required lower bounds. But for lower frequencies, it is hard to conclude anything. Thus, we have to omit the span of those finite number of eigenmodes from the main Hilbert space. However, this difficulty does not occur when a control is acting on the density part and thus, in that case we obtain the required  controllability result in the full space.

\subsection{Literature on the controllability results related to the  compressible Navier-Stokes equations}\label{Section-Literature}

In the past few years, the controllability  of the  compressible and incompressible fluids have  become a very significant topic to the control community.    Fern\'andez-Cara et al. \cite{Fernandez04} proved the local exact distributed controllability  of the Navier-Stokes system when a control is supported in a small open set; see also the references therein.  A local null-controllability result of $3$D Navier-Stokes system with distributed control having two vanishing components has been addressed in \cite{Coron14} by J.-M. Coron and P. Lissy. Badra et al. \cite{Badra16} proved the local exact controllability to the trajectories  for non-homogeneous (variable density) incompressible $2$D Navier-Stokes equations using boundary controls for both density and velocity.

In the case of compressible Navier-Stokes equations, we first mention the work  by  E. V. Amosova \cite{Amosova11} where she considered a compressible viscous fluid in $1$D w.r.t. the  Lagrangian coordinates with zero boundary condition on the velocity and an interior control acting on the velocity equation. She proved a local exact controllability result when the initial density is already on the targeted trajectory. 

Ervedoza et al.  \cite{Ervedoza12} proved a local exact controllability result for the $1$D compressible Navier-Stokes system  in a bounded domain $(0,L)$  for  regular initial data in $H^3(0,L)\times H^3(0,L)$ with two boundary controls, when time is large enough. This result has been improved in \cite{Ervedoza18} by choosing the initial data from $H^1(0,L)\times H^1(0,L)$; see also a generalized result \cite{Ervedo16} for   dimensions $2$ and $3$. 

We also refer that Chowdhury, Ramaswamy and Raymond \cite{Chowdhury12} established a null-controllability and stabilizability result  of a linearized (around a constant steady-state $(Q_0,0)$, $Q_0>0$) $1$D compressible Navier-Stokes equations. The authors proved that their system is  null-controllable for regular initial data by a distributed control acting everywhere in the velocity equation. Their result is proved to be sharp in the following sense:  the null-controllability cannot be achieved by a localized interior control (or by a boundary control) acting on the velocity part. 

Martin, Rosier and Rouchon in \cite{Martin13} considered the  wave equation with structural damping in $1$D; using the spectral analysis and method of moments, they obtained that their equation is null-controllable with a moving distributed control for regular initial conditions in $H^{s+2}\times H^s$ for $s>15/2$ at sufficiently large time. See also \cite{Chaves14} for the higher dimensional case. 

The $1$D compressible Navier–Stokes equations linearized around a constant steady state  with periodic boundary conditions is closely related to the structurally damped wave equation studied in \cite{Martin13}. Chowdhury and Mitra \cite{Chowdhury15b} studied the interior null-controllability of the linearized (around constant steady state $(Q_0, V_0)$, $Q_0 > 0, V_0 > 0$) $1$D compressible Navier–Stokes system   with periodic boundary conditions. Following the approach of  \cite{Martin13}, the authors  \cite{Chowdhury15b} established that their system is null-controllable by a localized interior control when the time is large enough, and for regular initial data in $\dot{H}^{s+1}_{per}\times H^{s}_{per}$ with $s > 13/2$.   They also achieved that the system is approximately controllable in $\dot{L}^2\times L^2$ using a localized interior control and  is null-controllable using periodic boundary control with regular initial data $\dot H^{s+1}_{per}\times H^s_{per}$ for $s> 9/2$.

In \cite{Chowdhury14}, Chowdhury, Mitra, Ramaswamy and Renardy   considered the one-dimensional compressible Navier–Stokes equations linearized around a constant steady state $(Q_0 , V_0 ), Q_0 > 0, V_0 > 0,$ with homogeneous periodic boundary conditions in the interval $(0, 2\pi)$.  They proved that the linearized system with homogeneous periodic boundary conditions is null controllable in  $\dot{H}_{per}^1 \times L^2$ by a localized interior control when the time $T>\frac{2\pi}{V_0}$.  Moreover, in their work the distributed null-controllability result in $\dot{H}_{per}^1 \times L^2$ is sharp in the sense that the controllability fails in $\dot{H}_{per}^s \times L^2$ for any  $0\leq s<1$. As usual, the large time for controllability is needed due to the presence of transport part and indeed, the  null-controllability fails for small time, see \cite{Debayan15}.  

Chowdhury \cite{Chowdhury15} considered same linearized Navier–Stokes system around $(Q_0, V_0)$ with $Q_0 >0, V_0 > 0$ in $(0, L)$ with homogeneous Dirichlet boundary conditions and an interior control acting only on the velocity equation on a open subset $(0,l)\subset (0,L)$.  He proved  the  approximate controllability of the linearized system in $L^2(0,L)\times L^2(0,L)$ with a localized control in $L^2(0,T;L^2(0,l))$ when $T>\frac{L-l}{V_0}$. 

In the context of the controllability of  coupled transport-parabolic system (which is the main feature of linearized compressible Navier-Stokes equations), we must mention the work \cite{Lebeau98} by Lebeau and Zuazua where the distributed null-controllability of Thermoelasticity system has been  studied. More recently, Beauchard, Koenig and  Le Balc'h \cite{Beauchard20} considered the linear parabolic-transport system with constant coefficients and coupling of order zero and one with locally distributed controls  posed on the one-dimensional torus $\mathbb T$. Following the approach of Lebeau and Zuazua \cite{Lebeau98}, they proved the null-controllability at sufficiently large time when there are as many controls as equations. On the other hand, when the control acts only on the transport (resp. parabolic) component, they obtained  an algebraic necessary and sufficient condition on the coupling term for the null-controllability, and  their controllability studies  based on a detailed spectral analysis. According to the more general result established in \cite{Beauchard20}, we can say that for a $2\times 2$ coupled parabolic-transport system (with periodic boundary conditions), the  null-controllability with one localized interior control  holds true  in $L^2(\mathbb T)\times \dot L^2(\mathbb T)$  (resp. in $\dot H^2(\mb T)\times H^2(\mb T)$) when the control acts only on the transport (resp. parabolic) component. 

For the one-dimensional compressible Navier-Stokes system of non-barotropic fluids, linearized around a constant steady state with null velocity, the lack of null controllability by localized interior or boundary control is proved in \cite{Debayan15}. The same occurs in a short time for the steady states with non-zero velocity. One may find few stabilization results for linearized compressible Navier-Stokes system available in \cite{Arfaoui11}, \cite{Chowdhury12, Chowdhury15a, Chowdhury21}, \cite{Mitra15, Mitra17}.

\subsection{Our approach and achievement of the present work}

As mentioned earlier, in compressible Navier-Stokes system, the interesting feature is the coupling between first order transport equation and second order momentum equation of parabolic type. Thermoelastic systems, viscoelastic fluid models are also examples of coupled system involving both transport and parabolic effects. Note that the transport equation is of first order (lack of regularity occurs) and it has Dirichlet condition only a part of the boundary in the system \eqref{main-control-prbm}. This makes our system  more complicated to handle compare to the case of coupled  heat-wave equations. Because of that we need to consider the auxiliary control problem \eqref{lcnse3} which allows us to obtain the spectral analysis of the associated adjoint operator in a concrete  way, though it is intricate to study. 

It was shown in \cite{Chowdhury14, Chowdhury15b} that the linearized compressible Navier-Stokes system with Periodic-Periodic boundary conditions, there is a sequence of generalize eigenfunctions of the associated adjoint operator that forms a Riesz Basis for the state Hilbert space. The success in obtaining this result lies in the simplicity of the corresponding characteristic equations as well as the explicit structure of all eigenfunctions in terms of Fourier basis. 

However, for our operator $(A^*, D(A^*))$ (defined in \eqref{op_A*}), the characteristic equation is a third order ODE and the eigenvalue equation is a non-standard transcendental equation which are difficult to analyze. In fact,  the method  (invariant subspace idea) used in \cite{Chowdhury14, Chowdhury15b} is not practically applicable to this case. But, we manage to compute the eigenfunctions and eigenvalues at least asymptotically.  Then,  by using  an abstract result of {\em perturbation of Riesz basis} due to B.-Z. Guo \cite{Bao1}, we are able to develop the Riesz Basis property of our set of eigenfunctions.

To study the boundary controllability, we mention that the usual extension method  is not really convenient for the Navier-Stokes system. This is because when we put one interior control in the system upon extending the domain, then restricting the solution on the boundary will give rise to two boundary controls for the system.  In this context, we refer some  earlier null-controllability results \cite{Mitra17}, \cite{Ervedoza12}, \cite{Ervedoza18} with  one interior control or two boundary controls both for density and velocity.

The main interest in the present work is that we directly handle the boundary controllability with only one control acting on the density or velocity part. Moreover, using the moments method, we achieve the null-controllability for our system \eqref{lcnse3} (consequently, \eqref{main-control-prbm})  starting with the initial space $H^s_{per}(0,1)\times L^2(0,1)$ for $s>1/2$ which is a good gain of this work compare to the existing works of similar setting as mentioned in Section \ref{Section-Literature}. Although, when a control acts on the velocity part, our limitation is that we can achieve the null-controllability on a strict subspace (of finite codimension) $\mathcal H \subset H^s_{per}(0,1)\times L^2(0,1)$ for $s>1/2$. The   joint parabolic-hyperbolic Ingham-type inequality in \Cref{prop_ingham} will help us to deduce this result. 
In any case, to the best of our knowledge, the null-controllability with a boundary control in density/ velocity through Dirichlet condition has been handled for the first time in this work. 

\subsection{Paper organization}

The paper is organized as follows. 
\begin{enumerate}
\item[--] In \Cref{Section-well-posedness}, we discuss the well-posedness results of the main systems and some associated  results have been proved in  the Appendix.

\item[--] We split the spectral analysis for the associated adjoint operator into two sections for the ease of reading.  \Cref{Sec-Spectral-short} contains a short description of the spectral properties whereas  the detailed  analysis with the main proofs are prescribed in \Cref{Spectral-Details}.	

\item[--] In \Cref{Approx-con}, we obtain the lower bounds for the observation terms which are crucial to determine the null-controllability for the system \eqref{lcnse3} or \eqref{lcnse4}. 
 
\item[--] Then, \Cref{Section-Moments} is devoted to prove the null-controllability of the system \eqref{lcnse3}, that is \Cref{Thm-control-dens} (and consequently of \eqref{main-control-prbm}, i.e., \Cref{Thm-control-dens-Diri}) using the method of moments. To be precise, we solve a set of  mixed parabolic-hyperbolic moments problem to obtain the required result. Then, we prove the approximate controllability result, namely \Cref{Coro-approx} in Section \ref{sec-aprox}.

\item[--] In Section \ref{sec_ingham}, we prove the joint Ingham-type inequality given by \Cref{prop_ingham} and as a consequence, we have shown the partial null-controllability result for the system \eqref{lcnse4} when a control acts on the velocity, namely \Cref{thm_control_veloct}.

\item[--] Finally, the appendix contains several useful results related to our work. In \Cref{Appendix-well-posed}, we give the proof of the well-posedness results for our control problem. In \Cref{Appendix-A3}, we prove that the resolvent to the adjoint operator $A^*$ (or, the operator $A$) associated to our control system is compact. To this end, in \Cref{Appendix-Han}, we recall some important results from the work \cite{Hansen94} regarding the solvability of mixed parabolic-hyperbolic moments problem.  \Cref{Appendix-D} contains some hidden regularity result for the system \eqref{lcnse3} (namely, showing that $\rho(t,1)\in L^2(0,T)$) which allows us to show the controllability of the main system \eqref{main-control-prbm} once we have the controllability for the system \eqref{lcnse3}. 
\end{enumerate}

\section{Well-posedness of the system}\label{Section-well-posedness}
\medskip 
Let us first recall the operator $A^*$ defined by \eqref{op_A*}. Then, we write the adjoint system associated to the control problem \eqref{lcnse3}: let $(\sigma, v)$ be the adjoint state and the system reads as
\begin{equation}\label{lcnse_adjoint}
\begin{cases}
-\sigma_t - \sigma_x - v_{x}=f   &\text{in } (0,T)\times (0,1),\\
-v_{t} - v_{xx} - v_x - \sigma_x = g &\text{in } (0,T)\times (0,1),\\
\sigma(t,0)=\sigma(t,1)     &\text{for } t\in(0,T), \\
v(t,0)=v(t,1)=0            &\text{for } t\in(0,T), \\
\sigma(T,x)= \sigma_T(x),\ \ v(T,x)= v_T(x)  &\text{for } x\in(0,1).
\end{cases}
\end{equation}
Shortly, one may express it by
\begin{align}\label{adj_lcnse}
-V^\prime(t) =A^* V(t) + F(t) , \ \ \forall t\in (0,T), \quad V(T)=V_T,
\end{align}
where the state is $V:=(\sigma, v)$,  given final data is  $V_T:=(\sigma_T,v_T)$ and source term is $F:=(f,g)$. 

To show the well-posedness of the solution to \eqref{lcnse3}, let us first write the following lemma.
\begin{lemma}\label{lemma_wellposedness}
The operator $A$ (resp. $A^*$) is maximal dissipative in $L^2(0,1)\times L^2(0,1)$, that is, $(A,{D}(A))$ (resp. $(A^*, D(A^*))$) generates a strongly continuous semigroup of contractions in $L^2(0,1)\times L^2(0,1)$. 
\end{lemma}
The proof of Lemma \ref{lemma_wellposedness} can be done in a standard fashion. For the sake of completeness, we give the proof  in Appendix  \ref{Appendix-A1}.  

To prove  the existence and uniqueness of solution to the control system \eqref{lcnse3}, it is enough to consider $V_T=(0,0)$ in \eqref{adj_lcnse}. With this, we first write the following result concerning the existence of unique solution to the adjoint system \eqref{adj_lcnse}.
\begin{proposition}\label{Prop-Adjoint}
For any given  $F:=(f,g)\in L^2(0,T; L^2(0,1))\times L^2(0,T;L^2(0,1))$, there exists a unique weak solution $V:=(\sigma, v)$ to the system  \eqref{adj_lcnse}  in the space $L^2(0,T;L^2(0,1))\times L^2(0,T;H^1_0(0,1))$. 

Moreover, the map $F\mapsto V$ is linear and continuous from $L^2(0,T; L^2(0,1))\times L^2(0,T;L^2(0,1))$ to \\  $\mathcal{C}([0,T];L^2(0,1))\times [\mathcal{C}([0,T];L^2(0,1))\cap L^2(0,T; H^1_0(0,1))]$.
\end{proposition} 
Once we have the existence of semigroup defined by $(A^*, D(A^*))$ (as per Lemma \ref{lemma_wellposedness}), then the proof of the above proposition can be adapted from the work \cite[Chap IV, Sec. 4.3]{Girinon08}. We omit the details here.

Now, we can define the notion of solutions to the control systems \eqref{lcnse3} and \eqref{lcnse4} in the sense of transposition (see for instance \cite{Coron07})  where a non-trivial boundary source term is appearing.   
\begin{definition}\label{Defi-solu-1}
	\begin{itemize} 
		\item 	For given initial state $U_0:=(\rho_0, u_0)\in L^2(0,1)\times L^2(0,1)$ and boundary data $p\in L^2(0,T)$, a function ${U}:= (\rho, u)\in L^2(0,T; L^2(0,1))\times L^2(0,T;L^2(0,1))$ is a solution to the system \eqref{lcnse3} if for  any given ${F}:=(f,g) \in L^2(0,T; L^2(0,1))\times L^2(0,T;L^2(0,1))$, the following identity holds true: 
		\begin{align*} 
			\int_{0}^{T}\int_{0}^{1}\rho(t,x)f(t,x)dxdt+\int_{0}^{T}\int_{0}^{1}u(t,x)g(t,x) dx dt =\ip{U_0(\cdot)}{{V}(0,\cdot)}_{L^2\times L^2} + \int_{0}^{T}\sigma(t,1) p(t) dt,
		\end{align*}
		where $V:=(\sigma,v)$ is the unique weak solution to the adjoint system \eqref{adj_lcnse} with $V_T=(0,0)$.	
		
		\smallskip 
		
		\item 		For given initial state $U_0:=(\rho_0, u_0)$ and boundary data $q\in L^2(0,T)$,	a function ${U}:= (\rho, u)\in L^2(0,T;(H^1(0,1))^{\prime})\times L^2(0,T;L^2(0,1))$ is a solution to the system \eqref{lcnse4} if for  any given ${F}:=(f,g) \in L^2(0,T; H^1(0,1))\times L^2(0,T;L^2(0,1))$, the following identity holds true: 
	\begin{multline*} 
		\int_{0}^{T}\ip{\rho(t,\cdot)}{f(t,\cdot)}_{(H^1)^{\prime},H^1} dt+\int_{0}^{T}\int_{0}^{1}u(t,x)g(t,x) dx dt\\ =\ip{U_0(\cdot)}{{V}(0,\cdot)}_{L^2\times L^2} + \int_{0}^{T}\big[\sigma(t,1)+v_x(t,1)\big] q(t) dt,
	\end{multline*}
	where $V:=(\sigma,v)$ is the unique weak solution to the adjoint system \eqref{adj_lcnse} with $V_T=(0,0)$. 
	\end{itemize} 
\end{definition}

Let us state the following theorems that concern the existence and uniqueness of solution to the control problems \eqref{lcnse3} and \eqref{lcnse4}.  
\begin{theorem}\label{Thm-existnce-control-sol}
For every $p\in L^2(0,T)$ and every ${U}_0:=(\rho_0, u_0)\in L^2(0,1)\times L^2(0,1)$, the system \eqref{lcnse3} has a unique weak solution ${U}:=(\rho, u)$ belonging to the space $L^2(0,T;L^2(0,1))\times L^2(0,T;L^2(0,1))$ in the sense of transposition and the operator defined by
\begin{equation*}
({U}_0,p) \mapsto {U}({U}_0,p),
\end{equation*}
is linear and continuous from $(L^2(0,1)\times L^2(0,1))\times L^2(0,T)$ into $L^2(0,T;L^2(0,1))\times L^2(0,T;L^2(0,1))$.
	
Moreover, the solution satisfies the following regularity result,
\begin{align}\label{regularity} 
(\rho, u) \in \mathcal C^0 ([0,T]; L^2(0,1)) \times  [ \mathcal C^0([0,T];L^2(0,1)) \cap L^2(0,T;H^1_0(0,1)) ]
\end{align}  
with the  estimate
\begin{multline}\label{continuity_estimate}
	\|\rho\|_{\mathcal C^0([0,T]; L^2(0,1))} + \|u\|_{\mathcal C^0([0,T]; L^2(0,1)) \cap L^2(0,T; H^1_0(0,1))} \\
\leq C\left(\norm{\rho_0}_{L^2(0,1)}+\norm{u_0}_{L^2(0,1)}+\norm{p}_{L^2(0,T)}\right),
\end{multline}
for some constant $C>0$. 
\end{theorem}

\begin{theorem}\label{Thm-existnce-control-sol-vel}
	For every $q\in L^2(0,T)$ and every ${U}_0:=(\rho_0, u_0)\in L^2(0,1)\times L^2(0,1)$, the system \eqref{lcnse4} has a unique weak solution ${U}:=(\rho, u)$ belonging to the space $L^2(0,T;(H^1(0,1))^{\prime})\times L^2(0,T;L^2(0,1))$ in the sense of transposition and
	 the operator defined by
	\begin{equation*}
		({U}_0,q) \mapsto {U}({U}_0,q),
	\end{equation*}
	is linear and continuous from $(L^2(0,1)\times L^2(0,1))\times L^2(0,T)$ into $L^2(0,T;(H^1(0,1))^{\prime})\times L^2(0,T;L^2(0,1))$.
\end{theorem}

We give a sketch of the proof for  \Cref{Thm-existnce-control-sol} in  \Cref{Appendix_A2}. The proof for  \Cref{Thm-existnce-control-sol-vel} will be followed from  \cite[Section 3]{Chowdhury13}.

\smallskip 

\section{A short description of the spectral properties  of the adjoint operator}\label{Sec-Spectral-short}

\medskip 

In this section, we briefly describe the  spectral properties of the  adjoint operator $A^*$ associated to our control system \eqref{lcnse3}.  The main results concerning the spectral analysis will be written here without giving all the details. A detailed study will be posed at the end of this paper, namely in Section \ref{Spectral-Details}.

\subsection{The eigenvalue problem}
Let us  denote ${\Phi}:=(\xi,\eta)$ and  consider the following eigenvalue problem:
\begin{equation*}
A^*{\Phi}=\lambda{\Phi}, \quad \text{for } \lambda \in \mb C,
\end{equation*}
which is more explicitly given by
\begin{equation}\label{eigenfunction}
\begin{aligned}
\xi^{\prime}(x)+\eta^{\prime}(x)&=\lambda\xi(x), \quad x\in (0,1)  ,\\
\eta^{\prime\prime}(x)+\eta^{\prime}(x)+\xi^{\prime}(x)&=\lambda\eta(x), \quad x\in (0,1),\\
\xi(0)&=\xi(1),\\
\eta(0)=0,\ \ \eta(1)&=0.
\end{aligned}
\end{equation}

We now state the following proposition. 

\begin{proposition}\label{Prop_spectral}
For the operator $A^*$, we have the following results.
\begin{itemize}
	\item[(i)] We have $\ker A^* = \Span \{ (1,0)  \} $.

\item[(ii)] The resolvent operator associated with $A^*$ is compact and that the spectrum of $A^*$ is discrete.  

\item[(iii)] All non-zero eigenvalues have negative real parts. 
	
\item[(iv)] All eigenvalues $\lambda$  are  geometrically simple.
\end{itemize}
\end{proposition}

A quick observation tells that: when $\lambda=0$, then  $(1,0)$ is an eigenfunction of the operator $A^*$, which is the part (i) of the above proposition.

The proof of points (iii) and (iv)  are given   in  Section \ref{Spectral-Details};   point (ii) is proved in Appendix \ref{Appendix-A3}.

\subsection{The set of eigenvalues}
Let us declare the properties of the eigenvalues of the operator $A^*$. More precisely, we have the following lemma.
\begin{lemma}\label{lemma_eigenvalue_eigenfunc}
Let be the operator $(A^*, D(A^*))$ given by \eqref{op_A*}. Then, there exists some numbers $k_0, n_0 \in \mathbb N^*$ such that  $A^*$ has three sets of eigenvalues: the parabolic part $\{\lambda^p_k\}_{k\geq k_0}$, the hyperbolic part $\{\lambda^h_k\}_{|k|\geq k_0}$ and a finite family $\{0\}\cup\{\widehat \lambda_n\}_{n=1}^{n_0}$ of lower frequencies. Moreover, the parabolic and hyperbolic branches satisfies the following asymptotic  properties:  
\begin{subequations} 
\begin{align}\label{eigenvalues-lambda}
\lambda^p_k &= -k^2\pi^2 -2c_k k\pi -2id_k  k\pi + O(1), \quad \text{for all } k\geq k_0 \text{ large}, \\\label{eigenvalues-gamma}
\lambda^h_k &= -1-\alpha_{1,k} -i(2k\pi +\alpha_{2,k}),  \quad \text{for all } |k|\geq k_0 \text{ large}, 
\end{align}
\end{subequations}
where $\{c_k\}_{k\geq k_0}$, $\{d_k\}_{k\geq k_0}$, $\{\alpha_{1,k}\}_{\mod{k}\geq k_0}$, $\{\alpha_{2,k}\}_{\mod{k}\geq k_0}$ are real bounded sequences with their absolute value less or equal to $\pi/2$ such that 
\begin{align*}
c_k \to 0, \ \ d_k \to 0 , \ \ \text{as } k\to +\infty ,\\ \alpha_{1,k} \to 0, \ \ \alpha_{2,k} \to 0, \ \ \text{as } |k|\to +\infty,
\end{align*}
and, with in addition,  
\begin{align}\label{asymp_sin}
c_k = O(k^{-1}), \ \ d_k = O(k^{-1}), \quad \alpha_{1,k}=O(\mod{k}^{-1}),\ \ \alpha_{2,k}= O(|k|^{-1}), \ \ \text{for large modulus of } k. 
\end{align}
\end{lemma}

The proof of the above lemma is one of the crucial part of our work and it is heavy; the details have been provided  in  Sections \ref{Section-eigenvalues} and \ref{Proof-lemma-3.2}. 

\smallskip 

For simplicity, we set $\lambda_0=0$ and the associated eigenfunction by $\Phi_{\lambda_0}=(1,0)$. We further   denote the set of eigenvalues by $\sigma(A^*)$, where
\begin{align}\label{spectrum}
\sigma(A^*) : =  \Big\{\lambda^p_{k} , \, k \geq k_0; \ \ \lambda^h_k, \,  |k|\geq k_0\Big\} \cup \Big\{\lambda_0\Big\} \cup \Big\{\widehat\lambda_n, \, 1\leq n \leq n_0 \Big\}.
\end{align}

\smallskip 

\subsection{The set of eigenfunctions}
We start by writing the  following proposition. 
\begin{proposition}\label{Prop-eigenfunctions}
Let be $k_0, n_0$ as given by Lemma \ref{lemma_eigenvalue_eigenfunc}.   Then, the operator $A^*$ has three sets of  eigenfunctions: the parabolic part $\{\Phi_{\lambda^p_k}\}_{k\geq k_0}$, the hyperbolic part $\{\Phi_{\lambda^h_k}\}_{|k|\geq k_0}$  and a finite family $\{\Phi_{\lambda_0}\}\cup\{\Phi_{\widehat \lambda_n}\}_{n=1}^{n_0}$ corresponding to the lower frequencies. 

Furthermore,  we have the following:
\begin{enumerate} 
\item[1.]  The parabolic and hyperbolic  parts of the eigenfunctions have asymptotic expressions for large modulus of $k$, given by \eqref{eigen-1-xi}--\eqref{eigen-1-eta} and \eqref{eigen-2-xi}--\eqref{eigen-2-eta} respectively. 
\item[2.] The 	eigenfamily, denoted by 
\begin{align}\label{Eigenfamily}
\E(A^*) : = \Big\{\Phi_{\lambda^p_k}, \,  k\geq k_0; \ \ \Phi_{\lambda^h_k}, \,  |k|\geq k_0 \Big\} \cup \Big\{ \Phi_{\lambda_0}\Big\} \cup \Big\{ \Phi_{\widehat\lambda_n}, \, 1\leq n\leq n_0  \Big\}, 
\end{align}
forms a Riesz basis in $L^2(0,1)\times L^2(0,1)$, and as a consequence, it is  a dense family in $H^{-s}_{\text{per}}(0,1)\times L^2(0,1)$ for any $s>0$.
\end{enumerate}
\end{proposition}

The existence of parabolic and hyperbolic parts of the family of eigenfunctions are proved in Section \ref{subsection-eigenfunc}. Then, using a result from \cite{Bao1}, we shall prove the existence of lower frequencies of  eigenvalues $\{\widehat \lambda_n\}_{n=1}^{n_0}$ and the associated eigenfunctions $\{\Phi_{\widehat \lambda_n}\}_{n=1}^{n_0}$.  Moreover, we conclude that the  set of eigenfunctions $\E(A^*)$ forms a Riesz basis for  $L^2(0,1)\times L^2(0,1)$. 

We hereby introduce the components of the eigenfunctions as given below:
\begin{align}\label{eigen-para}
\Phi_{\lambda^p_k}:=(\xi_{\lambda^p_k}, \eta_{\lambda^p_k}) \ \ \text{ associated with } \lambda^p_k , \ \ \forall k \geq k_0,
\end{align}
\begin{align}\label{eigen-hyper}
\Phi_{\lambda^h_k}: = (\xi_{\lambda^h_k}, \eta_{\lambda^h_k})  \ \ \text{ associated with } \lambda^h_k , \ \ \forall |k|\geq k_0.
\end{align}
The asymptotic expressions of those components are given below and are obtained in Section \ref{subsection-eigenfunc}. For large $k\geq k_0$, the two components of the eigenfunction $\Phi_{\lambda^p_k}$  are   
\begin{multline}\label{eigen-1-xi}
{\xi_{\lambda^p_k}(x)} =  e^{x\left(-k^2\pi^2 -2c_k k\pi - 2 i d_k k\pi  + O(1) \right)} \times O\big(\frac 1k \big)  \\ 
+\left(\frac{i}{k\pi} +O\big(\frac{1}{k^2}\big)\right) \left( e^{i(k\pi +c_k)+O(k^{-1})-\frac{1}{2}-d_k} -  e^{-k^2\pi^2 -2c_kk\pi -2id_k k\pi + O(1)} \right) e^{x \left(-i(k\pi +c_k)-\frac{1}{2}+d_k +O(k^{-1}) \right)} \\ 
+\left(-\frac{i}{k\pi} +O\big(\frac{1}{k^2}\big)\right)  \left( e^{-k^2\pi^2 -2c_k k\pi-2id_k k\pi  + O(1)}- e^{ - i(k\pi +c_k)-\frac{1}{2}+d_k+O(k^{-1})}\right) e^{x\left(i(k\pi +c_k)-\frac{1}{2}-d_k +O(k^{-1}) \right)},
\end{multline}
\begin{align}\label{eigen-1-eta}
{\eta_{\lambda^p_k}(x)}&= e^{x\left(-k^2\pi^2 -2c_k k\pi -2id_k k\pi  + O(1) \right)} \times O\big(\frac{1}{k^3}  \big) \\ \notag
&\quad+ \left( e^{i(k\pi +c_k)+O(k^{-1})-\frac{1}{2}-d_k} -  e^{-k^2\pi^2 -2c_kk\pi -2id_k k\pi + O(1)} \right) e^{x \left(-i(k\pi +c_k)-\frac{1}{2}+d_k +O(k^{-1}) \right)} \\\notag
&\quad+ \left( e^{-k^2\pi^2 -2c_k k\pi-2id_k k\pi  + O(1)}- e^{ - i(k\pi +c_k)-\frac{1}{2}+d_k+O(k^{-1})}\right) e^{x\left(i(k\pi +c_k)-\frac{1}{2}-d_k +O(k^{-1}) \right)},\notag
\end{align}
for all $x\in (0,1)$. The above functions are obtained later in \eqref{eigen-xi_lambda} and \eqref{component_eta_lambda} combined with Remark \ref{remark-asymptic-coeffi}. 

\smallskip 

Next, we prescribe the explicit forms of the two components of   $\Phi_{\lambda^h_k}$ for  $|k|\geq k_0$ large. Here and in the sequel, we  introduce the sign function given by 
\begin{align} \label{sign-func}
\sgn(k) = \begin{dcases}  1 \quad \text{when } k \geq 0 , \\
-1 \quad \text{when } k < 0 ,
\end{dcases}
\end{align}
and we have for all $|k|\geq k_0>0$. We have 
\begin{multline}\label{eigen-2-xi}
{\xi_{\lambda^h_k}(x)} =  \left(e^{\sgn(k) \sqrt{|k\pi|} - \frac{1}{2} -i\sqrt{|k\pi|} + O(|k|^{-\frac{1}{2}})}  - e^{-\sgn(k) \sqrt{|k\pi|} - \frac{1}{2} +i\sqrt{|k\pi|} + O(|k|^{-\frac{1}{2}})} \right) \\
\times \frac{(-\alpha_{1,k}+2ik\pi + O(1))}{k\pi e^{ \sqrt{|k\pi|} + \frac{1}{\sqrt{|k|}} } } \times e^{-x\left(\alpha_{1,k} + i(2k\pi +\alpha_{2,k}) +O(|k|^{-1})\right)} \\
+ \bigg(e^{-\sgn(k)\sqrt{|k\pi|} - \frac{1}{2} +i\sqrt{|k\pi|} + O(|k|^{-\frac{1}{2}})}  - e^{-\alpha_{1,k} - i(2k\pi +\alpha_{2,k}) + O(|k|^{-1})}   \bigg) \\
\times \frac{1}{k\pi e^{ \sqrt{|k\pi|} + \frac{1}{\sqrt{|k|}} } }\left(\sgn(k) \frac{1}{2\sqrt{|k\pi|}} + \frac{i}{2\sqrt{|k\pi|}} + O\Big(\frac{1}{|k|}\Big)\right)	\times e^{x \left(\sgn(k) \sqrt{|k\pi|} - \frac{1}{2} -i\sqrt{|k\pi|} + O(|k|^{-\frac{1}{2}})\right)} \\
+\bigg(e^{-\alpha_{1,k} -i(2k\pi +\alpha_{2,k}) + O(|k|^{-1})}  - e^{\sgn(k) \sqrt{|k\pi|} - \frac{1}{2} -i\sqrt{|k\pi|} + O(|k|^{-\frac{1}{2}})} \bigg) \\
\times \frac{1}{k\pi e^{ \sqrt{|k\pi|} + \frac{1}{\sqrt{|k|}} } } \left( -\sgn(k)\frac{1}{2\sqrt{|k\pi|}} - \frac{i}{2\sqrt{|k\pi|}} + O\Big(\frac{1}{|k|}\Big) \right) \times e^{x\left(-\sgn(k) \sqrt{|k\pi|} - \frac{1}{2} +i\sqrt{|k\pi|} + O(|k|^{-\frac{1}{2}}) \right)},
\end{multline}
\begin{multline}\label{eigen-2-eta}
{\eta_{\lambda^h_k}(x)}=\frac{1}{k\pi e^{ \sqrt{|k\pi|} + \frac{1}{\sqrt{|k|}} } } \left(e^{\sgn(k) \sqrt{|k\pi|} - \frac{1}{2} -i\sqrt{|k\pi|}+O(|k|^{-\frac{1}{2}})}  - e^{-\sgn(k)\sqrt{|k\pi|} - \frac{1}{2} +i\sqrt{|k\pi|} + O(|k|^{-\frac{1}{2}})} \right)\\
\times  e^{-x\left(\alpha_{1,k} + i(2k\pi +\alpha_{2,k}) +O(|k|^{-1})\right)} \\
+\frac{1}{k\pi e^{ \sqrt{|k\pi|} + \frac{1}{\sqrt{|k|}} } }\left(e^{-\sgn(k)\sqrt{|k\pi|} - \frac{1}{2} +i\sqrt{|k\pi|} + O(|k|^{-\frac{1}{2}})} - e^{-\alpha_{1,k} - i(2k\pi +\alpha_{2,k}) + O(|k|^{-1})}   \right)\\
\times e^{x \left(\sgn(k)\sqrt{|k\pi|} - \frac{1}{2} -i\sqrt{|k\pi|} + O(|k|^{-\frac{1}{2}}) \right)}  \\
+\frac{1}{k\pi e^{ \sqrt{|k\pi|} + \frac{1}{\sqrt{|k|}} } } \left(e^{-\alpha_{1,k} - i(2k\pi +\alpha_{2,k}) + O(|k|^{-1})}  - e^{\sgn(k)\sqrt{|k\pi|} - \frac{1}{2} -i\sqrt{|k\pi|} + O(|k|^{-\frac{1}{2}})}  \right)\\
\times  e^{x\left(-\sgn(k)\sqrt{|k\pi|} - \frac{1}{2} +i\sqrt{|k\pi|} + O(|k|^{-\frac{1}{2}})  \right)},
\end{multline}
for all $x\in (0,1)$. Those components are obtained later in \eqref{eigen-xi-gamma} and \eqref{eigen-eta-gamma} respectively.

We now write  the  upper bounds of the norms of our eigenfunctions $\Phi_{\lambda^p_k}$, $k\geq k_0$ and $\Phi_{\lambda^h_k}$, $|k|\geq k_0$. A short proof is given in Section \ref{Spectral-Details}.

\begin{lemma}[{Bounds of the eigenfunctions}]\label{lemma_bound_eigenfunc}
Recall the eigenfunctions $\Phi_{\lambda^p_k}=(\xi_{\lambda^p_k}, \eta_{\lambda^p_k})$, $\forall k\geq k_0$ and  $\Phi_{\lambda^h_k}=(\xi_{\lambda^h_k}, \eta_{\lambda^h_k})$,   $\forall |k| \geq k_0$ given by \eqref{eigen-1-xi}--\eqref{eigen-1-eta} and \eqref{eigen-2-xi}--\eqref{eigen-2-eta} respectively.   There exists a constant $C>0$ independent in $k$, such that we have the following.
\begin{enumerate}
\item[1.] For any $k\geq k_0$, we have  
\begin{align}\label{bound_xi_lambda}
\begin{dcases}
\|\xi_{\lambda^p_k}\|_{L^2(0,1)}\leq Ck^{-1},\\
\|\xi_{\lambda^p_k}\|_{H^{-s}_{per}(0,1)}\leq Ck^{-s-1}, \ \ \text{for } 0<s<1, \\
\|\xi_{\lambda^p_k}\|_{H^{-s}_{per}(0,1)}\leq Ck^{-2}, \ \ \text{for } s\geq 1, \\
\|\eta_{\lambda^p_k}\|_{L^2(0,1)}\leq C.
\end{dcases}
\end{align}

\item[2.] On the other hand, for any  $|k|\geq k_0$, we have 
\begin{align}\label{bound_xi_gamma}
\begin{cases}
\|\xi_{\lambda^h_k}\|_{L^2(0,1)}\leq C, \\
\|\xi_{\lambda^h_k}\|_{ H^{-s}_{per}(0,1)}\leq C|k|^{-s}, \ \ \text{for } 0<s<1 , \\
\|\xi_{\lambda^h_k}\|_{H^{-s}_{per}(0,1)}\leq C|k|^{-1}, \ \ \text{for } s\geq 1, \\
\|\eta_{\lambda^h_k}\|_{L^2(0,1)}\leq C|k|^{-1}.
\end{cases}
\end{align}
\end{enumerate}
\end{lemma}
 
\medskip

\paragraph{\bf Riesz basis properties of the eigenfunctions}\label{para-Riesz}
As mentioned earlier, we have the existence of eigenfunctions for large frequencies of $|k|$ with their asymptotic formulations.   In this paragraph, using some result from \cite{Bao1}, we shall confirm the existence of eigenfunctions for lower frequencies and moreover, we show that the family of eigenfunctions forms a Riesz basis for $L^2(0,1)\times L^2(0,1)$. 

In this regard, let us first recall the following result. 
\begin{theorem}[B.-Z. GUO \cite{Bao1}]\label{Thm_Bao}
Let $\mathcal A$ be a densely defined discrete operator (i.e., the  resolvent of $\mathcal A$ is compact) in a Hilbert space $H$. Let $\{\phi_n \}_{1}^{\infty}$ be a Riesz basis of $H$.	If there are an integer $N \geq 0$ and a sequence of generalized eigenvectors  $\{\psi_n \}_{N+1}^{\infty}$ of $\mathcal A$ such that $$ \sum_{N+1}^{\infty} \|\phi_n-\psi_n\|^2<+\infty ,$$ then the following results hold.
\begin{itemize}
\item[(i)] There are a constant $M>N$  and generalized eigenvectors $\{\psi_{n0} \}_{1}^{M}$ of $\mathcal A$ such that $\{\psi_{n0} \}_{1}^{M}\cup \{\psi_{n} \}_{M+1}^{\infty}$ forms a Riesz Basis for $H$. 
\item[(ii)]  Let $\{\psi_{n0} \}_{1}^{M}\cup \{\psi_{n} \}_{M+1}^{\infty}$  correspond to the eigenvalues $\{\lambda_n\}_{1}^{\infty}$ of $\mathcal A$.  Then the spectrum $\sigma(\mathcal A) = \{\lambda_n\}_{1}^{\infty}$, where $\lambda_n$ is counted according to its algebraic multiplicity.
\item[(iii)] If there is an $M_0>0$ such that $\lambda_n \neq \lambda_m$ for all $m, n > M_0$, then there is an $N_0>M_0$ such that all $\lambda_n$ are algebraically simple if $n>N_0$.
\end{itemize}
\end{theorem}

The first assumption of Theorem \ref{Thm_Bao} holds true in our case since we know that the resolvent operator of $A^*$ is compact,  thanks to the Proposition \ref{Prop_spectral}--part (i). So, the next duty is to find a family	 $\{\Psi_{k}, \ k\in \mathbb N^*; \  \widetilde \Psi_{k}, \ k\in \mathbb Z\}$ that defines  a Riesz  basis for $L^2(0,1)\times  L^2(0,1)$ and  that is quadratically close to $\mathcal E(A^*)$. But it is enough to show that this property holds for large frequencies. Precisely, our goal is to show the following:
\begin{equation*}
\sum_{k\geq k_0}\norm{\Phi_{\lambda^p_k}-\Psi_{k}}_{L^2\times L^2}^2+\sum_{\mod{k}\geq k_0}\norm{\Phi_{\lambda^h_k}-\widetilde \Psi_{k}}_{L^2\times L^2}^2<+\infty.
\end{equation*}

Let us  consider  the following functions:
\begin{subequations}
\begin{align} \label{Functions_Psi_k}
\Psi_{k}(x)& := \begin{pmatrix} \phi_k \\  \psi_k \end{pmatrix} =\begin{pmatrix}0\\[4pt]2ie^{-\frac{1}{2}(1+x)}\sin(k\pi(1-x))\end{pmatrix}, \quad \forall k \in \mb N^*, \\ \label{Functions_Psi_tilde_k}
\widetilde \Psi_{k}(x) & := \begin{pmatrix} \widetilde\phi_k \\ \widetilde  \psi_k \end{pmatrix} = 	\begin{pmatrix}  2i \sgn(k) e^{-\frac{1}{2} -i \sgn(k) \sqrt{|k\pi|}} e^{-2ik\pi x}\\[4pt] 0 \end{pmatrix}, \quad \forall k\in \mb Z.
\end{align} 
\end{subequations}
It can be shown that the family $\{\Psi_{k}, \ k\in \mathbb N^*; \  \widetilde \Psi_{k}, \ k\in \mathbb Z\}$ of above functions forms a Riesz basis for $L^2(0,1)\times L^2(0,1)$. 

Also,  we have the following result.
\begin{lemma}\label{Lemma_quadratic_close}
The family  $\{\Psi_{k}, \, k\in \mathbb N^*; \  \widetilde \Psi_{k}, \, k\in \mathbb Z\}$ given by \eqref{Functions_Psi_k}--\eqref{Functions_Psi_tilde_k} is quadratically close to the family of eigenfunctions $\mathcal E(A^*)$.
\end{lemma}

\begin{proof}  We prove the lemma into two steps. As we mentioned earlier, it is enough to show the quadratically closeness property for large modulus of eigenvalues, that is to 
	$\{\Phi_{\lambda^p_k}, \,  k\geq k_0; \   \Phi_{\lambda^h_k}, \, |k|\geq k_0\}$. 
	
{1. \em $\Psi_{k}$ is quadratically close to $\Phi_{\lambda^p_k}$.} Having the first component $\phi_k=0$ of $\Psi_k$, we see 
\begin{equation}\label{norm_difference_phi_k}
\norm{\xi_{\lambda^p_k}-\phi_{k}}^2_{L^2}=\norm{\xi_{\lambda^p_k}}^2_{L^2}\leq Ck^{-2},\ \  \forall  k\geq k_0 \text{ large}.
\end{equation}

Next, we focus on the second component of the eigenfunction $\Phi_{\lambda^p_k}$ (see \eqref{eigen-1-eta}), we can rewrite it as 
\begin{align*}
\eta_{\lambda^p_k}(x) &= e^{x\left(-k^2\pi^2 -2c_k k\pi -2id_k k\pi  + O(1) \right)} \times O\big(\frac{1}{k^3} \big) \\
&\quad+ e^{-k^2\pi^2 -2c_k k\pi -2id_k k\pi  + O(1)}e^{-\frac{1}{2}(1+x)}\left(e^{ix(k\pi+c_k)}e^{-x(d_k+O(k^{-1}))}-e^{-ix(k\pi+c_k)}e^{x(d_k+O(k^{-1}))}\right)\\
&\quad+e^{-\frac{1}{2}(1+x)}\left(e^{i(1-x)(k\pi+c_k+id_k)+O(k^{-1})}-e^{-i(1-x)(k\pi+c_k+id_k)+O(k^{-1})}\right).
\end{align*}
The last term in above can be expressed as
\begin{align*}
&e^{i(1-x)(k\pi+c_k+id_k)+O(k^{-1})}-e^{-i(1-x)(k\pi+c_k+id_k)+O(k^{-1})}\\
&=2i \sin((1-x)(k\pi+c_k+id_k)) + O(k^{-1})\\
&\sim_{+\infty}  2i \sin ( k\pi(1-x)) +  O(k^{-1}),
\end{align*}
because $c_k,d_k=O(k^{-1})$, thanks to the Lemma \ref{lemma_eigenvalue_eigenfunc}. Thus for large $k\in \mathbb N^*$, one has  
\begin{align*}
\mod{\eta_{\lambda^p_k}(x)-\psi_{k}(x)}&\leq C\left( e^{x(-k^2\pi^2-2c_kk\pi)} + O(k^{-1}) \right),
\end{align*}
and therefore
\begin{equation*}
\norm{\eta_{\lambda^p_k}-\psi_{k}}^2_{L^2}\leq  \frac{C}{k^2}, \quad  \forall k\geq k_0 \text{ large enough}.
\end{equation*}
Hence, for sufficiently large $k_0 \in \mathbb N^*$, we have
\begin{equation}
\sum_{k\geq k_0}\norm{\Phi_{\lambda^p_k}-\Psi_{k}}_{L^2\times L^2}^2\leq C\left(\sum_{k\geq k_0}k^{-2}+\sum_{k\geq k_0} k^{-2} \right)<+ \infty.
\end{equation}	

\smallskip 

{2. \em $\widetilde \Psi_{k}$ is quadratically close to $\Phi_{\lambda^h_k}$.} 

Let us compute following quantity $\xi_{\lambda^h_k} - \widetilde \phi_{k}$, where $\xi_{\lambda^h_k}$ is defined by \eqref{eigen-2-xi}.  We have $\forall x\in (0,1)$ that
\begin{multline}\label{difference_xi_tildephi}
\xi_{\lambda^h_k}(x) - \widetilde \phi_k(x) =  \frac{2i}{e^{\sqrt{|k\pi|} + \frac{1}{\sqrt{|k|}}  } } \left(e^{\sgn(k) \sqrt{|k\pi|} - \frac{1}{2} -i\sqrt{|k\pi|} + O(|k|^{-\frac{1}{2}})}  - e^{-\sgn(k) \sqrt{|k\pi|} - \frac{1}{2} +i\sqrt{|k\pi|} + O(|k|^{-\frac{1}{2}})} \right) \\
\times  e^{-x\left(\alpha_{1,k} + i(2k\pi +\alpha_{2,k}) +O(|k|^{-1})\right)} - 2i\sgn(k) e^{-\frac{1}{2}-i\sgn(k)\sqrt{|k\pi|} } e^{-2ik\pi x}\\
+ \frac{(-\alpha_{1,k} + O(1))}{k\pi e^{ \sqrt{|k\pi|}+\frac{1}{\sqrt{|k|}} }} \left(e^{\sgn(k) \sqrt{|k\pi|} - \frac{1}{2} -i\sqrt{|k\pi|} + O(|k|^{-\frac{1}{2}})}  - e^{-\sgn(k) \sqrt{|k\pi|} - \frac{1}{2} +i\sqrt{|k\pi|} + O(|k|^{-\frac{1}{2}})} \right)\\
\times  e^{-x\left(\alpha_{1,k} + i(2k\pi +\alpha_{2,k}) +O(|k|^{-1})\right)} \\
+ \bigg(e^{-\sgn(k)\sqrt{|k\pi|} - \frac{1}{2} +i\sqrt{|k\pi|} + O(|k|^{-\frac{1}{2}})}  - e^{-\alpha_{1,k} - i(2k\pi +\alpha_{2,k}) + O(|k|^{-1})}   \bigg) \\
\times \frac{1}{k\pi e^{\sqrt{|k\pi|} +\frac{1}{\sqrt{|k|}} }}\left(\sgn(k) \frac{1}{2\sqrt{|k\pi|}} + \frac{i}{2\sqrt{|k\pi|}} + O\Big(\frac{1}{|k|}\Big)\right)	\times e^{x \left(\sgn(k) \sqrt{|k\pi|} - \frac{1}{2} -i\sqrt{|k\pi|} + O(|k|^{-\frac{1}{2}})\right)} \\
+\bigg(e^{-\alpha_{1,k} -i(2k\pi +\alpha_{2,k}) + O(|k|^{-1})} 	- e^{\sgn(k) \sqrt{|k\pi|} - \frac{1}{2} -i\sqrt{|k\pi|} + O(|k|^{-\frac{1}{2}})} \bigg) \\
\times \frac{1}{k\pi e^{\sqrt{|k\pi|} +\frac{1}{\sqrt{|k|}} }}\left( -\sgn(k)\frac{1}{2\sqrt{|k\pi|}} - \frac{i}{2\sqrt{|k\pi|}} + O\Big(\frac{1}{|k|}\Big) \right) \times e^{x\left(-\sgn(k) \sqrt{|k\pi|} - \frac{1}{2} +i\sqrt{|k\pi|} + O(|k|^{-\frac{1}{2}}) \right)},
\end{multline}
We just calculate the first term of the difference  \eqref{difference_xi_tildephi}: for positive $k\geq k_0$ large, we observe that
\begin{align*}
&\frac{2i}{e^{\sqrt{|k\pi|} + \frac{1}{\sqrt{|k|}}  } } \left(e^{\sgn(k) \sqrt{|k\pi|} - \frac{1}{2} -i\sqrt{|k\pi|} + O(|k|^{-\frac{1}{2}})}  - e^{-\sgn(k) \sqrt{|k\pi|} - \frac{1}{2} +i\sqrt{|k\pi|} + O(|k|^{-\frac{1}{2}})} \right) \\
& \ \ \times  e^{-x\left(\alpha_{1,k} + i(2k\pi +\alpha_{2,k}) +O(|k|^{-1})\right)} - 2i\sgn(k) e^{-\frac{1}{2}-i\sgn(k)\sqrt{|k\pi|} } e^{-2ik\pi x} \\
& \sim_{+\infty} 2i e^{-\frac{1}{2} - i \sqrt{|k\pi|}} e^{-2ik\pi x} \left( e^{-x(\alpha_{1,k} +i\alpha_{2,k} + O(|k|^{-1})  )  }-1 \right)+O(\mod{k}^{-1})\\
& \sim_{+\infty}   2i e^{-\frac{1}{2} - i \sqrt{|k\pi|}} e^{-2ik\pi x} O(|k|^{-1})+O(\mod{k}^{-1}).
\end{align*}
The last inclusion holds due to the fact that  $\alpha_{1,k}\sim_{+\infty}O(|k|^{-1})$ and $\alpha_{2,k}\sim_{+\infty}O(|k|^{-1})$, thanks to the Lemma \ref{lemma_eigenvalue_eigenfunc}. 

Rest of the terms in the difference \eqref{difference_xi_tildephi} are always behaving like $O(|k|^{-1})$. Thus, we have
\begin{align*}
\norm{\xi_{\lambda^h_k} - \widetilde \phi_{k} }_{L^2}\leq C |k|^{-1} , \quad \forall |k|\geq k_0 \text{ large}.
\end{align*}
On the other hand, it is easy to deduce that
\begin{align*}
\norm{\eta_{\lambda^h_k}-\tilde{\psi}_k}_{L^2}=\norm{\eta_{\lambda^h_k}}_{L^2}\leq C|k|^{-1},
\end{align*}
thanks to Lemma \ref{lemma_bound_eigenfunc}. Hence,  we have 
\begin{align*}
\sum_{\mod{k}\geq k_0}\norm{\Phi_{\lambda^h_k}-\widetilde \Psi_{k}  }^2_{L^2\times L^2} \leq C \sum_{|k|\geq k_0}   |k|^{-2} <+\infty. 
\end{align*}
This ends the proof.
\end{proof}

\begin{proof}[\bf Proof of Proposition \ref{Prop-eigenfunctions}]
First, recall that the countable number of eigenfunctions $\{\Phi_{\lambda^p_k}\}_{k\geq k_0}$ and $\{\Phi_{\lambda^h_k}\}_{|k|\geq k_0}$, with their asymptotic expressions are already given by \eqref{eigen-1-xi}--\eqref{eigen-1-eta}, \eqref{eigen-2-xi}--\eqref{eigen-2-eta}, and obtained  in Section \ref{subsection-eigenfunc}. Also, recall the particular case when $\lambda_0=0$ is an eigenvalue with eigenfunction $\Phi_{\lambda_0}=(1,0)$. 
		 
Now, thanks to Lemma \ref{Lemma_quadratic_close}, we can apply the point (i) of Theorem \ref{Thm_Bao} to ensure the existence of  eigenfunctions for lower frequencies. More precisely, there exist an $n_0\in \mb N^*$ and the eigenfunctions $\{\Phi_{\widehat \lambda_n}\}_{1}^{n_0}$  associated to the eigenvalues $\{\widehat \lambda_n\}_{1}^{n_0}$ of the operator $A^*$, where 
\begin{align*}
\Phi_{\widehat \lambda_n} : = ( \xi_{\widehat \lambda_n} , \eta_{\widehat \lambda_n} ), \quad \text{for } \ 1 \leq n \leq n_0.
\end{align*}
Moreover, we can guarantee that the family 
\begin{align*}
\E(A^*) : = \Big\{\Phi_{\lambda^p_k}, \,  k\geq k_0; \ \ \Phi_{\lambda^h_k}, \,  |k|\geq k_0 \Big\} \cup \Big\{ \Phi_{\lambda_0}\Big\} \cup \Big\{ \Phi_{\widehat\lambda_n}, \, 1\leq n\leq n_0  \Big\}, 
\end{align*}
forms a Riesz basis in $L^2(0,1)\times L^2(0,1)$.  
 
As a consequence, we have that the set  of eigenfunctions $\E(A^*)$ forms a complete family in $H^{-s}_{{per}}(0,1)\times L^2(0,1)$ for any $s>0$.

The proof ends. 
\end{proof}

\begin{remark}\label{Remark-eigen-A}
In the same way, one can prove that the set of eigenvalues and eigenfunctions of $A$ (denoted by $\sigma(A)$ and $\mathcal E(A)$ respectively) have similar properties as of the eigenpairs of $A^*$.

In this case, we can find some $\widetilde k_0\in \mathbb N^*$ (large enough) such that $A$ has the eigenvalues of parabolic and hyperbolic nature for $|k|\geq \widetilde k_0$. For later use, we denote the eigenfunctions of $A$, respectively by $\widetilde \varPhi^p_{k}$, $k\geq \widetilde k_0$ and $\widetilde \varPhi^h_k$, $|k|\geq \widetilde k_0$ corresponding to the parabolic and hyperbolic branches of eigenvalues.   

Moreover,   using the result of \Cref{Thm_Bao}, we can show that the set $\mathcal E(A)$  forms a Riesz basis for the space $L^2(0,1)\times L^2(0,1)$ and indeed, $\mathcal E(A)$  generates the space $H^s_{per}(0,1)\times L^2(0,1)$ for $s>0$. 
\end{remark}

\smallskip 

\section{Estimations of the observation terms}\label{Approx-con}
\smallskip

In this section, we are going to find some lower bounds of the observation terms associated to our control systems. In this regard, we use the notations $\B^*_\rho$ and $\B^*_u$  which represent the observation operators for the density and velocity case respectively, and their formal expressions are given below.  
\begin{itemize}
	\item 
	The observation operator corresponding to  \eqref{lcnse3} (control in density) is defined by 
	\begin{align}\label{Form-observation}
		\B^*_\rho  =   \begin{pmatrix} 1 \\ 0 \end{pmatrix} \mathds{1}_{\{x=1\}} & : D(A^*)   \to \mb R ,
	\end{align}
	such that 
	\begin{align}\label{op_observation}
		\B^*_\rho \Phi=  \xi(1), \quad \forall  \Phi=(\xi, \eta) \in D(A^*).
	\end{align}
	\item The observation operator corresponding to  \eqref{lcnse4} (control in velocity) is  defined by   
	\begin{align} \label{Form-observation-vel}
		\B^*_u  =  \mathds{1}_{\{x=1\}} \begin{pmatrix} 1 \\ 0 \end{pmatrix}  + \mathds{1}_{\{x=1\}} \begin{pmatrix}  0 \\ 1 \end{pmatrix}\frac{\partial}{\partial x}   & : D(A^*)   \to \mb R ,
	\end{align}
	such that
	\begin{align}\label{op_observation-vel}
		\B^*_u \Phi=  \xi(1)+\eta^\prime(1), \quad \forall  \Phi=(\xi, \eta) \in D(A^*).
	\end{align}
	
\end{itemize}

\subsection{Observation estimates when a control acts on density}

\begin{lemma}\label{Prop_Approximate_Controllability}
Recall the set of eigenfunctions $\E(A^*)$ given by \eqref{Eigenfamily} of the operator $A^*$.  Then,  we have the following result:
\begin{align}\label{non-zer-obs}
\B^*_\rho \Phi \neq 0 , \quad \forall \, \Phi \in \E(A^*),
\end{align}
where $\B^*_\rho$ is the observation operator associated to the system \eqref{lcnse3}, defined by \eqref{Form-observation}--\eqref{op_observation}.
\end{lemma}
\begin{proof}
	For the particular case when $\lambda_0=0$, the eigenfunction
	 $$\B^*_\rho \Phi_{\lambda_0} = 1 \neq 0.$$
		
Let us pick any $\Phi:=(\xi, \eta) \in \E(A^*)$ corresponding to some  eigenvalue $\lambda\neq 0$ and recall the eigenvalue  problem \eqref{eigenfunction}.   Substituting  the first equation of \eqref{eigenfunction} in the second one, we get 
\begin{align}\label{xi-eta-1}
\eta^{\prime \prime }(x) + \lambda \xi (x) - \lambda \eta(x)=0, \ \ \forall x \in (0,1).
\end{align} 
Performing a differentiation, we have 
\begin{align*}
\eta^{\prime \prime \prime }(x) + \lambda \xi^\prime (x) - \lambda \eta^\prime(x)=0, \ \ \forall x \in (0,1).
\end{align*} 
Then,	by substituting $\xi^\prime=\lambda\eta -\eta^{\prime \prime}-\eta^\prime$ in above, we get the differential equation satisfied only by $\eta$ as given by
\begin{subequations}
\begin{align}\label{etaeq}	&\eta^{\prime\prime\prime}(x)-\lambda\eta^{\prime\prime}(x)-2\lambda\eta^{\prime}(x)+\lambda^2\eta(x)=0  , \ \ \forall x \in (0,1),\\\label{etaboun}	
&\eta(0)=0,\ \ \eta(1)=0,\ \ \eta^{\prime\prime}(0)=\eta^{\prime\prime}(1).
\end{align}
\end{subequations}
To prove the proposition, we assume  in contrary that there exists some $\Phi =(\xi, \eta)\in \E(A^*)$ such that $\B^* \Phi= \xi(1)=0$ and  this gives from the relation \eqref{xi-eta-1} that
\begin{align*}
\eta^{\prime \prime }(0)= \eta^{\prime \prime}(1)=0. 
\end{align*}
Now, our claim is to prove that $\eta =0$ in $(0,1)$ which will imply $\xi=0$  and that is enough for  proving the proposition, since  all the eigenvalues and eigenfunctions are non-trivial.

\medskip 

Define an extension map $\vartheta:\mathbb{R}\to\mathbb{R}$ by
\begin{equation}\label{extension-map}
\vartheta(x)=\begin{cases}\eta(x),&\ \ x\in(0,1),\\0,&\ \ x\in\mathbb{R}\setminus(0,1).\end{cases}
\end{equation}
Then the transformed equation for  \eqref{etaeq} is
\begin{multline}\label{eq-u-prev} 
\vartheta^{\prime\prime\prime}(x)-\lambda \vartheta^{\prime\prime}(x)-2\lambda \vartheta^{\prime}(x)+\lambda^2 \vartheta(x) \\
=-\eta^{\prime\prime}(1)\delta_{x=1}+\eta^{\prime\prime}(0)\delta_{x=0}-\eta^{\prime}(1)(\delta_{x=1}^{\prime}-\lambda\delta_{x=1})+\eta^{\prime}(0)(\delta_{x=0}^{\prime}-\lambda\delta_{x=0}) , \quad \forall x \in \mb R.
\end{multline}
We mention here that the  idea of introducing this  extension map (as \eqref{extension-map}) to study the non-vanishing property of the observation terms has been addressed in \cite{Rosier97}. 

Let us use the conditions  $\eta^{\prime\prime}(0)=\eta^{\prime\prime}(1)=0$ in  \eqref{eq-u-prev}, which gives 
\begin{equation}\label{eq-u}
\vartheta^{\prime\prime\prime}(x)-\lambda \vartheta^{\prime\prime}(x)-2\lambda \vartheta^{\prime}(x)+\lambda^2 \vartheta(x)=-\eta^{\prime}(1)(\delta_{x=1}^{\prime}-\lambda\delta_{x=1})+\eta^{\prime}(0)(\delta_{x=0}^{\prime}-\lambda\delta_{x=0}) , \ \ \forall x \in \mb R.
\end{equation}
Thus, the existence of an $\eta$ satisfying \eqref{etaeq}--\eqref{etaboun} is equivalent to the existence of $\alpha,\beta,\lambda$ with $(\alpha,\beta)\neq(0,0)$ and $\lambda \in \sigma(A^*)$,   such that
\begin{equation}\label{eq-u-1}
\vartheta^{\prime\prime\prime}(x)-\lambda \vartheta^{\prime\prime}(x)-2\lambda \vartheta^{\prime}(x)+\lambda^2 \vartheta(x)=-\alpha(\delta_{x=1}^{\prime}-\lambda\delta_{x=1})+\beta(\delta_{x=0}^{\prime}-\lambda\delta_{x=0}), \ \ \forall x \in \mb R.
\end{equation}
Without loss of generality, we can assume $\alpha\neq0$. Indeed, $\alpha=\eta^{\prime}(1)=0$ implies $\eta=0$ from the eigen equation \eqref{etaeq}-\eqref{etaboun} and the assumption $\eta^{\prime\prime}(1)=0$. 

Taking Fourier transform on both sides of \eqref{eq-u-1}, we obtain
\begin{equation*}
\left((iz)^3-\lambda(iz)^2-2\lambda(iz)+\lambda^2\right)\hat{\vartheta}(z)=-\alpha(ize^{-iz}-\lambda e^{-iz})+\beta(iz-\lambda), \ \ \text{for } z \in \mb C.
\end{equation*}
Therefore
\begin{equation*}
\hat{\vartheta}(z)  =\frac{-\alpha e^{-iz}(iz-\lambda)+\beta(iz-\lambda)}{(iz)^3-\lambda(iz)^2-2\lambda(iz)+\lambda^2}=\frac{(-\alpha e^{-iz}+\beta)(iz-\lambda)}{(iz)^3-\lambda(iz)^2-2\lambda(iz)+\lambda^2}, \ \ \text{for } z \in \mb C.
\end{equation*}
Since $\hat{\vartheta}$ is the Fourier transform of a function $\eta\in H^1_0(0,1)$, by the Paley-Wiener theorem, the function
\begin{equation}
\hat{\vartheta}(z) =\frac{(-\alpha e^{-iz}+\beta)(iz-\lambda)}{(iz)^3-\lambda(iz)^2-2\lambda(iz)+\lambda^2} , \ \ \text{for } z \in \mb C,
\end{equation} 
is entire. Thus, the roots of $(iz)^3-\lambda(iz)^2-2\lambda(iz)+\lambda^2$ are also the roots of $(-\alpha e^{-iz}+\beta)(iz-\lambda)$ with the same multiplicity. So, the main work is to calculate the roots of 
\begin{align}\label{equ-numerator}
(-\alpha e^{-iz}+\beta)(iz-\lambda)=0, \quad \text{for } z \in \mb C.
\end{align}
Without loss of generality, we assume the following function  in  $iz \in \mb C$, 
\begin{equation}\label{eq-u-2}
\hat{\vartheta}(iz)=\frac{(-\alpha e^{z}+\beta)(-z-\lambda)}{-z^3-\lambda z^2+2\lambda z+\lambda^2}, \ \ \text{for } z \in \mb C.
\end{equation}
In \eqref{eq-u-2}, the roots of $(-\alpha e^{z}+\beta)(-z-\lambda)$ are $z=-\lambda$ and the zeros of $e^z=\frac{\beta}{\alpha}$ (as we have $\alpha \neq 0$). We also note that $-\lambda$ is not a root of the  polynomial equation
\begin{align}\label{cubic_poly}
-z^3-\lambda z^2+2\lambda z+\lambda^2 =0 .
\end{align} 
Let   $r_1, r_2, r_3$ be the roots of the equation \eqref{cubic_poly}. Then one must have 
\begin{equation*}
e^{r_1}=e^{r_2}=e^{r_3}=\frac{\beta}{\alpha},
\end{equation*}
that is,
\begin{equation*}
r_2=r_1+2il\pi,\ \ r_3=r_1+2im\pi,\ \ \forall  l,m \in \mathbb{Z}.
\end{equation*}
Since these are also the roots of the polynomial equation \eqref{cubic_poly}, we have 
\begin{subequations} 
\begin{align}
\label{rel-1}	r_1 + r_2 + r_3  &=-\lambda, \\
\label{rel-2}	 r_1 r_2 + r_2 r_3 + r_1 r_3 &= -2\lambda, \\ 
\label{rel-3}	  r_1 r_2 r_3 &=\lambda^2. 
\end{align}
\end{subequations}
Thus, we readily have from \eqref{rel-1}, 
\begin{equation}\label{p} 
3r_1+2il\pi+2im\pi=-\lambda, \ \ \text{ i.e., } \ r_1=\frac{1}{3}(-\lambda-2il\pi-2im\pi), 
\end{equation}
and therefore, 
\begin{equation}\label{q,r}
r_2=\frac{1}{3}(-\lambda+4il\pi-2im\pi), \ \ \  r_3=\frac{1}{3}(-\lambda-2il\pi+4im\pi).
\end{equation}
From the relation \eqref{rel-2}, we have
\begin{equation*}
\lambda^2+6\lambda+4(l^2-lm+m^2)\pi^2=0.
\end{equation*}
Solving, we get some particular values of eigenvalues $\lambda$, which are
\begin{equation*}
\lambda=\frac{-6\pm\sqrt{36-16\pi^2(l^2-lm+m^2)}}{2}=-3\pm\sqrt{9-4\pi^2(l^2-lm+m^2)}
\end{equation*}
Since $l,m\in\mathbb{Z}$, therefore $l^2-lm+m^2\geq0$\footnote{For $lm=0$, $l^2-lm+m^2=l^2+m^2\geq0$, for $lm<0, l^2-lm+m^2>0$ and for $lm>0, l^2-lm+m^2=(l-m)^2+lm>0$.} and $l^2-lm+m^2=0$ if and only if $l=m=0$\footnote{If $l^2-lm+m^2=0$ and $m\neq0$ then $(\frac{l}{m})^2-(\frac{l}{m})+1=0$ has no real solutions. Therefore $m=0$ and hence $l=0$.}. Thus for $l\neq0$ and $m\neq0$
\begin{equation}\label{value-lambda}
\lambda=-3\pm i\sqrt{4\pi^2(l^2-lm+m^2)-9}.
\end{equation}
Finally, from the relation \eqref{rel-3}, we get
\begin{align*}
-\lambda^3-27\lambda^2+(-12l^2\pi^2+12lm\pi^2-12m^2\pi^2)\lambda-16il^3\pi^3+24il^2m\pi^3+24ilm^2\pi^3-16im^3\pi^3=0,
\end{align*}
and the real part of which gives 
\begin{equation*}
\Re(\lambda^3)+27\Re(\lambda^2)+12\pi^2(l^2-lm+m^2)\Re(\lambda)=0.
\end{equation*}
Now, using the particular values of $\lambda$ as obtained in \eqref{value-lambda}, we get
\begin{equation*}
-108\pi^2 (l^2-lm+m^2)-108=0
\end{equation*}
that is
\begin{equation*}
l^2-lm+m^2=-1/\pi^2 <0, 
\end{equation*}
which is a contradiction. Therefore,  the only possibility is $l=m=0$, which gives from \eqref{value-lambda} that $\lambda = -6$ (since $\lambda \neq 0$). 

On the other hand, from \eqref{p} and \eqref{q,r}, it  implies that the cubic polynomial \eqref{cubic_poly} has root $-\frac{\lambda}{3}=2$ of multiplicity $3$ which is again a contradiction. Hence,  we conclude that $\alpha \neq 0$ cannot be possible. 
 
Therefore, the only possibility is  $\alpha=\beta=0$, which gives (comparing \eqref{eq-u} and \eqref{eq-u-1}) that $\eta^\prime(0)=\eta^\prime(1)=0$. But, we have by assumptions that $\eta(0)=\eta(1)=0$ and $\eta^{\prime \prime}(0)=\eta^{\prime \prime}(1)=0$, i.e.,  $\eta=0$ in $(0,1)$, and as a consequence, $\xi=0$ in $(0,1)$.

Hence, the proof of lemma follows. 
\end{proof}

\medskip 

The next lemma shows that the observation terms satisfy some lower bounds which are not exponentially small. In fact, these lower bounds are crucial to conclude the null-controllability of the concerned system  \eqref{lcnse3}. 
We have the following lemma.
\begin{lemma}[Observation estimates: control on density] \label{observation_estimate-density}
Recall the set of eigenfunctions $\E(A^*)$, given by \eqref{Eigenfamily}.  Then, there exists a constant $C>0$, independent in $k$, such that we have the following observation estimates for the parabolic and hyperbolic parts: 
\begin{subequations}
\begin{align}\label{obs_den-1}
|\B^*_\rho \Phi_{\lambda^p_k}| &\geq \frac{C}{k\pi}  , \quad \forall k \geq k_0, \\ \label{obs_den-2}
|\B^*_\rho \Phi_{\lambda^h_k} | &\geq C , \quad \forall |k| \geq k_0,
\end{align}
\end{subequations}
where the number $k_0$ introduced by Lemma \ref{lemma_eigenvalue_eigenfunc}.
\end{lemma}

\begin{proof}  
Using  the definition of $\B^*_\rho$ introduced by \eqref{op_observation}, we have
\begin{align*}
&\mathcal B^*_\rho \Phi_{\lambda^p_k} = \xi_{\lambda^p_k}(1), \quad \forall k \geq k_0, \\
&\mathcal B^*_\rho \Phi_{\lambda^h_k} = \xi_{\lambda^h_k}(1) , \quad \forall |k|\geq k_0.
\end{align*}
We recall Lemma \ref{Prop_Approximate_Controllability} which ensures that $\B^*_\rho \Phi \neq 0$ for all $\Phi \in \mathcal E(A^*)$.  In particular   $\xi_{\lambda^p_k}(1) \neq 0$ for all $k \geq k_0$ and $\xi_{\lambda^h_k}(1)\neq 0$ for all $|k| \geq k_0$. Thus, it is enough to obtain the lower bounds for large modulus of $k$. 

\medskip 

(i)	Let us recall the expressions of $\xi_{\lambda^p_k}$ from \eqref{eigen-1-xi}, so that we have
\begin{multline*}
{\xi_{\lambda^p_k}(1)} = e^{\left(-k^2\pi^2 -2c_k k\pi - 2 i d_k k\pi  + O(1) \right)} \times O\big(\frac 1k \big)  \\ 
+\left(\frac{i}{k\pi} +O\big(\frac{1}{k^2}\big)\right) \left( e^{i(k\pi +c_k)+O(k^{-1})-\frac{1}{2}-d_k} -  e^{-k^2\pi^2 -2c_kk\pi -2id_k k\pi + O(1)} \right) e^{\left(-i(k\pi +c_k)-\frac{1}{2}+d_k +O(k^{-1}) \right)} \\ 
+\left(-\frac{i}{k\pi} +O\big(\frac{1}{k^2}\big)\right)  \left( e^{-k^2\pi^2 -2c_k k\pi-2id_k k\pi  + O(1)}- e^{ - i(k\pi +c_k)-\frac{1}{2}+d_k+O(k^{-1})}\right) e^{\left(i(k\pi +c_k)-\frac{1}{2}-d_k +O(k^{-1}) \right)}.
\end{multline*}
From the above expression, it is easy to deduce that there exists some constant $C>0$, independent in $k$, such that 
\begin{align*} 
\mod{\xi_{\lambda^p_k}(1)}	 \geq \frac{C}{k\pi},  \quad \text{for all $k\geq k_0$ large}. 
\end{align*}
		
\medskip 
		
(ii)	On the other hand, from the expression of $\xi_{\lambda^h_k}$ given by \eqref{eigen-2-xi}, we have
\begin{multline}\label{xi-gamma-k-1}
{\xi_{\lambda^h_k}(1)} =  \left(e^{\sgn(k) \sqrt{|k\pi|} - \frac{1}{2} -i\sqrt{|k\pi|} + O(|k|^{-\frac{1}{2}})}  - e^{-\sgn(k) \sqrt{|k\pi|} - \frac{1}{2} +i\sqrt{|k\pi|} + O(|k|^{-\frac{1}{2}})} \right) \\
\times \frac{(-\alpha_{1,k}+2ik\pi + O(1))}{k\pi e^{ \sqrt{|k\pi|} + \frac{1}{\sqrt{|k|}} } } \times e^{-\alpha_{1,k} - i(2k\pi +\alpha_{2,k}) +O(|k|^{-1})} \\
+ \bigg(e^{-\sgn(k)\sqrt{|k\pi|} - \frac{1}{2} +i\sqrt{|k\pi|} + O(|k|^{-\frac{1}{2}})}  - e^{-\alpha_{1,k} - i(2k\pi +\alpha_{2,k}) + O(|k|^{-1})}   \bigg) \\
\times \frac{1}{k\pi e^{ \sqrt{|k\pi|} + \frac{1}{\sqrt{|k|}} } }\left(\sgn(k) \frac{1}{2\sqrt{|k\pi|}} + \frac{i}{2\sqrt{|k\pi|}} + O\Big(\frac{1}{|k|}\Big)\right)	\times e^{\left(\sgn(k) \sqrt{|k\pi|} - \frac{1}{2} -i\sqrt{|k\pi|} + O(|k|^{-\frac{1}{2}})\right)} \\
+\bigg(e^{-\alpha_{1,k} -i(2k\pi +\alpha_{2,k}) + O(|k|^{-1})}  - e^{\sgn(k) \sqrt{|k\pi|} - \frac{1}{2} -i\sqrt{|k\pi|} + O(|k|^{-\frac{1}{2}})} \bigg) \\
\times \frac{1}{k\pi e^{ \sqrt{|k\pi|} + \frac{1}{\sqrt{|k|}} } } \left( -\sgn(k)\frac{1}{2\sqrt{|k\pi|}} - \frac{i}{2\sqrt{|k\pi|}} + O\Big(\frac{1}{|k|}\Big) \right) \times e^{\left(-\sgn(k) \sqrt{|k\pi|} - \frac{1}{2} +i\sqrt{|k\pi|} + O(|k|^{-\frac{1}{2}}) \right)}.
\end{multline}
We deduce the result for positive $k\geq k_0$, the same will be true for $k\leq -k_0$. 

The leading term in \eqref{xi-gamma-k-1} for $k\geq k_0$ is 
\begin{align*}
J:= \frac{1}{k\pi e^{\frac{1}{\sqrt{k}}}}\big(-\alpha_{1,k} + 2i{k\pi} + O(1)\big)	e^{-\alpha_{1,k}-i(2k\pi  +\alpha_{2,k})+O(k^{-\frac{1}{2}})} \times  e^{-\frac{1}{2}-i\sqrt{k\pi}+O(k^{-\frac{1}{2}})},
\end{align*}
and thus, it is clear to understand that there exists some $C>0$, independent in $k$, such that
\begin{align*} 
|\xi_{\lambda^h_k}(1)|\geq C, \quad \text{for all $k\geq k_0$ large}. 
\end{align*}
This completes the proof. 
\end{proof}

\subsection{Observation estimates when a control acts on  velocity}

 In this case, we can only able to prove the appropriate lower bounds of the observation terms $\B^*_u\Phi$ for large frequencies. For lower frequencies, it is difficult to conclude anything. In fact, 
the idea applied to prove \Cref{Prop_Approximate_Controllability} for the density case is not working when the position of the control is changed to the velocity component.

Let us write the following lemma that gives the observation estimates for higher frequencies.
\begin{lemma}[Observation estimates: control in velocity]\label{observation_estimate-vel}
	Recall the set of eigenfunctions $\E(A^*)$ given by \eqref{Eigenfamily} and the number $k_0$ introduced in Lemma \ref{lemma_eigenvalue_eigenfunc}.  Then, there exists some natural number $\widehat k_0\geq k_0$  and a constant $C>0$ such that we have the following observation estimates: 
	\begin{subequations}
		\begin{align}
			\label{obs_est-1}
			|\B^*_u \Phi_{\lambda^p_k}| &\geq C k\pi , \quad \forall k \geq \widehat k_0, \\ \label{obs_est-2}
			|\B^*_u\Phi_{\lambda^h_k}| &\geq \frac{C}{\sqrt{|k\pi|}} \quad \forall |k|\geq \widehat k_0,
		\end{align}
	\end{subequations}
	where $C$ does not depend on $k$.
\end{lemma}
\begin{proof}
	Using the definition of $\B^*_u$, given by \eqref{Form-observation-vel}--\eqref{op_observation-vel}, we have 
	\begin{align*}
		&\mathcal B^*_u \Phi_{\lambda^p_k} = \xi_{\lambda^p_k}(1) + \eta^\prime_{\lambda^p_k}(1), \quad \forall k \geq k_0, \\
		&\mathcal B^*_u \Phi_{\lambda^h_k} = \xi_{\lambda^h_k}(1) + \eta^\prime_{\lambda^h_k}(1), \quad \forall |k|\geq k_0. 
	\end{align*} 
	\begin{itemize}
		
		\item[(i)] Recall the expressions of $\xi_{\lambda^p_k}$ and  $\eta_{\lambda^p_k}$, given by  \eqref{eigen-1-xi} and \eqref{eigen-1-eta} respectively, so that we have
		\begin{multline*}
			\xi_{\lambda^p_k}(1) + \eta^\prime_{\lambda^p_k}(1)=
			e^{-k^2\pi^2 -2c_k k\pi -2id_k k\pi+ O(1)} \times O\big(\frac 1k\big)  \\
			+ \big(-ik\pi +O(1) \big) \left( e^{i(k\pi +c_k)+O(k^{-1})-\frac{1}{2}-d_k} -  e^{-k^2\pi^2 -2c_kk\pi -2id_k k\pi + O(1)} \right) e^{-i(k\pi +c_k)-\frac{1}{2}+d_k +O(k^{-1})} \\
			+	\big(ik\pi +O(1) \big)  \left( e^{-k^2\pi^2 -2c_k k\pi-2id_k k\pi  + O(1)}- e^{ - i(k\pi +c_k)-\frac{1}{2}+d_k+O(k^{-1})}\right) e^{i(k\pi +c_k)-\frac{1}{2}-d_k +O(k^{-1}) } .
		\end{multline*}
		The leading term in the above expression is 
		\begin{align*}
			L_k:=-\big[2ik\pi +O(1)\big] e^{-1 +O(k^{-1})} ,
		\end{align*}
		and that there exists some $\widehat k_0\in \mb N^*$ which is larger or equal to $k_0$ such that, one has   
		\begin{align*}
			|L_k| \geq C k\pi, \quad \text{for all  } k\geq  \widehat k_0.
		\end{align*}
		Other lower order terms can be  bounded by $Ck\pi e^{-k^2\pi^2}$ and thus,
		\begin{align*}
			\mod{\xi_{\lambda^p_k}(1)+\eta^\prime_{\lambda^p_k}(1)}\geq Ck\pi(1-e^{-k^2\pi^2}), \quad \text{for all } k\geq \widehat k_0,
		\end{align*}
		for some $C>0$ that does not depend on $k$. So, the estimate \eqref{obs_est-1} follows.

		\medskip 
		
		\item[(ii)] For the set of eigenfunctions \eqref{eigen-2-xi}--\eqref{eigen-2-eta} associated to $\lambda^h_k$, the observation terms are 
		\begin{align*}
			&{\xi_{\lambda^h_k}(1)+\eta_{\lambda^h_k}^{\prime}(1)}= \left(e^{\sgn(k)\sqrt{\mod{k\pi}} - \frac{1}{2} -i\sqrt{\mod{k\pi}} + O(\mod{k}^{-\frac{1}{2}}) } - e^{-\sgn(k)\sqrt{\mod{k\pi}} - \frac{1}{2} +i\sqrt{\mod{k\pi}} + O(\mod{k}^{-\frac{1}{2}})} \right)\\
			&\quad\quad\times \frac{O(1)}{k\pi e^{\sqrt{\mod{k\pi}}+\frac{1}{\sqrt{\mod{k}}}}}\times e^{-\alpha_{1,k} - i(2k\pi +\alpha_{2,k}) +O(\mod{k}^{-1})} \\
			&\quad+\left(e^{-\sgn(k)\sqrt{\mod{k\pi}} -\frac{1}{2} +i\sqrt{\mod{k\pi}} + O(\mod{k}^{-\frac{1}{2}})} - e^{-\alpha_{1,k} - i(2k\pi +\alpha_{2,k}) + O(\mod{k}^{-1})}   \right)\\
			&\quad\quad\times\frac{\Big(\sgn(k)\sqrt{\mod{k\pi}}-i\sqrt{\mod{k\pi}}+O(1)\Big)}{k\pi e^{\sqrt{\mod{k\pi}}+\frac{1}{\sqrt{\mod{k}}}}}\times e^{\big(\sgn(k)\sqrt{\mod{k\pi}} - \frac{1}{2} -i\sqrt{\mod{k\pi}} + O(\mod{k}^{-\frac{1}{2}}) \big)} \\
			&\quad+ \left(e^{-\alpha_{1,k} - i(2k\pi +\alpha_{2,k}) + O(\mod{k}^{-1})}  - e^{\sgn(k)\sqrt{\mod{k\pi}} - \frac{1}{2} -i\sqrt{\mod{k\pi}} + O(\mod{k}^{-\frac{1}{2}}) } \right)\\
			&\quad\quad\times\frac{\Big(-\sgn(k)\sqrt{\mod{k\pi}}+i\sqrt{\mod{k\pi}}+O(1)\Big)}{k\pi e^{\sqrt{\mod{k\pi}}+\frac{1}{\sqrt{\mod{k}}}}}\times e^{\big(-\sgn(k)\sqrt{\mod{k\pi}} - \frac{1}{2} +i\sqrt{\mod{k\pi}} + O(\mod{k}^{-\frac{1}{2}})  \big)}.
		\end{align*}
		
		For $k\geq k_0>0$,
		the leading term in the above expression is 
		\begin{align*}
			J_k:= -\frac{1}{k\pi}(\sqrt{\mod{k\pi}}-i\sqrt{\mod{k\pi}}+O(1))	e^{-\alpha_{1,k}-i(2k\pi + \alpha_{2,k})+O(\mod{k}^{-1})} e^{-\frac{1}{2}-\frac{1}{\sqrt{|k|}}+i\sqrt{\mod{k\pi}}+O(\mod{k}^{-\frac{1}{2}})}, 
		\end{align*}
		Then,	it is clear to understand that there exists some $C>0$ and  $\widehat k_0$ (this  $\widehat k_0$ could be larger than previous  and in that case we shall choose the bigger one for both)  such that
		\begin{align*}
			|J_k| \geq \frac{C}{\sqrt{\mod{k \pi}}}, \quad \text{for all } k\geq \widehat k_0  .
		\end{align*}	
		The lower order terms can be bounded by $C \left( 1/\mod{k\pi} +  (1/\sqrt{\mod{k\pi}}) e^{-\sqrt{\mod{k\pi}}}\right)$ for large $k$ and thus we have 
		\begin{align*}
			|\xi_{\lambda^h_k}(1) + \eta_{\lambda^h_k}^{\prime}(1)|\geq \frac{C}{\sqrt{\mod{k\pi}}}  \quad \forall k\geq \widehat k_0.
		\end{align*}
		
		A similar estimate can be deduced for $k\leq -k_0<0$. We skip the details. 
		
		The proof is complete.
	\end{itemize}
\end{proof}

\section{Null-controllability using the method of moments: control on density}\label{Section-Moments}

\smallskip 

In this section, we will first prove the null-controllability of the system \eqref{lcnse3}, more precisely  Theorem \ref{Thm-control-dens}. To prove the existence of a null-control $p\in L^2(0,T)$ for the system \eqref{lcnse3}, we use the so-called moments technique. As a consequence, we prove the null-controllability of the system \eqref{main-control-prbm}, that is \Cref{Thm-control-dens-Diri}.
Using the above null-controllability results, we shall also prove that the system \eqref{main-control-prbm} is approximately controllable in $L^2(0,1)\times L^2(0,1)$, that is \Cref{Coro-approx}. 

Let $(\rho,u)$ be the solution to the system \eqref{lcnse3} with a boundary control $p$. Then, the following lemma gives an equivalent criterion for the null-controllability.
\begin{lemma}\label{Lemma-control-form}
The system \eqref{lcnse3} is null-controllable at time $T>0$ if and only if there exists a control $p\in  L^2(0,T)$ such that
\begin{equation}\label{formulation_control}
\ip{\begin{pmatrix}\sigma(0) \\ v(0)\end{pmatrix}}{\begin{pmatrix}\rho_0\\u_0\end{pmatrix}}_{ H^{-s}_{per}\times L^2, H^s_{per}\times L^2} = -\int_{0}^{T}\overline{\sigma(t,1)}p(t)dt,
\end{equation}
where  $(\sigma,v)$ is the solution to  the adjoint system \eqref{lcnse_adjoint} with $(f,g)=(0,0)$  and any given final  data $(\sigma_T,v_T)\in  H^{-s}_{per}(0,1)\times L^2(0,1)$ for $s>0$.
\end{lemma}

\smallskip 
\subsection{Formulation of the mixed parabolic-hyperbolic moments problem}
To prove the null-controllability of system \eqref{lcnse3}, we shall formulate and solve a set of moments problem using the strategy developed in \cite{Hansen94}. For sake of completeness, we recall the main results from \cite{Hansen94} in Appendix \ref{Appendix-Han} and in this section, we reformulate the problem with respect to our setting. 

\smallskip 

Let us  recall  that the set of eigenvalues $\sigma(A^*)$, given by \eqref{spectrum}, satisfies the following.

\smallskip 

The sequence $\{\lambda^h_k\}_{|k|\geq k_0}$ satisfies
\begin{itemize}
\item[(i)] For all $|k|, |j|\geq k_0$,  $\lambda^h_k\neq\lambda^h_j$ unless $j=k$,
\item[(ii)] $\lambda^h_k = \beta - c k\pi i + \nu_k$  for all   $|k|\geq k_0$, with $\beta=-1$, $c=2$ and $\nu_k = -\alpha_{1,k}-i\alpha_{2,k}$, 
\end{itemize}
where it is clear that $\{\nu_k\}_{|k|\geq k_0}\in\ell^2$, from the properties of $\alpha_{1,k}$ and $\alpha_{2,k}$ (see Lemma \ref{lemma_eigenvalue_eigenfunc}).   

\smallskip 

Also, there exists positive constants $A_0, B_0 , \delta, \epsilon$ and $0\leq\theta<\pi/2$ for which $\{\lambda^p_k\}_{k\geq k_0}$ satisfies
\begin{itemize}
\item[(i)] $\mod{\arg(-\lambda^p_k)}\leq\theta$ for all $k\geq k_0$,
\item[(ii)] $\mod{\lambda^p_k- \lambda^p_j}\geq\delta\mod{k^2-j^2}$ for all $k\neq j$, $k,j\geq k_0$,
\item[(iii)] $\epsilon(A_0 + B_0 k^2) \leq \mod{\lambda^p_k} \leq A_0 + B_0 k^2$ for all  $k \geq k_0$. 
\end{itemize}  
Moreover, the  sets of eigenvalues are mutually disjoint. 
\begin{align*}
\{\lambda^h_k \}_{|k|\geq k_0}   \cap  \{ \lambda^p_k\}_{k\geq k_0} =\emptyset, \ \  
\{\lambda^h_k \}_{|k|\geq k_0} \cap \big( \{\lambda_0\} \cup \{ \widehat \lambda_n\}_{n=1}^{n_0}\big) = \emptyset, \ \
\{\lambda^p_k \}_{k\geq k_0} \cap \big(\{\lambda_0\}\cup\{ \widehat \lambda_n\}_{n=1}^{n_0}\big)=\emptyset.
\end{align*}
Thus, the set of spectrum $\sigma(A^*)$  satisfies the Hypothesis \ref{assump-han} in \Cref{Appendix-Han}  except for the finite set $\{\lambda_0\}\cup\{\widehat \lambda_n\}_{n=1}^{n_0}$. But this will not lead any problem to construct and solve the associated moments equations. Let us go to the detail. 

\smallskip 

\paragraph{\bf General setting}
We  first  recall Theorem \ref{thm-han} and Lemma \ref{lem-han} from \Cref{Appendix-Han}. As per those results,  our  goal is to find uniformly separated spaces $\mathcal W_{[0,T]}$ and $\mathscr E_{[0,T]}$ in $L^2(0,T)$ for $T> t_c=1$ (where $t_c=2/c$ as introduced in Appendix \ref{Appendix-Han} and in our case $c=2$).

We start with $T>1$. Then,  we pick a subset $\{\widehat \lambda_{n_l}\}_{l=1}^{l_0}$ (where $l_0\leq n_0$)  from the finite set  of eigenvalues $\{\widehat \lambda_n\}_{n=1}^{n_0}$ in such a way that the set
\begin{align}\label{set-riesz-basis-T}
\{1\}\cup\{e^{\lambda^h_k  t}\}_{|k|\geq k_0} \cup  \{e^{\widehat \lambda_{n_l}  t}\}_{l=1}^{l_0} 
\end{align} 
forms a {\em Riesz basis} for the space defined by 
\begin{align}\label{W-0-T}
\mathcal W_{[a,a+T]} = \text{closed span } \big(\{1\}\cup\{e^{\lambda^h_k  t}\}_{|k|\geq k_0} \cup  \{e^{\widehat \lambda_{n_l}  t}\}_{l=1}^{l_0}\big) \ \text{ in } L^2(a,a+T), \ \text{for any } a\in \mathbb R ,
\end{align}
and this is certainly possible in the light of   \Cref{lem-han} since  we consider $T>1$. 
In above, the singleton $\{1\}$ is arising due to the  eigenvalue $\lambda_0=0$ as $e^{\lambda_0 t}=1$.

\smallskip 

Next, we denote the rest of the elements of $\{\widehat \lambda_{n}\}_{n=1}^{n_0}$ by $\{\widehat \lambda_{n_j}\}_{j=1}^{j_0}$ so that  $$ \{\widehat \lambda_{n_l}\}_{l=1}^{l_0}\cup \{\widehat \lambda_{n_j}\}_{j=1}^{j_0} = \{\widehat \lambda_{n}\}_{n=1}^{n_0}.$$

Finally, we consider the space 
\begin{align}\label{E-0-T}
\mathscr 	E_{[0,T]} = \text{closed span } \big(\{e^{-\lambda^p_k  t}\}_{k\geq k_0} \cup  \{e^{-\widehat \lambda_{n_j}  t}\}_{j=1}^{j_0}\big) \ \text{ in } L^2(0,T).
\end{align}
\begin{lemma}\label{Lemma-uni-sepa}
The spaces $\mathcal W_{[0,T]}$ and $\mathscr E_{[0,T]}$ defined by \eqref{W-0-T} and \eqref{E-0-T} respectively, are uniformly separated in $L^2(0,T)$ for $T>1$. This does not hold for $T\leq 1$.
\end{lemma}
Since  the family \eqref{set-riesz-basis-T}  forms a Riesz basis  for the space $\mathcal W_{[0,T]}$  for $T>1$, then following the idea of proving Theorem \ref{thm-han} given by \cite[Theorem 4.2]{Hansen94}, we can prove Lemma \ref{Lemma-uni-sepa}.  
	
\medskip 

\paragraph{\bf The set of moments problem}
Recall that the set of eigenfunctions $\E(A^*)$ of $A^*$ (given by \eqref{Eigenfamily}) defines a complete family in $ H^{-s}_{per}(0,1)\times L^2(0,1)$ for $s>0$, thanks to Proposition \ref{Prop-eigenfunctions}. Thus, it is enough to check the control problem \eqref{formulation_control} for the eigenfunctions of $A^*$.
In what follows,  the problem \eqref{lcnse3} is null-controllable at given time $T>1$ if and only if there exists some $p\in L^2(0,T)$ such that we have the following: 
\begin{align}\label{Moments-parabolic}
\begin{dcases}
-\int_{0}^{T} \overline{e^{\lambda^p_k (T-t)}} p(t)\,dt &= 	m_{1,k}, \quad \forall k\geq k_0, \\
-\int_{0}^{T} \overline{e^{\widehat \lambda_{n_j}(T-t)}} p(t)\,dt&= 	m_{1,j}, \quad \forall \, 1\leq j\leq j_0,
\end{dcases}
\end{align}
\begin{align}\label{Moments-hyperbolic}
\begin{dcases}
	-\int_0^T e^{\lambda_0(T-t)} p(t) \, dt &= m_{2,0}, \ \ \text{i.e., } \ -\int_0^T p(t)\, dt = m_{2,0}, \\
-\int_{0}^{T} \overline{e^{\lambda^h_k (T-t)}} p(t)\,dt &= 	m_{2,k}, \quad \forall |k|\geq k_0, \\
-\int_{0}^{T} \overline{e^{\widehat \lambda_{n_l}(T-t)}} p(t)\,dt&= 	m_{1,l}, \quad \forall \, 1\leq l\leq l_0,
\end{dcases}
\end{align}
where 
\begin{align}\label{m-1-k}
\begin{dcases} 
m_{1,k} = \frac{\overline{e^{\lambda^p_k T}} \ip{\vector{\xi_{\lambda^p_k}}{\eta_{\lambda^p_k}}}{\vector{\rho_0}{u_0}}_{{{H}}_{per}^{-s}\times L^2, H_{per}^{s}\times L^2}}{\overline{\xi_{\lambda^p_k}(1)}}, \quad \forall k\geq k_0, \\
m_{1,j} = \frac{	\overline{e^{\widehat \lambda_{n_j} T}}
\ip{\vector{\xi_{\widehat \lambda_{n_j}}}{\eta_{\widehat \lambda_{n_j}}}}{\vector{\rho_0}{u_0}}_{{{H}}_{per}^{-s}\times L^2, H_{per}^{s}\times L^2}}{\overline{\xi_{\widehat \lambda_{n_j}}(1)} } , \quad \forall \, 1\leq j\leq j_0,
\end{dcases}
\end{align}
and 
\begin{align}\label{m-2-k}
\begin{dcases} 
	m_{2,0} = \left\langle \vector{1}{0}, \vector{\rho_0}{u_0}   \right\rangle_{H^{-s}_{per}\times L^2, H^s_{per}\times L^2} = \int_0^1 \rho_0(x) \, dx, \\
m_{2,k} = \frac{\overline{e^{\lambda^h_k T}} \ip{\vector{\xi_{\lambda^h_k}}{\eta_{\lambda^h_k}}}{\vector{\rho_0}{u_0}}_{{{H}}_{per}^{-s}\times L^2, H_{per}^{s}\times L^2}}{\overline{\xi_{\lambda^h_k}(1)}}, \quad \forall |k|\geq k_0, \\
m_{2,l} = \frac{	\overline{e^{\widehat \lambda_{n_l} T}}\ip{\vector{\xi_{\widehat \lambda_{n_j}}}{\eta_{\widehat \lambda_{n_j}}}}{\vector{\rho_0}{u_0}}_{{{H}}_{per}^{-s}\times L^2,H_{per}^{s}\times L^2}}{\overline{\xi_{\widehat \lambda_{n_l}}(1)} }, \quad \forall \, 1\leq l\leq l_0.
\end{dcases}
\end{align}

The above set of equations \eqref{Moments-parabolic}--\eqref{Moments-hyperbolic} are the so-called moments problem which are well-defined since $\B^*_{\rho}\Phi=\xi(1)\neq 0$ for any eigenfunction $\Phi\in \E(A^*)$ as proved in Lemma \ref{Prop_Approximate_Controllability}. Let us now study   the solvability of those equations.

\subsection{Proof of the null-controllability result}
We are now ready to prove our controllability result.
\begin{proof}[\bf Proof of Theorem \ref{Thm-control-dens}] Let any parameter $s>1/2$, initial data $(\rho_0, u_0)\in  H^{s}_{per}(0,1)\times L^2(0,1)$ and time $T>1$ be given. 

Then, our goal is to apply the result of Theorem \ref{Thm-main_Han} in \Cref{Appendix-Han} to solve the set of moments problem \eqref{Moments-parabolic}--\eqref{Moments-hyperbolic}.
To do this it is enough to show the following facts:  for any $r>0$
\begin{align}\label{Claim-1} 
|m_{1,k}| e^{r k}  \to 0 \  \text{ as } \ k \to +\infty,
\end{align}
and 
\begin{align}\label{Claim-2}
\sum_{|k|\geq k_0} |m_{2,k}|^2 <+\infty .
\end{align}
On the other hand, the finite sequences $\{m_{1,j}\}_{j=1}^{j_0}$ and $\{m_{2,0}\}\cup\{m_{2,l}\}_{l=1}^{l_0}$ are always bounded (since all the observation terms are non-zero due to \Cref{Prop_Approximate_Controllability}) and thus there is no trouble for the  lower frequencies. 

\smallskip 

\begin{itemize} 
\item[--]  We mainly consider the case when $1/2< s< 1$. The case when $s\geq 1$ can be deduced in a similar manner using the precise bounds of the eigenfunctions given by \Cref{lemma_bound_eigenfunc}. 

Recall the expression of $m_{1,k}$ for $k\geq k_0$ from \eqref{m-1-k}, we have
\begin{align}\label{estimate-m-1-k}
|m_{1,k}| &\leq C  \|(\rho_0, u_0) \|_{H^{s}_{per}\times L^2} \,  e^{\Re(\lambda^p_k) T} \frac{\|\xi_{\lambda^p_k}\|_{H^{-s}_{per}} +  \|\eta_{\lambda^p_k}\|_{L^2} }{|\overline{\xi_{\lambda^p_k}(1)}|}  \\ \notag 
& \leq C  \|(\rho_0, u_0)\|_{H^{s}_{per}\times L^2} \, e^{-k^2\pi^2 T} k \pi \big( k^{-s-1} +1 \big), 
\end{align}
thanks to the bounds of the eigenfunctions \eqref{bound_xi_lambda} and  observation estimate \eqref{obs_den-1}. Indeed,  the bound  \eqref{estimate-m-1-k} directly implies the Claim \eqref{Claim-1}  due to the presence of $e^{-k^2\pi^2T}$ in the right hand side of \eqref{estimate-m-1-k}.   

Thus,  in view of  \Cref{Prop-sol-moment-para} in \Cref{Appendix-Han}, there exists a function $p_1\in \mathscr E:=\mathscr E_{[0,T]}$ that solves the set of equations \eqref{Moments-parabolic}. 

\smallskip 	
	
\item[--] As previous, we start with $1/2<s<1$.	Here, we show that $\{m_{2,k}\}_{|k|\geq k_0} \in \ell^2$.  In this regard, we recall the bounds of the eigenfunctions given by \eqref{bound_xi_gamma}  and the  observation estimate \eqref{obs_den-2}, which yields 
\begin{align*}
\sum_{|k| \geq k_0}|m_{2,k}|^2 &\leq C\|(\rho_0,u_0)\|^2_{{{H}}_{per}^{s}\times L^2}\sum_{\mod{k}\geq k_0}\frac{\|\xi_{\lambda^h_k}\|^2_{H^{-s}_{per}} +  \|\eta_{\lambda^h_k}\|^2_{L^2} }{|\overline{\xi_{\lambda^h_k}(1)}|^2}  \\
&\leq C\|(\rho_0,u_0)\|^2_{{{H}}_{per}^{s}\times L^2}  \sum_{\mod{k}\geq k_0} \left(\mod{k}^{-2s}+|k|^{-2}\right)\\
&\leq C\|(\rho_0,u_0)\|^2_{{{H}}_{per}^{s}\times L^2} .
\end{align*}
The above series converges due to the sharp choice $s>1/2$ and indeed, it is clear that for $s\leq 1/2$, the series $\displaystyle \sum_{\mod{k}\geq k_0}\frac{1}{\mod{k\pi}^{2s}}$ diverges. For $s\geq 1$, we use the bound $\|\xi_{\lambda^h_k}\|_{H^{-s}_{per}}\leq C |k|^{-1}$ and then accordingly the result follows. 
	
Therefore,  in view of \Cref{Prop-sol-moment-hyper} in \Cref{Appendix-Han}, there exists a function $p_2\in \mathcal W:=\mathcal W_{[0,T]}$ that solves the set of equations \eqref{Moments-hyperbolic}. 
\end{itemize}

Now, our goal is to apply \Cref{Thm-main_Han} from \Cref{Appendix-Han}.  We have, as consequence of \Cref{Lemma-uni-sepa}, the space 
\begin{align}\label{space-V}
\mathcal V:= \mathscr E + \mathcal W 
\end{align}
is closed and thus a Hilbert space with $\|\cdot\|_{\mathcal V}:=\|\cdot\|_{L^2(0,T)}$, so $\mathcal V= \mathscr E\oplus \mathcal W$. Likewise, we have $\mathcal V:=\mathscr E^{\perp}\oplus \mathcal W^{\perp}$. Therefore, the restrictions $P_{\mathscr E}|_{\mathcal W^\perp}$ and $P_{\mathcal W}|_{\mathscr E^\perp}$ are isomorphisms, where $P_{\mathscr E}$ and $P_{\mathcal W}$ denote the orthogonal projections from $\mathcal V$ onto $\mathscr E$ and $\mathcal W$ respectively.

Now, we consider 
\begin{align}
p:= (P_{\mathscr E}|_{\mathcal W^{\perp}})^{-1} p_1 + (P_{\mathcal W}|_{\mathscr E^{\perp}})^{-1} p_2 ,
\end{align}
which certainly belongs to the space $L^2(0,T)$ and  simultaneously solves the set of moments problem \eqref{Moments-parabolic}--\eqref{Moments-hyperbolic} for $T>1$ and  any $\rho_0 \in {{H}}_{per}^{s}(0,1)$ for $s>1/2$ and $u_0\in L^2(0,1)$.
	
This concludes  the proof of Theorem \ref{Thm-control-dens}. 
\end{proof}

\begin{remark}[Compatibility condition]
Integrating the first equation of \eqref{lcnse3} in $(0,1)$ and then $(0,T)$, we get
\begin{equation*}
	\int_{0}^{1}\rho(T,x)dx-\int_{0}^{1}\rho_0(x)dx-\int_{0}^{T}p(t)dt=0.
\end{equation*}
So, for the null-controllability we need the following compatibility condition
\begin{equation*}
	\int_{0}^{1}\rho_0(x)dx=-\int_{0}^{T}p(t)dt.
\end{equation*}
But this always holds true since $p$ satisfies the moments problem introduced before  and in particular the first  equation of   \eqref{Moments-hyperbolic}.  
\end{remark}

\begin{proof}[\bf Proof of Theorem \ref{Thm-control-dens-Diri}]
We have already shown the existence of a control $p\in L^2(0,T)$ for the system \eqref{lcnse3}. Now, to prove the existence of a control $h\in L^2(0,T)$ for the main control problem \eqref{main-control-prbm}, all we need to show that $\rho(t,1)\in L^2(0,T)$ where $\rho$ is the solution component of the system \eqref{lcnse3} associated with the control function $p\in  L^2(0,T)$. But the proof for $\rho(t,1)\in L^2(0,T)$ is followed from an auxiliary result given in \Cref{Appendix-D} (\Cref{lemma-hidden}). 

Hence, we define $h(t)=\rho(t,1)+p(t)$ for all $t\in (0,T)$, which  plays the role of a Dirichlet control function for the main system \eqref{main-control-prbm}. The proof is complete. 
\end{proof}

\subsection{Approximate controllability result with $L^2\times L^2$ initial data}\label{sec-aprox}
As a corollary of the above  null-controllability results, in this section, we shall prove the approximate controllability of the  system \eqref{main-control-prbm}  in the space $L^2(0,1)\times L^2(0,1)$.

\begin{proof}[\bf Proof of \Cref{Coro-approx}] The proof will be made in two steps.
	
	\begin{itemize} 
\item {\bf Step 1:} {\em Approximate controllability of the auxiliary system \eqref{lcnse3}.}	
Recall that we have proved the null-controllability of the auxiliary system \eqref{lcnse3} at time $T>1$ in the space ${H}^s_{{per}}(0,1) \times L^2(0,1)$ for $s>\frac{1}{2}$. 
Let $U_0=(\rho_0,u_0)\in L^2(0,1)\times L^2(0,1)$ be any given initial state and $\epsilon>0$ be arbitrary. Then there exists $\widetilde{U}_0=(\widetilde{\rho}_0,\widetilde{u}_0)\in {H}^s_{{per}}(0,1)\times L^2(0,1)$ for $s>\frac{1}{2}$ such that
	\begin{equation}
		\norm{U_0-\widetilde{U}_0}_{L^2\times L^2}<\epsilon.
	\end{equation}
	
	Let us denote the reachable set $R(T;U_0)$ starting from initial data $U_0\in L^2(0,1)\times L^2(0,1)$ for time $T>1$,  
	\begin{align*}
		R(T; U_0) = \Big\{ U(T)=(\rho(T), u(T)) : U=(\rho,u) \text{ is the solution to \eqref{lcnse3}} \\
		\text{with control  } p\in L^2(0,T) \Big\}. 
	\end{align*}

	Let any $\overline U_0 \in {H}^s_{per}(0,1)\times L^2(0,1)$ be given. Since  the system \eqref{lcnse3} is null-controllable at time $T>1$ in ${H}^s_{\text{per}}(0,1)\times L^2(0,1)$ for $s>\frac{1}{2}$, one has  $0\in R(T;\widetilde{U}_0-\overline{U}_0)=R(T;\widetilde{U}_0)-S(T)\overline{U}_0$ for any $T>1$, where $\{S(t)\}_{t\geq 0}$ is the strongly continuous semigroup defined by $(A, D(A))$ in $L^2(0,1)\times L^2(0,1)$. Hence, $S(T)\overline{U}_0\in R(T;\widetilde{U}_0)$ for any $T>1$. 
	
	Now, since the set of eigenfunctions of $A$ generates the space $H^s_{per}(0,1)\times L^2(0,1)$ (see \Cref{Remark-eigen-A}), we note that 
	$S(T)\big({H}^s_{{per}}(0,1)\times L^2(0,1)\big)$ is dense in $L^2(0,1)\times L^2(0,1)$. 
	
	On the other hand, we have the following relation for any $T>1$ and $s>1/2$,
	\begin{equation}
		S(T)\big({H}^s_{{per}}(0,1)\times L^2(0,1)\big) \subseteq R(T;\widetilde{U}_0)\subseteq L^2(0,1)\times L^2(0,1).
	\end{equation}
	But,  $S(T)\big({H}^s_{{per}}(0,1)\times L^2(0,1)\big)$ is dense in $L^2(0,1)\times L^2(0,1)$ and thus $R(T;\widetilde{U}_0)$ is dense in $L^2(0,1)\times L^2(0,1)$ for any $T>1$ and $s>\frac{1}{2}$. More precisely, for any $U_T\in L^2(0,1)\times L^2(0,1)$, there exists a control $\widetilde{p}\in L^2(0,T)$ such that the solution $U_{\widetilde{U}_0,\widetilde{p}}:=(\widetilde \rho, \widetilde u)$ to the system 
	\begin{align}\label{approx-sys-1}
		\begin{dcases}
			\widetilde\rho_t+	\widetilde\rho_x +\widetilde u_{x} =0   &\text{in } (0,T)\times (0,1),\\
				\widetilde u_{t} - 	\widetilde u_{xx} +	\widetilde u_x +	\widetilde\rho_x =0 &\text{in } (0,T)\times (0,1),\\
				\widetilde \rho(t,0)=	\widetilde \rho(t,1)+ 	\widetilde p(t)      &\text{for } t\in (0,T) , \\
				\widetilde u(t,0)=0, \ 	\widetilde u(t,1)=0  &\text{for } t\in (0,T) ,\\
				\widetilde \rho(0,x)=	\widetilde \rho_0(x),\  	\widetilde u(0,x)=	\widetilde u_0(x)  &\text{for } x\in(0,1),
		\end{dcases}
	\end{align}
	 satisfies (where $\widetilde U_0=(\widetilde \rho_0, \widetilde u_0)\in H^s(0,1)\times L^2(0,1)$ for $s>1/2$, as  introduced before)
	\begin{equation}\label{eq-eps-1}
		\norm{U_{\widetilde{U}_0,\widetilde{p}}(T,\cdot)-U_T(\cdot)}_{L^2\times L^2}<\epsilon,
	\end{equation}
	for any $T>1$. 
	
	Now, consider the following system 
	\begin{align}\label{approx-sys-2}
		\begin{dcases}
			\rho_t+	\rho_x +u_{x} =0   &\text{in } (0,T)\times (0,1),\\
			 u_{t} - u_{xx} + u_x +	\rho_x =0 &\text{in } (0,T)\times (0,1),\\
		 \rho(t,0)=	 \rho(t,1)+ 	\widetilde p(t)      &\text{for } t\in (0,T) , \\
			u(t,0)=0, \  u(t,1)=0  &\text{for } t\in (0,T) ,\\
			\rho(0,x)=	\rho_0(x),\  u(0,x)= u_0(x)  &\text{for } x\in(0,1),
		\end{dcases}
	\end{align}
	with the same function $\widetilde p$ as in \eqref{approx-sys-2} and the initial data $U_0=(\rho_0,u_0)\in L^2(0,1)\times L^2(0,1)$ as chosen earlier. 
	
Denote the solution to \eqref{approx-sys-2} by $U_{U_0, \widetilde p}:=(\rho, u)$. 	Then, using the continuity estimate \eqref{continuity_estimate} from \Cref{Thm-existnce-control-sol}, we deduce that 
	\begin{equation}\label{eq-eps-2}
		\norm{U_{U_0,\widetilde{p}}(T,\cdot)-U_{\widetilde{U}_0,\widetilde{p}}(T,\cdot)}_{L^2\times L^2}\leq C\norm{U_0-\widetilde{U}_0}_{L^2\times L^2}<\epsilon.
	\end{equation}
	Finally, using \eqref{eq-eps-1}--\eqref{eq-eps-2} and  the triangle inequality, we conclude
	\begin{align}\label{approx-3}
		&\norm{U_{U_0,\widetilde{p}}(T,\cdot)-U_T(\cdot)}_{L^2\times L^2}\\  \notag 
		&\leq\norm{U_{U_0,\widetilde{p}}(T,\cdot)-U_{\widetilde{U}_0,\widetilde{p}}(T,\cdot)}_{L^2\times L^2}+\norm{U_{\widetilde{U}_0,\widetilde{p}}(T,\cdot)-U_T(\cdot)}_{L^2\times L^2} \\  \notag 
		&<2\epsilon.
	\end{align}
	This proves the proof for the approximate controllability of the system \eqref{lcnse3} in the space $L^2(0,1)\times L^2(0,1)$ when $T>1$.

\smallskip 

\item {\bf Step 2:} {\em Approximate controllability of the  system \eqref{main-control-prbm}.}  This is a consequence of the previous step. We set 
\begin{align*}
	\widehat h(t) = \rho(t,1) + \widetilde p(t), \quad \forall t\in (0,T),
\end{align*}
in the equation \eqref{approx-sys-2} (we can do that since $\rho(t,1)\in L^2(0,T)$ due to \Cref{lemma-hidden} in \Cref{Appendix-D}), which works as an approximate control for the main system \eqref{main-control-prbm} since  one has  $U_{U_0, \widehat h}:= U_{U_0, \widetilde p}$ and thus from the estimate \eqref{approx-3}, we have 
 \begin{align*}
 	\norm{U_{U_0,\widehat{h}}(T,\cdot)-U_T(\cdot)}_{L^2\times L^2} <2\epsilon.
 	\end{align*}
This completes the proof of the approximate controllability of \eqref{main-control-prbm} in the space $L^2(0,1)\times L^2(0,1)$ provided $T>1$.

\end{itemize}
\end{proof}

\section{A partial null-controllability result for the velocity case: using an Ingham-type inequality}\label{sec_ingham}


\smallskip 

In this section, we prove the partial null-controllability (in a subspace of $H^s_{per}(0,1)\times L^2(0,1)$) of the system \eqref{lcnse4} (that is, Theorem \ref{thm_control_veloct})   by showing a proper  observability inequality. 
A parabolic-hyperbolic joint Ingham-type inequality is the main ingredient to conclude this result.

Let $(\rho,u)$ be the solution to the system \eqref{lcnse4} with a boundary control $q$ acting on velocity. Then, the following lemma gives an equivalent criterion for the null-controllability.
\begin{lemma}\label{Lemma-control-form-vel}
	The system \eqref{lcnse4} is null-controllable at time $T>0$ if and only if there exists a control $q\in  L^2(0,T)$ such that
	\begin{equation}\label{formulation_control-vel}
		\ip{\begin{pmatrix}\sigma(0) \\ v(0)\end{pmatrix}}{\begin{pmatrix}\rho_0\\u_0\end{pmatrix}}_{ H^{-s}_{per}\times L^2,  H^s_{per}\times L^2} = \int_{0}^{T}\left(\overline{\sigma(t,1)} + \overline{v_x(t,1)}\right) q(t)dt,
	\end{equation}
	where $(\sigma,v)$ is the solution to the adjoint system \eqref{lcnse_adjoint} with $(f,g)=(0,0)$ and any given final  data $(\sigma_T,v_T)\in H^{-s}_{per}(0,1)\times L^2(0,1)$.
\end{lemma}



\subsection{A combined parabolic-hyperbolic  Ingham-type inequality}
This subsection is devoted to  prove a Ingham-type inequality stated in \Cref{prop_ingham}. The proof will be based on following the idea of \cite[Theorem 4.2]{Chowdhury14}. 

\begin{proof}[\bf Proof of \Cref{prop_ingham}]
	Let us denote $\tilde{\lambda}_k=\lambda_k-\beta$,  $\forall k\in\mathbb{N}^*$ and $\tilde{\gamma}_k=\gamma_k-\beta$,  $\forall {k}\in\mathbb{Z}$. Let $N\in\mathbb{N}^*$ be as given in the hypothesis. Then, we have the following known parabolic and hyperbolic Ingham inequalities
	\begin{align}
		&\int_{0}^{T}\mod{\sum_{k\geq N}a_ke^{\tilde{\lambda}_k(T-t)}}^2dt\geq C\sum_{k\geq N}\mod{a_k}^2e^{2\Re(\tilde{\lambda}_k)T}\ \ \text{for any } T>0,\label{parabolic_ingham}\\
		&C_1\sum_{\mod{k}\geq N}\mod{b_k}^2\leq\int_{0}^{T}\mod{\sum_{\mod{k}\geq N}b_ke^{\tilde{\gamma}_k(T-t)}}^2dt\leq C_2\sum_{\mod{k}\geq N}\mod{b_k}^2\ \ \text{for any }T>1
		\label{hyperbolic_ingham},
	\end{align}
	see \cite{Montes99}, \cite{Edward06}, \cite{Ingham} (and references therein) for more details.
	
	\smallskip 
	
	Let us denote
	\begin{equation}\label{U-p-U-h}
		U^p(t)=\sum_{k\geq N}a_ke^{\tilde{\lambda}_k(T-t)},\ \ U^h(t)=\sum_{\mod{k}\geq N}b_ke^{\tilde{\gamma}_k(T-t)},\ \ t\geq0,
	\end{equation}
	and
	\begin{equation}
		U(t)=U^p(t)+U^h(t),\ \ t\geq0.
	\end{equation}
	We also define for $t>1$
	\begin{subequations}
		\begin{align}
			\widetilde{U}^p(t)=U^p(t)-U^p(t-1)=\sum_{k\geq N}a_k\left(1-e^{\tilde{\lambda}_k}\right)e^{\tilde{\lambda}_k(T-t)},  \label{U-tilde-p}\\
			\widetilde{U}^h(t)=U^h(t)-U^h(t-1)=\sum_{\mod{k}\geq N}b_k\left(1-e^{\tilde{\gamma}_k}\right)e^{\tilde{\gamma}_k(T-t)}, \label{U-tilde-h}
		\end{align}
	\end{subequations}
	and 
	\begin{equation}\label{def-U-tilde} 
		\widetilde{U}(t)=\widetilde{U}^p(t)+\widetilde{U}^h(t)=U(t)-U(t-1).
	\end{equation}
	Then, we have
	\begin{align*}
		\int_{1}^{T}\mod{\widetilde{U}(t)}^2 dt & \leq\int_{1}^{T}\mod{U(t)}^2dt+\int_{1}^{T}\mod{U(t-1)}^2dt \\
		& \leq C\int_{0}^{T}\mod{U(t)}^2dt  .
	\end{align*}
	We now compute the $L^2$-norms of the functions $\widetilde{U}^p$ and $\widetilde{U}^h$ separately. Applying the hyperbolic Ingham inequality \eqref{hyperbolic_ingham}, we get
	\begin{align*}
		\int_{1}^{T}\mod{\widetilde{U}^h(t)}^2 dt \leq C\sum_{\mod{k}\geq N}\mod{b_k}^2\mod{1-e^{\tilde{\gamma}_k}}^2 . 
	\end{align*}
	Since $1-e^{\tilde{\gamma}_k}=1-e^{\nu_k}$ and $\{\nu_k\}_{|k|\geq N} \in \ell^2$, we can choose $N$ large enough such that $\mod{1-e^{\tilde{\gamma}_k}}^2<\epsilon$ for all $\mod{k}\geq N$. It follows that,
	\begin{equation}\label{esti_U_h}
		\int_{1}^{T}\mod{\widetilde{U}^h(t)}^2 dt \leq C\epsilon \sum_{\mod{k}\geq N}\mod{b_k}^2.
	\end{equation}
	
	Now, recall \eqref{def-U-tilde} so that one has $\displaystyle \widetilde U^p(t) = \widetilde U(t) -\widetilde U^h(t)$. Using the triangle inequality, we get
	\begin{align}\label{esti-U-tilde-p}
		\int_{1}^{T}\mod{\widetilde{U}^p(t)}^2dt&\leq C\int_{1}^{T}\mod{\widetilde{U}(t)}^2dt+C\int_{1}^{T}\mod{\widetilde{U}^h(t)}^2dt  \\  \notag 
		&\leq C\int_{0}^{T}\mod{U(t)}^2dt+C\epsilon\sum_{\mod{k}\geq N}\mod{b_k}^2 . 
	\end{align}
	
	Let be  $0<\tau<T$. Applying the parabolic Ingham inequality \eqref{parabolic_ingham} to the quantity $\widetilde U^p(t)$ (given by \eqref{U-tilde-p}),  we obtain
	\begin{align*}
		\int_{T-\tau}^{T}\mod{\widetilde{U}^p(t)}^2dt=\int_{0}^{\tau}\mod{\widetilde{U}^p(T-t)}^2dt & \geq C\sum_{k\geq N}\mod{a_k}^2 |1-e^{\tilde{\lambda}_k}|^2 e^{2\Re(\tilde{\lambda}_k)\tau}  \\
		&\geq C\sum_{k\geq N}\mod{a_k}^2e^{2\Re(\tilde{\lambda}_k)\tau},
	\end{align*}
	thanks to the properties of $\tilde \lambda_k$.  
	Note that the above constant $C$ depends on $\tau$. Let us now choose $\tau>0$ small enough such that $T-\tau>1$. Thus, we get
	\begin{align}\label{13-2}
		\int_{1}^{T}\mod{\widetilde{U}^p(t)}^2 dt
		\geq\int_{T-\tau}^{T}\mod{\widetilde U^p(t)}^2dt\geq C\sum_{k\geq N}\mod{a_k}^2e^{2\Re(\tilde{\lambda}_k)\tau}  . 
	\end{align}
	
	Recall the function $U^p(t)$ given by \eqref{U-p-U-h},  we deduce that
	\begin{align}\label{13-1} 
		\int_{0}^{T-\tau}\mod{U^p(t)}^2dt &
		\leq\sum_{k\geq N}\mod{a_k}^2 \int_{0}^{T-\tau}e^{2\Re(\tilde{\lambda}_k)(T-t)}dt \\
		&\leq \sum_{k\geq N}\mod{a_k}^2\bigg|\frac{e^{\Re(\tilde{\lambda}_k)\tau}-e^{2\Re(\tilde{\lambda}_k)T}}{2\Re(\tilde{\lambda}_k)} \bigg| \notag \\
		&\leq C\sum_{k\geq N}\mod{a_k}^2e^{2\Re(\tilde{\lambda}_k)\tau}, \notag
	\end{align}
	thanks to fact that $\mod{\Re(\tilde \lambda_k)}^2\geq C$ for  $k\geq N$ large enough (combining the hypothesis (ii) and (iv) in \Cref{prop_ingham} satisfied by $\{\lambda_k\}_{k\in \mb N^*}$). 
	
	Now, using the facts \eqref{13-2} and \eqref{esti-U-tilde-p} in \eqref{13-1}, we have 
	\begin{equation}\label{13-3}
		\int_{0}^{T-\tau}\mod{U^p(t)}^2dt\leq C\left(\int_{0}^{T}\mod{U(t)}^2dt+\epsilon\sum_{\mod{k}\geq N}\mod{b_k}^2\right).
	\end{equation}
	
	\smallskip 
	
	Since $T-\tau>1$, applying the hyperbolic Ingham inequality \eqref{hyperbolic_ingham} to $U^h(t)$ and then following a  triangle inequality, we have 
	\begin{align*}
		\sum_{\mod{k}\geq N}\mod{b_k}^2\leq C\int_{0}^{T-\tau}\mod{U^h(t)}^2dt & \leq C\left(\int_{0}^{T-\tau}\mod{U(t)}^2dt + \int_{0}^{T-\tau}\mod{U^p(t)}^2dt \right)\\
		&\leq C\left(\int_{0}^{T}\mod{U(t)}^2dt+\epsilon\sum_{\mod{k}\geq N}\mod{b_k}^2\right),
	\end{align*}
	thanks to the estimate \eqref{13-3}. 
	
	Now, fix  $\epsilon>0$ small enough such that $1-C\epsilon>0$. As a consequence, there is some constant $C>0$ depending only on $T$ such that, we have 
	\begin{equation}\label{13-4} 
		\sum_{\mod{k}\geq N}\mod{b_k}^2 dt \leq C\int_{0}^{T}\mod{U(t)}^2dt.
	\end{equation}
	
	On the other hand, using the parabolic Ingham inequality to $U^p(t)$, followed by a triangle inequality, hyperbolic Ingham inequality (to $U_h(t)$) and the result \eqref{13-4}, we obtain 
	\begin{align*}
		\sum_{k\geq N}\mod{a_k}^2e^{2\Re(\tilde{\lambda}_k)T}\leq C\int_{0}^{T}\mod{U^p(t)}^2dt&\leq C\left(\int_{0}^{T}\mod{U(t)}^2dt+\int_{0}^{T}\mod{U^h(t)}^2dt\right)\\
		&\leq C\left(\int_{0}^{T}\mod{U(t)}^2dt+\sum_{\mod{k}\geq N}\mod{b_k}^2dt\right)\\
		&\leq C\int_{0}^{T}\mod{U(t)}^2dt .
	\end{align*}

	Thus, eventually we have 
	\begin{equation}\label{13-5}
		\sum_{k\geq N}\mod{a_k}^2e^{2\Re(\tilde{\lambda}_k)T}+\sum_{\mod{k}\geq N}\mod{b_k}^2\leq C\int_{0}^{T}\mod{U(t)}^2dt.
	\end{equation}

	Recall that $\tilde \lambda_k = \lambda_k-\beta$, $\tilde \gamma_k= \gamma_k -\beta$, and that
	\begin{align}\label{13-6} 
		\int_{0}^{T}\mod{U(t)}^2dt&=\int_{0}^{T}\mod{\sum_{k\geq N}a_ke^{(\lambda_k-\beta)(T-t)}+\sum_{\mod{k}\geq N}b_ke^{(\gamma_k-\beta)(T-t)}}^2dt\\  \notag 
		&\leq C\int_{0}^{T}\mod{\sum_{k\geq N}a_ke^{\lambda_k(T-t)}+\sum_{\mod{k}\geq N}b_ke^{\gamma_k(T-t)}}^2dt.
	\end{align}
	Moreover, it is easy to see that $$e^{2\Re(\tilde{\lambda}_k)T}=e^{2\Re(\lambda_k)T-2\Re(\beta)T}\geq Ce^{2\Re(\lambda_k)T}$$
	for some $C>0$ and thus combining \eqref{13-5} and \eqref{13-6}, we obtain 
	\begin{equation*}
		\sum_{k\geq N}\mod{a_k}^2e^{2\Re(\lambda_k)T}+\sum_{\mod{k}\geq N}\mod{b_k}^2\leq C\int_{0}^{T}\mod{\sum_{k\geq N}a_ke^{\lambda_k(T-t)}+\sum_{\mod{k}\geq N}b_ke^{\gamma_k(T-t)}}^2dt.
	\end{equation*}

	Finally, adding the finitely many terms   in the above summation  using a similar idea as in \cite[Theorem 4.3, Chapter 4]{Micu04} (since $\{\gamma_k\}_{k\in \mathbb Z}$ and $\{\lambda_k\}_{k\in \mathbb N^*}$ are disjoint), we can conclude that
	\begin{equation}
		\sum_{k\in\mathbb{N}^*}\mod{a_k}^2e^{2\Re(\lambda_k)T}+\sum_{k\in\mathbb{Z}}\mod{b_k}^2\leq C\int_{0}^{T}\mod{\sum_{k\in\mathbb{N}^*}a_ke^{\lambda_k(T-t)}+\sum_{k\in\mathbb{Z}}b_ke^{\gamma_k(T-t)}}^2dt.
	\end{equation} 
	This completes the proof.
\end{proof}

\smallskip

\subsection{Observability inequality and a partial null-controllability result}

Recall that when a control is acting on the velocity part, we have from \Cref{observation_estimate-vel}, 
\begin{align*}
	|\B^*_u \Phi_{\lambda^p_k}| \geq Ck\pi, \ \  \forall k\geq \widehat k_0, \quad 	|\B^*_u \Phi_{\lambda^h_k}| \geq \frac{C}{\sqrt{|k\pi|}} \ \  \forall |k|\geq \widehat k_0.
\end{align*}
Recall the number $\widetilde k_0$ from \Cref{Remark-eigen-A} and consider 
$$K_0:= \max \{ \widehat k_0, \widetilde k_0\},$$ and define the space 
\begin{align}\label{Space-V}
\mathcal	H:= \text{closure} \Big(\Span \left\{ \widetilde \varPhi^p_k, \ k\geq K_0, \ \  \widetilde \varPhi^h_k, \ |k|\geq K_0  \right\}\Big)  \ \ \text{in } H^s_{per}(0,1)\times L^2(0,1),
\end{align}
where we recall that $\{\widetilde \varPhi^p_{k}\}_{k\geq \widetilde k_0}$ and $\{\widetilde \varPhi^h_{k}\}_{|k|\geq \widetilde k_0}$ denote the eigenfunctions associated to the parabolic and hyperbolic branches of eigenvalues of the operator $A$, see \Cref{Remark-eigen-A}. It is clear that the space $\mathcal H$ has finite codimension.

Using \cite[Theorems 5, 6; Chapter 1]{Young}, it can be checked that the dual of the space $\mathcal H$ is 
\begin{align}\label{Space-Vstar}
	\mathcal H^*:= \text{closure} \Big(\Span \left\{\Phi_{\lambda^p_k}, \ k\geq K_0, \ \  \Phi_{\lambda^h_k}, \ |k|\geq K_0  \right\}\Big)  \ \ \text{in } H^{-s}_{per}(0,1)\times L^2(0,1),
\end{align}
where we recall that $\{\Phi_{\lambda^p_k}\}_{k\geq k_0}$ and $\{\Phi_{\lambda^h_{k}}\}_{|k|\geq k_0}$ denote the eigenfunctions associated to the parabolic and hyperbolic branches of eigenvalues of the operator $A^*$, see \Cref{Prop-eigenfunctions}.

\medskip 

We are now ready to prove our second main result, i.e., Theorem \ref{thm_control_veloct} of our work.

\begin{proof}[\bf Proof of Theorem \ref{thm_control_veloct}]
Let us consider $(f,g)=(0,0)$ in the adjoint system \eqref{lcnse_adjoint} and  the  final data $(\sigma_T, v_T)$  of the following form:
\begin{align}\label{given-data}
	(\sigma_T,v_T)& =  \sum_{k\geq K_0} a_k \Phi_{\lambda^p_{k}} + \sum_{|k|\geq K_0} b_k \Phi_{\lambda^h_k} \\ \notag 
	&=\sum_{k\geq K_0}a_k(\xi_{\lambda^p_k}, \eta_{\lambda^p_k}) + \sum_{|k|\geq K_0} b_k (\xi_{\lambda^h_k}, \eta_{\lambda^h_k} ) ,
\end{align} 
where $\displaystyle \sum_{k\geq K_0} |a_k|^2+ \sum_{|k|\geq K_0} |b_k|^2  <+\infty$.

 Therefore, the solution to the adjoint system \eqref{lcnse_adjoint} looks like
\begin{align}\label{adjoint_in_terms_eigenfunc}
	(\sigma,v) = \sum_{k\geq K_0}a_k e^{\lambda^p_k (T-t)} (\xi_{\lambda^p_k}, \eta_{\lambda^p_k}) + \sum_{|k|\geq K_0} b_k e^{\lambda^h_k(T-t)} (\xi_{\lambda^h_k}, \eta_{\lambda^h_k} ) ,
\end{align}
and this yields 
\begin{align} \label{B_star_u}
	\B^*_u (\sigma, v) &= \sigma(t,1)+v_x(t,1)   \notag \\ 
	&=
	\sum_{k\geq K_0} a_k \, e^{\lambda^p_k (T-t)}  \left(\xi_{\lambda^p_k}(1)+\eta_{\lambda^p_k}^{\prime}(1)\right)
	+\sum_{|k|\geq K_0} b_k \, e^{\lambda^h_k (T-t)} \left(\xi_{\lambda^h_k}(1)+\eta_{\lambda^h_{k}}^{\prime}(1)\right).
\end{align}

Now	in one hand,  using the Ingham-type inequality \eqref{Ingham-ineq}, we have 
	\begin{align*}
		&\int_0^T |\sigma(t,1)+v_x(t,1)|^2 \, dt \\
		&
		\geq
		C_1\bigg(\sum_{k\geq K_0 }\mod{a_k(\xi_{\lambda^p_k}(1)+\eta_{\lambda^p_k}^{\prime}(1))}^2e^{2\Re(\lambda^p_k)T}+\sum_{|k|\geq K_0 }\mod{b_k(\xi_{\lambda^h_k}(1)+\eta_{\lambda^h_{k}}^{\prime}(1))}^2\bigg),
	\end{align*}
	for some $C_1>0$.
	
	On the other hand, the solution $(\sigma, v)$ to the adjoint system given by \eqref{adjoint_in_terms_eigenfunc} (with given data as in \eqref{given-data}) satisfies the following
	\begin{align}\label{Adjoint_at_zero}
		\|(\sigma(0), v(0))\|^2_{H^{-s}_{per} \times L^2}  
		&\leq C_2 \bigg(\sum_{k\geq K_0}|a_k|^2e^{2\Re(\lambda^p_k) T} \|(\xi_{\lambda^p_k}, \eta_{\lambda^p_k}) \|^2_{ H^{-s}_{per}\times L^2} \\  \notag
		&  \qquad  + \sum_{|k|\geq K_0} |b_k|^2 e^{2\Re (\lambda^h_k) T} \|(\xi_{\lambda^h_k}, \eta_{\lambda^h_k}) \|^2_{ H^{-s}_{per}\times L^2} 
\bigg),
	\end{align}
	for some $C_2>0$.
	
	We now recall the observation estimates  given by Lemma \ref{observation_estimate-vel} for large $|k|\geq \widehat k_0$, which yields
	\begin{align}\label{Use_obs_est}
		&\sum_{k\geq  K_0} \mod{a_k(\xi_{\lambda^p_k}(1)+\eta_{\lambda^p_k}^{\prime}(1))}^2e^{2\Re(\lambda^p_k)T}+\sum_{|k|\geq K_0 }\mod{b_k(\xi_{\lambda^h_k}(1)+\eta_{\lambda^h_{k}}^{\prime}(1))}^2\\
		& \ \quad \geq C_1 \left(\sum_{k\geq  K_0} |a_k|^2 e^{2 \Re (\lambda^p_k) T} k^{2} + \sum_{|k|\geq K_0} |b_k|^2 |k|^{-1}\right).  \notag 
	\end{align}
	Using the bounds of the eigenfunctions given by Lemma \ref{lemma_bound_eigenfunc}, it follows that, 
	%
	%
	\begin{align}\label{Adjoint_zero_estimate}
		&\sum_{k\geq K_0}|a_k|^2e^{2\Re(\lambda^p_k) T} \|(\xi_{\lambda^p_k}, \eta_{\lambda^p_k}) \|^2_{ H^{-s}_{per}\times L^2} + \sum_{|k|\geq K_0} |b_k|^2 e^{2\Re (\lambda^h_k) T} \|(\xi_{\lambda^h_k}, \eta_{\lambda^h_k}) \|^2_{ H^{-s}_{per}\times L^2} \\ \notag 
		&\quad 	\leq  C_2\left(\sum_{k\geq K_0} |a_k|^2 e^{2\Re (\lambda^p_k) T} \left(k^{-2s-2}+1  \right) + \sum_{|k|\geq K_0} |b_k|^2 \left(|k|^{-2s} + |k|^{-2} \right)\right).
	\end{align}
	Straightaway,  there exists some $\widehat K_0\geq K_0$ such that for all $|k|\geq \widehat K_0$, we have
	$C_2 k^2> C_1 \big(k^{-2s-2}+1\big)$ for any $s>0$ and
	$C_2|k|^{-1}>C_1 \big(|k|^{-2s} + |k|^{-2}\big)$  for any  $s>\frac{1}{2}$.  Using these and  combining \eqref{Use_obs_est}, \eqref{Adjoint_zero_estimate}, we  can write for any $s>\frac{1}{2}$, that
	\begin{align}\label{Towards_Obser}
		&\sum_{k\geq  \widehat K_0}\mod{a_k(\xi_{\lambda^p_k}(1)+\eta_{\lambda^p_k}^{\prime}(1))}^2e^{2\Re(\lambda^p_k)T}+\sum_{\mod{k}\geq \widehat K_0}\mod{b_k(\xi_{\lambda^h_k}(1)+\eta_{\lambda^h_{k}}^{\prime}(1))}^2\\  \notag 
		& \geq C\left(\sum_{k\geq \widehat K_0}|a_k|^2 e^{2\Re(\lambda^p_k) T} \|(\xi_{\lambda^p_k}, \eta_{\lambda^p_k})\|^2_{H^{-s}_{per}\times L^2} + \sum_{\mod{k}\geq \widehat K_0} |b_k|^2 e^{2\Re (\lambda^h_k) T} \|(\xi_{\lambda^h_k}, \eta_{\lambda^h_k})\|^2_{ H^{-s}_{per}\times L^2}\right),
	\end{align}
	for some $C>0$. 

	Adding finitely many terms from $K_0\leq |k| < \widehat K_0$ in the above summations  (since all the observation terms are non-zero from $K_0$-th mode), we get  the following observability inequality 
	\begin{align}\label{Obser-eigen}
		\int_0^T  |\sigma(t,1)+v_x(t,1)|^2 dt \geq C \|(\sigma(0), v(0))\|^2_{H^{-s}_{per} \times L^2} , \quad \text{for } \ s>\frac{1}{2},
	\end{align}
for given  data $(\sigma_T, v_T)$ as \eqref{given-data}. 

Then, the density argument gives the required observability inequality \eqref{Obser-eigen} with any given adjoint state $(\sigma_T, v_T)\in \mathcal H^*$ (defined by \eqref{Space-Vstar}).  
	This is a necessary and sufficient  for  the null-controllability of  system \eqref{lcnse4} with given initial data $(\rho_0, u_0)\in \mathcal H$, where $\mathcal H$ is given by \eqref{Space-V}.

	This completes the proof of Theorem \ref{thm_control_veloct}.
	\end{proof}

\begin{remark}
	Some remarks are in order.
	\begin{itemize} 
\item	Proving null-controllability of the system \eqref{lcnse4} using the moments method does not provide a better result (w.r.t. the regularity of initial states) compare to the result obtained by Ingham-type inequality.

\item  On the other hand, the Ingham-type inequality \eqref{Ingham-ineq} does not give any conclusion regarding the controllability of the first case that is when a control is acting on the density part. In that case, the moments method really fits to get the controllability in a better space, namely in $H^s_{per}(0,1)\times L^2(0,1)$ for $s>1/2$.
\end{itemize}
\end{remark}

\smallskip

\section{Detailed spectral analysis of the adjoint operator}\label{Spectral-Details}

\smallskip 

We recall the eigenvalue problem \eqref{eigenfunction} from Section \ref{Sec-Spectral-short} and  for our convenience, we rewrite here below,
\begin{equation}\label{eigenvalue_problem}
\begin{aligned}
\xi^{\prime}(x)+\eta^{\prime}(x)&=\lambda\xi(x),  \quad x\in (0,1),\\
\eta^{\prime\prime}(x)+\eta^{\prime}(x)+\xi^{\prime}(x)&=\lambda\eta(x),   \quad x\in (0,1),\\
\xi(0)&=\xi(1),\\
\eta(0)=0,\ \ \eta(1) &=0.
\end{aligned}
\end{equation}

\medskip 

We divide the analysis into several steps. Let us begin by the following result.

\medskip 

\paragraph{\bf All non-trivial eigenvalues have negative real parts} This is the proof of point (iii)--\Cref{Prop_spectral}. 

Multiplying the first equation of \eqref{eigenvalue_problem} by $\overline{\xi}$, the second one by  $\overline{\eta}$ and then integrating, we  obtain
\begin{align*}
\int_{0}^{1}\overline{\xi(x)}\xi^{\prime}(x)dx+ \int_{0}^{1}\overline{\xi(x)}\eta^{\prime}(x)dx&=\lambda\int_{0}^{1}|\xi(x)|^2 dx\\
\int_{0}^{1}\overline{\eta(x)}\eta^{\prime\prime}(x)dx+\int_{0}^{1}\overline{\eta(x)}\eta^{\prime}(x)dx+\int_{0}^{1}\overline{\eta(x)}\xi^{\prime}(x)dx&=\lambda\int_{0}^{1}|\eta(x)|^2dx.
\end{align*}
Adding these two equations we get
\begin{align}\label{eq-6-1}
\int_{0}^{1}\overline{\xi(x)}\xi^{\prime}(x)dx+\int_{0}^{1}\overline{\eta(x)}\eta^{\prime}(x)dx+\int_{0}^{1}\overline{\xi(x)}\eta^{\prime}(x)dx+\int_{0}^{1}\overline{\eta(x)}\xi^\prime(x)dx \notag \\
\quad+\int_{0}^{1}\overline{\eta(x)}\eta^{\prime\prime}(x)dx=\lambda\int_{0}^{1}|\xi(x)|^2dx+\lambda\int_{0}^{1}|\eta(x)|^2dx.
\end{align}
Here, one can observe that 
\begin{equation}\label{eq-6-2}
\int_0^1 \overline{\xi(x)} \xi^\prime(x)dx = \frac{1}{2}\int_0^1\frac{d}{dx} |\xi(x)|^2dx + i \int_0^1 \Im (\overline{\xi(x)} \xi^\prime(x)) dx = i \int_0^1 \Im (\overline{\xi(x)} \xi^\prime(x)) dx,
\end{equation}
thanks to the boundary condition $\xi(0)=\xi(1)$. 
	
Similarly, we can obtain 
\begin{equation}\label{eq-6-3}
\int_0^1 \overline{\eta(x)} \eta^\prime(x)dx =   i \int_0^1 \Im (\overline{\eta(x)} \eta^\prime(x)) dx.
\end{equation}

Using the relations  \eqref{eq-6-2}, \eqref{eq-6-3} in \eqref{eq-6-1} and performing an integration by parts we deduce that 
\begin{align*}
i\int_0^1 \left( \Im (\overline{\xi(x)} \xi^\prime(x)) + \Im (\overline{\eta(x)} \eta^\prime(x))\right)dx+\int_{0}^{1}{\xi^\prime(x)}\overline{\eta(x)}dx-\int_{0}^{1}\overline{\xi^\prime(x)}{\eta(x)}dx-\int_{0}^{1}|\eta^{\prime}(x)|^2 dx \\
=\lambda\int_{0}^{1}|\xi(x)|^2dx+\lambda\int_{0}^{1}|\eta(x)|^2dx,
\end{align*}
from which it is clear that 
\begin{align}\label{real_part} 
\Re(\lambda)=-\frac{\|\eta^\prime\|^2_{L^2}}{\|\xi\|^2_{L^2} + \|\eta\|^2_{L^2}} <0,
\end{align}
since $\eta^\prime =0$ is not possible. If yes, then from the boundary condition $\eta(0)=\eta(1)=0$, we have $\eta=0$ and this yields that $\xi=c$, for some constant $c$,  which is possible if and only if $\lambda=0$.  Therefore, when $\lambda\neq 0$, then one has the condition \eqref{real_part}. 

	\begin{remark}\label{Remark-integral-zero}
		It can be easily seen that the first component $\xi$ satisfies 
	$\int_0^1 \xi =0$ provided $\lambda\neq 0$. 
	\end{remark}
	
\medskip 

\paragraph{\bf All eigenvalues  are geometrically simple: proof of \Cref{Prop_spectral}--point (iv)}
On contrary, let us assume that for any eigenvalue $\lambda$, there are two distinct eigenfunctions $\Phi_1:=(\xi_1,\eta_1)$ and $\Phi_2:=(\xi_2,\eta_2)$ of $A^*$. We prove that $\Phi_1$ and $\Phi_2$ are linearly dependent. 

Let be  $\theta_1, \theta_2\in\mathbb{C}\setminus\{0\}$ and consider the linear combination $\Phi:=\theta_1\Phi_1+\theta_2\Phi_2$. Then $\Phi:=(\xi,\eta)$ also satisfies the eigenvalue  problem \eqref{eigenvalue_problem}. We now choose $\theta_1,\theta_2$ in such a way that $\xi(0)=0$ (one particular choice is $\theta_1=-\frac{\theta_2\xi_2(0)}{\xi_1(0)}$). Then, in the same spirit of  \Cref{Prop_Approximate_Controllability}, we can conclude that $\Phi=0$. 

This ensures the assumption that each eigenvalue of $A^*$ has geometric multiplicity $1$.

\subsection{Determining the eigenvalues for large modulus} \label{Section-eigenvalues}
We write the set of equations \eqref{eigenvalue_problem} satisfied by $\xi$ and $\eta$ into a single equation of $\eta$ as obtained in   \eqref{etaeq}--\eqref{etaboun}, given by
\begin{subequations}
\begin{align}
\label{etaeqn}	&\eta^{\prime\prime\prime}(x)-\lambda\eta^{\prime\prime}(x)-2\lambda\eta^{\prime}(x)+\lambda^2\eta(x)=0  , \ \ \forall x \in (0,1),\\
\label{bc1}	&\eta(0)=0,\ \ \eta(1)=0,\ \ \eta^{\prime\prime}(0)=\eta^{\prime\prime}(1).
\end{align}
\end{subequations}
Then, the auxiliary equation associated to \eqref{etaeqn} is 
\begin{equation}\label{auxiliary1}
m^3-\lambda m^2-2\lambda m+\lambda^2=0.
\end{equation}
Introduce  $\mu=-\lambda\in \mb C$  and rewrite the auxiliary equation \eqref{auxiliary1}
\begin{equation}\label{auxiliary2}
m^3+\mu m^2+2\mu m+\mu^2=0.
\end{equation}

We consider $a=\mu$, $b=2\mu$, $c=\mu^2$ so that
the roots of the cubic polynomial \eqref{auxiliary2} are given by
\begin{equation}\label{form_roots_charac}
\begin{aligned}
m_1 &=-\frac{1}{3}\left(a+C+\frac{D_0}{C}\right),\\
m_2&=-\frac{1}{3}\left(a+\frac{(-1+i\sqrt{3})}{2}\,C+\frac{(-1-i\sqrt{3})}{2}\, \frac{D_0}{C}\right),\\
m_3&=-\frac{1}{3}\left(a+\frac{(-1-i\sqrt{3})}{2}\,C+\frac{(-1+i\sqrt{3})}{2}\,\frac{D_0}{C}\right),
\end{aligned}
\end{equation}
with
\begin{equation*}
C=\left(\frac{D_1+\sqrt{D_1^2-4D_0^3}}{2}\right)^{1/3},
\end{equation*}
where
\begin{equation*}
D_0 = a^2-3b = \mu^2 - 6\mu, \ \   D_1 = 2a^3-9ab+27c = 2\mu^3+ 9\mu^2.
\end{equation*}
Using the binomial expansion and approximating for large  $\mod{\mu}$, we obtain
\begin{align*}
	\sqrt{D_1^2-4D_0^3}&=\left[(2\mu^3+9\mu^2)^2-4(\mu^2-6\mu)^3\right]^{1/2}\\
	&=\left[(4\mu^6+36\mu^5+81\mu^4)-4(\mu^6-18\mu^5+108\mu^4-216\mu^3)\right]^{1/2}\\
	&=\left[108\mu^5-351\mu^4+864\mu^3\right]^{1/2}\\
	&=6\sqrt{3}\mu^{5/2}\left[1-\left(\frac{13}{4\mu}-\frac{8}{\mu^2}\right)\right]^{1/2}\\
	&=6\sqrt{3}\mu^{5/2}\left[1-\frac{1}{2}\left(\frac{13}{4\mu}-\frac{8}{\mu^2}\right)-\frac{1}{8}\left(\frac{169}{16\mu^2}-\frac{52}{\mu^3}+O(\mu^{-4})\right)\right]\\
	&=6\sqrt{3}\mu^{5/2}\left[1-\frac{13}{8\mu}+\frac{343}{128\mu^2}+O(\mu^{-3})\right].
\end{align*}
Therefore,
\begin{align*}
	C
	&=\left[\mu^3+\frac{9}{2}\mu^2+3\sqrt{3}\mu^{5/2}-\frac{39\sqrt{3}}{8}\mu^{3/2}+\frac{1029\sqrt{3}}{128}\mu^{1/2}+O(\mu^{-1/2})\right]^{1/3} \\
	&=\mu\left[1+\frac{9}{2\mu}+3\sqrt{3}\mu^{-1/2}-\frac{39\sqrt{3}}{8}\mu^{-3/2}+O(\mu^{-5/2})\right]^{1/3}\\
	&=\mu\bigg[1+\frac{1}{3}\left(\frac{9}{2\mu}+3\sqrt{3}\mu^{-1/2}-\frac{39\sqrt{3}}{8}\mu^{-3/2}+O(\mu^{-5/2})\right) \\ & \qquad \qquad \qquad  -\frac{1}{9}\left(\frac{9}{2\mu}+3\sqrt{3}\mu^{-1/2}-\frac{39\sqrt{3}}{8}\mu^{-3/2}+O(\mu^{-5/2})\right)^2 \\
	& \qquad \qquad \qquad +\frac{5}{81}\left(\frac{9}{2\mu}+3\sqrt{3}\mu^{-1/2}-\frac{39\sqrt{3}}{8}\mu^{-3/2}+O(\mu^{-5/2})\right)^3 +O(\mu^{-3})\bigg]\\
	&=\mu\left[1+\frac{1}{3}\left(\frac{9}{2\mu}+3\sqrt{3}\mu^{-1/2}-\frac{39\sqrt{3}}{8}\mu^{-3/2}+O(\mu^{-5/2})\right)-\frac{1}{9}\left(\frac{27}{\mu}+27\sqrt{3}\mu^{-3/2}+O(\mu^{-2})\right)\right.\\
	&\left.\quad\quad\quad+\frac{5}{81}\left(81\sqrt{3}\mu^{-3/2}+O(\mu^{-2})\right) +O(\mu^{-3})\right]\\
	&=\mu\left[1+\frac{3}{2\mu}+\sqrt{3}\mu^{-1/2}-\frac{13\sqrt{3}}{8}\mu^{-3/2}+O(\mu^{-5/2})-\frac{3}{\mu}-3\sqrt{3}\mu^{-3/2}+O(\mu^{-2})+5\sqrt{3}\mu^{-3/2}+O(\mu^{-3})\right]\\
	&=\mu\left[1+\sqrt{3}\mu^{-1/2}-\frac{3}{2\mu}+\frac{3\sqrt{3}}{8}\mu^{-3/2}+O(\mu^{-5/2})\right]\\
	&=\mu+\sqrt{3}\mu^{1/2}-\frac{3}{2}+\frac{3\sqrt{3}}{8}\mu^{-1/2}+O(\mu^{-3/2}).
\end{align*}
Similarly we have,
\begin{align*}
	\frac{D_0}{C}&=\frac{\mu^2-6\mu}{\mu\left[1+\sqrt{3}\mu^{-1/2}-\frac{3}{2\mu}+\frac{3\sqrt{3}}{8}\mu^{-3/2}+O(\mu^{-5/2})\right]}\\
	&=(\mu-6)\left[1+\left(\sqrt{3}\mu^{-1/2}-\frac{3}{2\mu}+\frac{3\sqrt{3}}{8}\mu^{-3/2}+O(\mu^{-5/2})\right)\right]^{-1}\\
	&=(\mu-6)\left[1-\sqrt{3}\mu^{-1/2}+\frac{3}{2\mu}-\frac{3\sqrt{3}}{8}\mu^{-3/2}+O(\mu^{-5/2})+\left(\sqrt{3}\mu^{-1/2}-\frac{3}{2\mu}+\frac{3\sqrt{3}}{8}\mu^{-3/2}+O(\mu^{-5/2})\right)^2\right.\\
	&\left.\qquad\quad\quad\quad\qquad-\left(\sqrt{3}\mu^{-1/2}-\frac{3}{2\mu}+\frac{3\sqrt{3}}{8}\mu^{-3/2}+O(\mu^{-5/2})\right)^3+O(\mu^{-2})\right]\\
	&=(\mu-6)\bigg[1-\sqrt{3}\mu^{-1/2}+\frac{3}{2\mu}-\frac{3\sqrt{3}}{8}\mu^{-3/2}+O(\mu^{-5/2}) \\
	& \qquad \qquad \qquad  +\left(\frac{3}{\mu}-3\sqrt{3}\mu^{-3/2}+O(\mu^{-5/2})\right)-3\sqrt{3}\mu^{-3/2}+O(\mu^{-2})\bigg]\\
	&=(\mu-6)\left[1-\sqrt{3}\mu^{-1/2}+\frac{9}{2\mu}-\frac{51\sqrt{3}}{8}\mu^{-3/2}+O(\mu^{-5/2})\right]\\
	&=\mu-\sqrt{3}\mu^{1/2}+\frac{9}{2}-\frac{51\sqrt{3}}{8}\mu^{-1/2}-6+6\sqrt{3}\mu^{-1/2}+O(\mu^{-1})\\
	&=\mu-\sqrt{3}\mu^{1/2}-\frac{3}{2}-\frac{3\sqrt{3}}{8}\mu^{-1/2}+O(\mu^{-1}).
\end{align*}
So, the characteristic roots are (recall \eqref{form_roots_charac})
\begin{align*}
	m_1
	&=-\frac{1}{3}\bigg[\mu+\left(\mu+\sqrt{3}\mu^{1/2}-\frac{3}{2}+\frac{3\sqrt{3}}{8}\mu^{-1/2}+O(\mu^{-3/2})\right) \\
	&\qquad \qquad \qquad +\left(\mu-\sqrt{3}\mu^{1/2}-\frac{3}{2}-\frac{3\sqrt{3}}{8}\mu^{-1/2}+O(\mu^{-1})\right)\bigg]\\
	&=-\frac{1}{3}\left(3\mu-3+O(\mu^{-1})\right)\\
	&=-\mu+1+O(\mu^{-1}),
\end{align*}
\begin{align*}
	m_2
	&=-\frac{1}{3}\left[\mu+\frac{-1+i\sqrt{3}}{2}\left(\mu+\sqrt{3}\mu^{1/2}-\frac{3}{2}+O(\mu^{-1/2})\right)+\frac{-1-i\sqrt{3}}{2}\left(\mu-\sqrt{3}\mu^{1/2}-\frac{3}{2}+O(\mu^{-1/2})\right)\right]\\
	&=-\frac{1}{3}\left[\frac{3}{2}+3i\mu^{1/2}+O(\mu^{-1/2})\right]\\
	&=-\frac{1}{2}-i\mu^{1/2}+O(\mu^{-1/2}),
\end{align*}
\begin{align*}
	m_3
	&=-\frac{1}{3}\left[\mu+\frac{-1-i\sqrt{3}}{2}\left(\mu+\sqrt{3}\mu^{1/2}-\frac{3}{2}+O(\mu^{-1/2})\right)+\frac{-1+i\sqrt{3}}{2}\left(\mu-\sqrt{3}\mu^{1/2}-\frac{3}{2}+O(\mu^{-1/2})\right)\right]\\
	&=-\frac{1}{3}\left[\frac{3}{2}-3i\mu^{1/2}+O(\mu^{-1/2})\right]\\
	&=-\frac{1}{2}+i\mu^{1/2}+O(\mu^{-1/2}).
\end{align*}
Together, we write
\begin{align}\label{charac_roots}
\begin{cases}
m_1=-\mu+1+O(\mu^{-1}),\\ 
m_2=-\frac{1}{2}-i\mu^{1/2}+O(\mu^{-1/2}),\\ 
m_3=-\frac{1}{2}+i\mu^{1/2}+O(\mu^{-1/2}),
\end{cases}
\end{align}
with $\mu=-\lambda$ as mentioned earlier. Since $m_1, m_2$ and $m_3$  are distinct at least for large modulus of $\mu$, we can write the general solution to the equation \eqref{etaeqn} as
\begin{equation}\label{func-eta-1}
\eta(x)=C_1 e^{m_1 x} +  C_2 e^{m_2 x} + C_3 e^{m_3 x}, \ \ \ x \in (0,1),
\end{equation}
for some constants $C_1, C_2, C_3\in \mathbb C$.

Using the boundary conditions \eqref{bc1}, we get a system of linear equations in $C_1$, $C_2$ and $C_3$, given by
\begin{equation}\label{values_D}
\begin{aligned}
&C_1+C_2+C_3=0,\\
&C_1 e^{m_1}+C_2 e^{m_2}+C_3 e^{m_3}=0,\\
&C_1 m_1^2\left(1-e^{m_1}\right)+C_2 m_2^2\left(1-e^{m_2}\right)+C_3 m_3^2\left(1-e^{m_3}\right)=0.
\end{aligned}
\end{equation}
These system of equations \eqref{values_D} has a nontrivial solution if and only if
\begin{equation*}
\det
\begin{pmatrix}
1&1&1\\[4pt]e^{m_1}&e^{m_2}&e^{m_3}\\[4pt]m_1^2\left(1-e^{m_1}\right)&m_2^2\left(1-e^{m_2}\right) &m_3^2\left(1-e^{m_3}\right)
\end{pmatrix}
=0.
\end{equation*}
Expanding the determinant,  we obtain
\begin{align}\label{determinant_auxiliary}
m_1^2\left(1-e^{m_1}\right)\left(e^{m_3}-e^{m_2}\right)+m_2^2\left(1-e^{m_2}\right)\left(e^{m_1}-e^{m_3}\right)+m_3^2\left(1-e^{m_3}\right)\left(e^{m_2}-e^{m_1}\right)=0. 
\end{align}
We shall now compute the determinant term by term. 

$\bullet$ Plugging the values of $m_1$, $m_2$ and $m_3$ as given in \eqref{charac_roots}, we obtain
\begin{align}\label{subestimate-1}
	&m_1^2\left(1-e^{m_1}\right)\left(e^{m_3}-e^{m_2}\right)\\ \notag 
	&=\left(-\mu+1+O(\mu^{-1/2})\right)^2\left(1-e^{-\mu+1+O(\mu^{-1/2})}\right)\left(e^{-1/2+i\mu^{1/2}+O(\mu^{-1/2})}-e^{-1/2-i\mu^{1/2}+O(\mu^{-1/2})}\right)\\ \notag 
		&=\left(\mu^2-2\mu+O(\mu^{1/2})\right)\left(1-e^{-\mu+1+O(\mu^{-1})}\right)\left(e^{-1/2+O(\mu^{-1/2})}\left(\cos(\mu^{1/2})+i\sin(\mu^{1/2})\right)\right.\\\notag 
		&\left.\quad-e^{-1/2+O(\mu^{-1/2})}\left(\cos(\mu^{1/2})-i\sin(\mu^{1/2})\right)\right)\\ \notag 
	&=\left(\mu^2-2\mu+O(\mu^{1/2})\right)\left(1-e^{-\mu+1+O(\mu^{-1})}\right)\left[O(\mu^{-1/2}) e^{-1/2+O(\mu^{-1/2})}\cos(\mu^{1/2})\right.\\\notag 
	&\left.\qquad+i(2+O(\mu^{-\frac{1}{2}}))e^{-1/2+O(\mu^{-1/2})}\sin(\mu^{1/2})\right],
\end{align}
where we have used the facts that 
	\begin{equation*}
		e^{-1/2+O(\mu^{-1/2})}-e^{-1/2+O(\mu^{-1/2})}=e^{-1/2+O(\mu^{-1/2})}\Big(1-e^{O(\mu^{-\frac{1}{2}})}\Big)= e^{-1/2+O(\mu^{-1/2})} \times O(\mu^{-\frac{1}{2}}),
	\end{equation*}
	and 
	\begin{equation*}
		e^{-1/2+O(\mu^{-1/2})}+e^{-1/2+O(\mu^{-1/2})}=e^{-1/2+O(\mu^{-1/2})}\Big(1+e^{O(\mu^{-\frac{1}{2}})}\Big)=e^{-1/2+O(\mu^{-1/2})}\times(2+O(\mu^{-\frac{1}{2}})).
	\end{equation*}
$\bullet$ We also compute 
\begin{align*}
	&m_2^2\left(1-e^{m_2}\right)\left(e^{m_1}-e^{m_3}\right)\\
		&=\left(-\frac{1}{2}-i\mu^{1/2}+O(\mu^{-\frac{1}{2}})\right)^2\left(1-e^{-1/2-i\mu^{1/2}+O(\mu^{-1/2})}\right)\left(e^{-\mu+1+O(\mu^{-1})}-e^{-\frac{1}{2}+i\mu^{\frac{1}{2}}+O(\mu^{-\frac{1}{2}})}\right)\\
		&=\left(-\mu+i\mu^{\frac{1}{2}}+O(1)\right)\left(e^{-\mu+1+O(\mu^{-1})}+e^{-1+O(\mu^{-\frac{1}{2}})}-e^{-\mu+\frac{1}{2}-i\mu^{\frac{1}{2}}+O(\mu^{-\frac{1}{2}})}-e^{-\frac{1}{2}+i\mu^{\frac{1}{2}}+O(\mu^{-\frac{1}{2}})}\right)\\
	&=\left(-\mu+i\mu^{\frac{1}{2}}+O(1)\right)\left[e^{-\mu+1+O(\mu^{-1})}+e^{-1+O(\mu^{-\frac{1}{2}})}-e^{-\mu+\frac{1}{2}+O(\mu^{-\frac{1}{2}})}\left(\cos(\mu^{\frac{1}{2}})-i\sin(\mu^{\frac{1}{2}})\right)\right.\\
	&\left.\hspace{4cm}-e^{-\frac{1}{2}+O(\mu^{-\frac{1}{2}})}\left(\cos(\mu^{\frac{1}{2}})+i\sin(\mu^{\frac{1}{2}})\right)\right].
\end{align*}
$\bullet$ Finally, we have
\begin{align*}
	&m_3^2\left(1-e^{m_3}\right)\left(e^{m_2}-e^{m_1}\right)\\
		&=\left(-\frac{1}{2}+i\mu^{1/2}+O(\mu^{-\frac{1}{2}})\right)^2\left(1-e^{-\frac{1}{2}+i\mu^{\frac{1}{2}}+O(\mu^{-\frac{1}{2}})}\right)\times\left(e^{-\frac{1}{2}-i\mu^{\frac{1}{2}}+O(\mu^{-\frac{1}{2}})}-e^{-\mu+1+O(\mu^{-1})}\right)\\
	&=\left(-\mu-i\mu^{\frac{1}{2}}+O(1)\right)\left[-e^{-\mu+1+O(\mu^{-1})}-e^{-1+O(\mu^{-\frac{1}{2}})} +e^{-\mu+\frac{1}{2}+O(\mu^{-\frac{1}{2}})}\left(\cos(\mu^{\frac{1}{2}})+i\sin(\mu^{\frac{1}{2}})\right)\right.\\
	&\left.\hspace{4cm}+e^{-\frac{1}{2}+O(\mu^{-\frac{1}{2}})}\left(\cos(\mu^{\frac{1}{2}})-i\sin(\mu^{\frac{1}{2}})\right)\right].
\end{align*}

	$\bullet$ We add now the last two terms, in what follows
	\begin{align}\label{subestimate-2}
		&m_2^2\left(1-e^{m_2}\right)\left(e^{m_1}-e^{m_3}\right)+m_3^2\left(1-e^{m_3}\right)\left(e^{m_2}-e^{m_1}\right)\\ \notag 
		&=-\mu \left(e^{-\mu+1 + O(\mu^{-1})} + e^{-1+ O(\mu^{-\frac{1}{2}})} - e^{-\mu+1 + O(\mu^{-\frac{1}{2}})} - e^{-1+ O(\mu^{-\frac{1}{2}})} \right) 
		+ i \mu^{\frac{1}{2}} \big(2+ O(\mu^{-\frac{1}{2}})\big)  e^{-\mu+1 + O(\mu^{-1})} \\ \notag 
		& \ +i\mu^{\frac{1}{2}}(2+ O(\mu^{-\frac{1}{2}})) e^{-1+ O(\mu^{-\frac{1}{2}})} 
		 +  \cos \mu^{\frac{1}{2}} \bigg[  \big(-\mu - i\mu^{\frac{1}{2}} + O(1)\big)  \Big(e^{-\mu+\frac{1}{2} + O(\mu^{-\frac{1}{2}})} + e^{-\frac{1}{2}  + O(\mu^{-\frac{1}{2}})} \Big)  \\ \notag 
		&\qquad \qquad - \big(-\mu + i\mu^{\frac{1}{2}} +O(1)\big) \Big(e^{-\mu+\frac{1}{2} + O(\mu^{-\frac{1}{2}})} + e^{-\frac{1}{2}  + O(\mu^{-\frac{1}{2}})} \Big) \bigg] \\   \notag 
		&\qquad +i\sin \mu^{\frac{1}{2}} \bigg[ \big(-\mu- i\mu^{\frac{1}{2}} + O(1) \big) \Big(e^{-\mu+\frac{1}{2} + O(\mu^{-\frac{1}{2}})} - e^{-\frac{1}{2}  + O(\mu^{-\frac{1}{2}})} \Big)  \\ \notag 
		& \qquad \qquad +  \big(-\mu+ i\mu^{\frac{1}{2}}  +O(1) \big) \Big(e^{-\mu+\frac{1}{2} + O(\mu^{-\frac{1}{2}})}  -e^{-\frac{1}{2}  + O(\mu^{-\frac{1}{2}})}   \Big)   \bigg]\\ \notag 
		&= -\mu  O(\mu^{-\frac{1}{2}}) e^{-\mu+1 + O(\mu^{-1})} - \mu O(\mu^{-\frac{1}{2}}) e^{-1+O(\mu^{-\frac{1}{2}})}    
		+\big(2i\mu^{\frac{1}{2}} +O(1)\big) e^{-\mu+1 + O(\mu^{-1})} \\ \notag 
		& +\big(2i\mu^{\frac{1}{2}} +O(1)\big) e^{-1+O(\mu^{-\frac{1}{2}})}
		+\cos \mu^{\frac{1}{2}} \bigg[\big(\mu+O(1)\big)  O(\mu^{-\frac{1}{2}})    e^{-\mu+\frac{1}{2} + O(\mu^{-\frac{1}{2}})} \\ \notag  
		& \ \ +  \big(\mu+O(1)\big) O(\mu^{-\frac{1}{2}})  e^{-\frac{1}{2} + O(\mu^{-\frac{1}{2}})} 
		- i\mu^{\frac{1}{2}}( 2+O(\mu^{-\frac{1}{2}}) ) e^{-\mu+\frac{1}{2} +O(\mu^{-\frac{1}{2}}) } - i\mu^{\frac{1}{2}} ( 2+O(\mu^{-\frac{1}{2}}) ) e^{-\frac{1}{2} + O(\mu^{-\frac{1}{2}})}   \bigg] \\ \notag 
		&\quad +  i\sin \mu^{\frac{1}{2}} \bigg[\big(-2\mu +O(\mu^{\frac{1}{2}})\big)   e^{-\mu+\frac{1}{2} +O(\mu^{-\frac{1}{2}}) } + i\mu^{\frac{1}{2}} O(\mu^{-\frac{1}{2}}) e^{-\mu+\frac{1}{2} +O(\mu^{-\frac{1}{2}}) }  \\ \notag 
		&\quad\quad\quad\quad\quad\quad + \big(2\mu+ O(\mu^{\frac{1}{2}})\big) e^{-\frac{1}{2} + O(\mu^{-\frac{1}{2}})} + i\mu^{\frac{1}{2}} O(\mu^{-\frac{1}{2}}) e^{-\frac{1}{2} + O(\mu^{-\frac{1}{2}})} \bigg]   .
	\end{align}
We get after adding \eqref{subestimate-1} and \eqref{subestimate-2},
\begin{align*}
	&m_1^2\left(1-e^{m_1}\right)\left(e^{m_3}-e^{m_2}\right)+m_2^2\left(1-e^{m_2}\right)\left(e^{m_1}-e^{m_3}\right)+m_3^2\left(1-e^{m_3}\right)\left(e^{m_2}-e^{m_1}\right)\\
	=& \cos \mu^{\frac{1}{2}} (\mu^2-2\mu+O(\mu^{\frac{1}{2}})) O(\mu^{-\frac{1}{2}}) \left(e^{-\frac{1}{2}+O(\mu^{-\frac{1}{2}}) }- e^{-\mu + \frac{1}{2} + O(\mu^{-\frac{1}{2}})} \right) \\
	+& i\sin \mu^{\frac{1}{2}} (2+O(\mu^{-\frac{1}{2}})) (\mu^2-2\mu +O(\mu^{-\frac{1}{2}})) \left(e^{-\frac{1}{2}+O(\mu^{-\frac{1}{2}}) }- e^{-\mu + \frac{1}{2} + O(\mu^{-\frac{1}{2}})} \right) \\
	-&\mu  O(\mu^{-\frac{1}{2}}) e^{-\mu+1 + O(\mu^{-1})} - \mu O(\mu^{-\frac{1}{2}}) e^{-1+O(\mu^{-\frac{1}{2}})}   \\
	+&\big(2i\mu^{\frac{1}{2}} +O(1)\big) e^{-\mu+1 + O(\mu^{-1})} +\big(2i\mu^{\frac{1}{2}} +O(1)\big) e^{-1+O(\mu^{-\frac{1}{2}})}\\
	+&\cos \mu^{\frac{1}{2}} \bigg[\big(\mu+O(1)\big)  O(\mu^{-\frac{1}{2}})    e^{-\mu+\frac{1}{2} + O(\mu^{-\frac{1}{2}})} +  \big(\mu+O(1)\big) O(\mu^{-\frac{1}{2}})  e^{-\frac{1}{2} + O(\mu^{-\frac{1}{2}})} \\
	- & i\mu^{\frac{1}{2}}( 2+O(\mu^{-\frac{1}{2}}) ) e^{-\mu+\frac{1}{2} +O(\mu^{-\frac{1}{2}}) } - i\mu^{\frac{1}{2}} ( 2+O(\mu^{-\frac{1}{2}}) ) e^{-\frac{1}{2} + O(\mu^{-\frac{1}{2}})}   \bigg] \\
	& \ \ +  i\sin \mu^{\frac{1}{2}} \bigg[\big(-2\mu +O(\mu^{\frac{1}{2}})\big)   e^{-\mu+\frac{1}{2} +O(\mu^{-\frac{1}{2}}) } + i\mu^{\frac{1}{2}} O(\mu^{-\frac{1}{2}}) e^{-\mu+\frac{1}{2} +O(\mu^{-\frac{1}{2}}) }  \\ 
	& \ \ + \big(2\mu+ O(\mu^{\frac{1}{2}})\big) e^{-\frac{1}{2} + O(\mu^{-\frac{1}{2}})} + i\mu^{\frac{1}{2}} O(\mu^{-\frac{1}{2}}) e^{-\frac{1}{2} + O(\mu^{-\frac{1}{2}})} \bigg]  \\
	&=- \mu  O(\mu^{-\frac{1}{2}}) e^{-\mu+1 + O(\mu^{-1})} - \mu O(\mu^{-\frac{1}{2}}) e^{-1+O(\mu^{-\frac{1}{2}})} \\
	&\quad+\big(2i\mu^{\frac{1}{2}} +O(1)\big) e^{-\mu+1 + O(\mu^{-1})}  +\big(2i\mu^{\frac{1}{2}} +O(1)\big) e^{-1+O(\mu^{-\frac{1}{2}})} \\
	&\quad+ \cos \mu^{\frac{1}{2}} \bigg[ \Big(-\mu^2O(\mu^{-\frac{1}{2}})  +3\mu O(\mu^{-\frac{1}{2}})  - 2i \mu^{\frac{1}{2}}+ O(1)    \Big) e^{-\mu+\frac{1}{2} +O(\mu^{-\frac{1}{2}})} \\ 
	& \quad\quad\quad\quad\quad\quad + \Big(\mu^2 O(\mu^{-\frac{1}{2}})  -\mu  O(\mu^{-\frac{1}{2}}) - 2i \mu^{\frac{1}{2}}+ O(1) \Big) e^{-\frac{1}{2} +O(\mu^{-\frac{1}{2}})}     \bigg]\\
	&\quad+i\sin \mu^{\frac{1}{2}} \bigg[ \Big(-2\mu^2  +  O(\mu^{\frac{3}{2}})\Big) e^{-\mu+\frac{1}{2}+O(\mu^{-\frac{1}{2}})}+ \Big(2\mu^2 +O(\mu^{\frac{3}{2}}) \Big)e^{-\frac{1}{2} + O(\mu^{-\frac{1}{2}})} \bigg].
\end{align*}

Now, replacing the above quantity in the equation \eqref{determinant_auxiliary},  and then dividing it by $\mu^2$ (since $\mu\neq 0$),  we obtain the equation 
\begin{align}\label{main_eigen_equ}
F(\mu)=0,
\end{align}
where 
\begin{multline*}
F(\mu) = - 2\sin \mu^{\frac{1}{2}} \left( e^{-\mu+1} -1\right) + O(\mu^{-\frac{1}{2}}) \sin \mu^{\frac{1}{2}}\, e^{-\mu+1  + O(\mu^{-\frac{1}{2}})}  \\
+ O(\mu^{-\frac{1}{2}}) \sin \mu^{\frac{1}{2}}\, e^{ O(\mu^{-\frac{1}{2}})}+ \cos \mu^{\frac{1}{2}} \bigg[O(\mu^{-\frac{1}{2}}) e^{-\mu+1 + O(\mu^{-\frac{1}{2}}) } +O(\mu^{-\frac{1}{2}}) e^{ O(\mu^{-\frac{1}{2}}) }  \bigg]     \\ 
+ O(\mu^{-\frac{3}{2}}) e^{-\mu+\frac 32 + O(\mu^{-\frac{1}{2}})} + O(\mu^{-\frac{3}{2}}) e^{-\frac 12 +O(\mu^{-\frac{1}{2}})}.
\end{multline*}

\paragraph{\bf Application of Rouche's theorem}
Let $G$ be a function of $\mu$, defined as
\begin{equation*}
G(\mu)=-2\sin(\mu^{\frac{1}{2}})\left(e^{-\mu+1}-1\right).
\end{equation*}
Then
\begin{align*}
{F}(\mu)-G(\mu)=& \underbrace{\sin(\mu^{\frac{1}{2}})\left(O(\mu^{-\frac{1}{2}})e^{-\mu+1+O(\mu^{-\frac{1}{2}})}+O(\mu^{-\frac{1}{2}})e^{O(\mu^{-\frac{1}{2}})}\right)  }_{I_1}\\
& +  \underbrace{\cos(\mu^{\frac{1}{2}})\left(O(\mu^{-\frac{1}{2}})e^{-\mu+1+O(\mu^{-\frac{1}{2}})}+O(\mu^{-\frac{1}{2}})e^{O(\mu^{-\frac{1}{2}})}\right) }_{I_2}\\
& +\underbrace{O(\mu^{-\frac{3}{2}})e^{-\mu+\frac{3}{2}+O(\mu^{-\frac{1}{2}})}+O(\mu^{-\frac{3}{2}})e^{-\frac{1}{2}+O(\mu^{-\frac{1}{2}})}}_{I_3}\\
:=& I_1+I_2+I_3.
\end{align*}
Now since the function $G$ has two branches of zeros, we will calculate them separately and in each case, we use the Rouche's theorem to talk  about the zeros of the function $F$. 

\smallskip 

{\bf Case 1.} We first observe that  $\mu=k^2\pi^2$ is a zero of $G$ for each  $k\in\mathbb{N}^*$ and consider the following region in the complex plane
\begin{equation}
\mathcal{R}_k=\left\{z=x+iy\in\mathbb{C}\ : \ k\pi-\frac{\pi}{2}\leq x\leq k\pi+\frac{\pi}{2},\ \ -\frac{\pi}{2}\leq y\leq \frac{\pi}{2}\right\} , \quad \text{for } k\in \mb N^*. 
\end{equation}
Our goal is to prove that $\mod{{F}(\mu)-G(\mu)}<\mod{G(\mu)}$ on $\partial\mathcal{R}_k$. It is sufficient to prove that
\begin{equation}
\mod{\frac{{F}(\mu)-G(\mu)}{G(\mu)}}\to 0 \ \ \text{for }\mu\in\partial\mathcal{R}_k \text{ such that } \Re(\mu)\to+\infty.
\end{equation}
To avoid difficulties in notations, we denote $w=\mu^{\frac{1}{2}}$ and without loss of generality, we simply write $I_1$, $I_2$ and $I_3$  as the functions $w$.  Note that
\begin{align*}
\mod{\frac{I_1(w)}{G(w)}}&=\mod{\frac{O(w^{-1})e^{-w^2+1+O(w^{-1})}+O(w^{-1})e^{O(w^{-1})}}{e^{-w^2+1}-1}}\leq\frac{C}{\mod{w}}\frac{\mod{e^{-w^2+1}}+1}{\mod{e^{-w^2+1}-1}},
\end{align*}
and since $\frac{\mod{e^{-w^2+1}}+1}{\mod{e^{-w^2+1}-1}}$ is bounded when $\Re(w)\to+\infty$, therefore
\begin{equation*}
\mod{\frac{I_1(w)}{G(w)}}\to0,\ \ \text{as } \Re(w)\to+\infty.
\end{equation*}
We now compute
\begin{align*}
\mod{\frac{I_2(w)}{G(w)}}&=\mod{\frac{\cos(w)}{\sin(w)}}\frac{\mod{O(w^{-1})e^{-w^2+1+O(w^{-1})}+O(w^{-1})e^{O(w^{-1})}}}{\mod{e^{-w^2+1}-1}}\leq\frac{C}{\mod{w}}\mod{\frac{\cos(w)}{\sin(w)}}\frac{\mod{e^{-w^2+1}}+1}{\mod{e^{-w^2+1}-1}},
\end{align*}
which yields
\begin{equation*}
\mod{\frac{I_2(w)}{G(w)}}\to 0,\ \ \text{for }w\in\partial\mathcal{R}_k \text{ such that } \Re(w)\to+\infty,
\end{equation*}
because of the fact that $\mod{\frac{\cos(w)}{\sin(w)}}$ is bounded on $\partial\mathcal{R}_k$. We can say similarly for the third term that
\begin{equation*}
\mod{\frac{I_3(w)}{G(w)}}\to 0,\ \ \text{for }w\in\partial\mathcal{R}_k \text{ such that } \Re(w)\to+\infty,
\end{equation*}
as we have
\begin{equation*}
\mod{\frac{I_3(w)}{G(w)}}\leq\frac{C}{\mod{w}^3}\mod{\frac{1}{\sin(w)}}\frac{\mod{e^{-w^2+\frac{3}{2}}}+1}{\mod{e^{-w^2+1}-1}}.
\end{equation*}

\smallskip 

{\bf Case 2.} When $\sin(\mu^{\frac 12}) \neq 0$, $G(\mu)=0$  gives  $e^{-\mu+1}-1=0$, that is $\mu=1+2ik\pi$ for $k\in\mathbb{Z}$. In this case, we  consider the following region in the complex plane
\begin{equation}
\mathcal{S}_k=\left\{z=x+iy\in\mathbb{C}\ : \ 1-\frac{\pi}{2}\leq x\leq 1+\frac{\pi}{2},\ \ 2k\pi-\frac{\pi}{2}\leq y\leq 2k\pi+\frac{\pi}{2} \right\}.
\end{equation}
We need to  show that $\mod{{F}(\mu)-G(\mu)}<\mod{G(\mu)}$ on $\partial\mathcal{S}_k$. In particular, we prove that
\begin{equation*}
\mod{\frac{{F}(\mu)-G(\mu)}{G(\mu)}}\to 0 \ \ \text{for }\mu\in\partial\mathcal{S}_k \text{ such that } \Im(\mu)\to+\infty.
\end{equation*}
We compute
\begin{align*}
\mod{\frac{I_1(\mu)}{G(\mu)}}&=\frac{1}{\mod{\sin(\mu^{\frac{1}{2}})}}\mod{\frac{O(\mu^{-\frac{1}{2}})e^{-\mu+1+O(\mu^{-\frac{1}{2}})}+O(\mu^{-1})e^{O(\mu^{-\frac{1}{2}})}}{e^{-\mu+1}-1}}\leq\frac{C}{\mod{\mu}^{\frac{1}{2}}}\frac{1}{\mod{\sin(\mu^{\frac{1}{2}})}}\frac{\mod{e^{-\mu+1}}+1}{\mod{e^{-\mu+1}-1}},
\end{align*}
\begin{align*}
\mod{\frac{I_2(\mu)}{G(\mu)}}&=\mod{\frac{\cos(\mu^{\frac{1}{2}})}{\sin(\mu^{\frac{1}{2}})}}\frac{\mod{O(\mu^{-\frac{1}{2}})e^{-\mu+1+O(\mu^{-\frac{1}{2}})}+O(\mu^{-\frac{1}{2}})e^{O(\mu^{-\frac{1}{2}})}}}{\mod{e^{-\mu+1}-1}}\leq\frac{C}{\mod{\mu}^{\frac{1}{2}}}\mod{\frac{\cos(\mu^{\frac{1}{2}})}{\sin(\mu^{\frac{1}{2}})}}\frac{\mod{e^{-\mu+1}}+1}{\mod{e^{-\mu+1}-1}},  
\end{align*}
and 
\begin{equation*}
\mod{\frac{I_3(\mu)}{G(\mu)}}\leq\frac{C}{\mod{\mu}^{\frac{3}{2}}}\frac{1}{\mod{\sin(\mu^{\frac{1}{2}})}}\frac{\mod{e^{-\mu+\frac{3}{2}}}+1}{\mod{e^{-\mu+1}-1}}.
\end{equation*}
On $\partial\mathcal{S}_k$, $\mod{\cos(\mu^{\frac{1}{2}})}$ and $\mod{\sin(\mu^{\frac{1}{2}})}$ has both lower and upper bounds and $\frac{\mod{e^{-\mu+1}}+1}{\mod{e^{-\mu+1}-1}}, \frac{\mod{e^{-\mu+\frac{3}{2}}}+1}{\mod{e^{-\mu+1}-1}}$ are bounded. Therefore, for each $j=1,2,3$, we have 
\begin{equation*}
\mod{\frac{I_j(\mu)}{G(\mu)}}\to 0,\ \ \text{for }\mu\in\partial\mathcal{S}_k \text{ such that } \Im(\mu)\to+\infty.
\end{equation*}
Thus, combining the above two cases, we conclude that there exists some $k_0\in \mb N^*$ sufficiently large, such that 
\begin{equation}
\mod{{F}(\mu)-G(\mu)}<\mod{G(\mu)},\ \ \forall\mu\in\partial\mathcal{R}_k\cup\partial\mathcal{S}_k \text{ and for large }k. 
\end{equation}

\medskip 

Since any two regions $\mathcal{R}_k$ and $\mathcal{R}_l$ are disjoint for $k\neq l$ and in each region $\mathcal{R}_k$, there is exactly one root of $G$ (more precisely, the square-root of $\mu$), the same is true for the function $F$, thanks to the Rouche’s theorem. Similar phenomenon holds true in the region $\mathcal{S}_k$. To be more precise, we have the following. 

{\bf On the region $\mathcal{R}_k$: parabolic part.} For $k\geq k_0$, the function $F$ has a unique root in $\mathcal{R}_k$ of the form
\begin{equation*}
\mu^{\frac{1}{2}}_k=(k\pi+c_k)+id_k,
\end{equation*}
with $\mod{c_k},\ \mod{d_k}\leq\frac{\pi}{2}$. Therefore,  the first set of eigenvalues are given by
\begin{equation}\label{eigen-parabolic}
\lambda^p_k:=-\mu_k:=-k^2\pi^2-2c_kk\pi-2id_kk\pi-(c_k^2-d_k^2)-2ic_kd_k, \quad \forall k\geq k_0.
\end{equation}

{\bf On the region $\mathcal{S}_k$: hyperbolic part.} On the other hand, for $\mod{k}\geq k_0$, the function $F$ has a unique root in $\mathcal{S}_k$ of the form
\begin{equation*}
\tilde\mu_k=1+\alpha_{1,k}+i(2k\pi+\alpha_{2,k}),
\end{equation*}
with $\mod{\alpha_{1,k}},\ \mod{\alpha_{2,k}}\leq\frac{\pi}{2}$.

Therefore,  the second set of eigenvalues are given by
\begin{equation}\label{eigen-hyperbolic}
\lambda^h_k:=-\tilde \mu_k:=-1-\alpha_{1,k}-i(2k\pi+\alpha_{2,k}), \quad \forall |k| \geq k_0.
\end{equation}
This indeed proves the results \eqref{eigenvalues-lambda} and \eqref{eigenvalues-gamma} of our Lemma \ref{lemma_eigenvalue_eigenfunc}. 

\smallskip 

\subsection{Computing  the eigenfunctions for large frequencies}\label{subsection-eigenfunc}
From the set of equations \eqref{values_D}, one can obtain the following values of $C_1, C_2, C_3$
\begin{align}\label{values_D_auxi}
\begin{cases}
C_1= e^{m_2}-e^{m_3}, \\
C_2= e^{m_3}-e^{ m_1 }, \\
C_3= e^{m_1} - e^{m_2}.
\end{cases}
\end{align}
Note that $C_1,C_2$ and $C_3$ cannot be simultaneously zero. Once we have that, one can easily obtain the function $\eta(x)$, defined by \eqref{func-eta-1}, 
\begin{align}\label{func_eta}
\eta(x)= (e^{m_2}-e^{m_3}) e^{m_1 x} + (e^{m_3}-e^{ m_1 }) e^{m_2 x} + ( e^{m_1} - e^{m_2}) e^{m_3 x}, \ \ \forall x \in (0,1).
\end{align}
We now compute  the first and second derivatives of $\eta$ which will let us obtain the other component $\xi$ of the set of equations \eqref{eigenvalue_problem}. We see
\begin{align*}
\eta^\prime(x) = m_1(e^{m_2}-e^{m_3}) e^{m_1 x} + m_2 (e^{m_3}-e^{ m_1 }) e^{m_2 x}+ m_3 ( e^{m_1} - e^{m_2}) e^{m_3 x},
\end{align*}
\begin{align*}
\eta^{\prime \prime}(x) = m^2_1(e^{m_2}-e^{m_3 }) e^{m_1 x} + m^2_2 (e^{m_3}-e^{ m_1 }) e^{m_2 x}+ m_3^2 ( e^{m_1} - e^{m_2}) e^{m_3 x}.
\end{align*}
Now, from equation \eqref{eigenvalue_problem}, one can obtain
\begin{equation*}
\eta^{\prime\prime}(x)+\lambda\xi(x)=\lambda\eta(x),
\end{equation*}
and therefore, (writing  $\mu=-\lambda$) 
\begin{align}\label{expression_xi}
&\xi(x) \\
=& \frac{\eta^{\prime\prime} (x)+\mu\eta(x)}{\mu} \notag \\ \notag
=& \Big(\frac{m_1^2 +\mu}{\mu} \Big) (e^{m_2}-e^{m_3 }) e^{m_1 x} + \Big(\frac{m^2_2 + \mu}{\mu}\Big)  (e^{m_3}-e^{ m_1 }) e^{m_2 x} + \Big(\frac{m_3^2 + \mu}{\mu}\Big) ( e^{m_1} - e^{m_2}) e^{m_3 x}.
\end{align}

\paragraph{\bf Set of eigenfunctions associated with ${\lambda^p_k}$} 
For the set of eigenvalues $\{\lambda^p_k\}_{k\geq k_0}$ \eqref{eigen-parabolic}, we denote the eigenfunctions by $\Phi_{\lambda^p_k}$, $\forall k\geq k_0$, where we shall use the notation
\begin{align}\label{eigenfunc-1st}
\Phi_{\lambda^p_k}(x) = \begin{pmatrix}\xi_{\lambda^p_k}(x) \\ \eta_{\lambda^p_k}(x)\end{pmatrix}, \quad \forall k\geq k_0. 
\end{align}

\smallskip 

{\bf Computing ${\eta_{\lambda^p_k}}$.} 
Let us recall the values of $m_1$, $m_2$ and $m_3$ from \eqref{charac_roots} and observe that $O(\mu_k^{-1/2})=O(k^{-1})$. In what follows, we  deduce their explicit expressions (for large frequencies of $k$)  
\begin{align}\label{charac_roots_asymp}
\begin{cases} 
m_{1} = -k^2\pi^2 - 2c_k k\pi - 2id_k k\pi + O(1), \quad \forall k\geq k_0 \text{ large enough} , \\
m_{2}=-\frac{1}{2}+d_k  - i(k\pi +c_k) + O(k^{-1}), \quad \forall k\geq k_0 \text{ large enough} , \\
m_{3}= - \frac{1}{2}-d_k + i(k\pi +c_k)+O(k^{-1}) , \quad \forall k\geq k_0  \text{ large enough} . 
\end{cases}
\end{align}
where we have used the expression of $\mu=\mu_k$ from \eqref{eigen-parabolic}. 

Recall the values of $m_1$, $m_2$, $m_3$, given by \eqref{charac_roots_asymp} and from the expression \eqref{func_eta}, we get that 
\begin{align}\label{component_eta_lambda}
{\eta_{\lambda^p_k}(x)}&=\left(e^{-\frac{1}{2}+d_k - i(k\pi +c_k)+O(k^{-1})}- e^{-\frac{1}{2}-d_k  +i(k\pi +c_k)+O(k^{-1})} \right) e^{x\left(-k^2\pi^2 -2c_k k\pi -2id_k k\pi  + O(1) \right)}  \\\notag
&\quad+\left( e^{-\frac{1}{2}-d_k  +i(k\pi +c_k)+O(k^{-1})} -  e^{-k^2\pi^2 -2c_k k\pi -2id_k k\pi  + O(1)} \right) e^{x \left(-i(k\pi +c_k)-\frac{1}{2}+d_k +O(k^{-1}) \right)} \\\notag
&\quad+ \left( e^{-k^2\pi^2 -2c_k k\pi -2id_k k\pi  + O(1)}- e^{-\frac{1}{2}+d_k - i(k\pi +c_k)+O(k^{-1})}\right) e^{x\left(i(k\pi +c_k)-\frac{1}{2}-d_k +O(k^{-1}) \right)},\notag 
\end{align}
for all $x\in (0,1)$ and for all $k\geq k_0$ large enough. 

\smallskip 

{\bf  Computing ${\xi_{\lambda^p_k}}$.}   By using the values of $m_1, m_2, m_3$ from \eqref{charac_roots_asymp}, we have
\begin{align*}
	m_1^2 &= \left(-k^2\pi^2-2c_k k\pi -2 id_k k\pi +O(1) \right)^2 \\
	& = k^4\pi^4 + 4c_k k^3\pi^3 + 4i d_k k^3\pi^3 + O(k^2), \quad \forall k\geq k_0 \text{ large enough} ,\\
	m_2^2 &= \left(-\frac{1}{2}+d_k - i(k\pi +c_k) +O(k^{-1}) \right)^2 \\
	&= -k^2\pi^2 -2c_k k\pi +ik\pi -2id_k k\pi +O(1) , \quad \forall k\geq k_0 \text{ large enough} , \\
	m_3^2 &= \left(-\frac{1}{2}-d_k + i(k\pi +c_k) +O(k^{-1}) \right)^2 \\
	&= -k^2\pi^2 -2c_k k\pi- ik\pi -2id_k k\pi +O(1) , \quad \forall k\geq k_0 \text{ large enough} .
\end{align*}
Also recall, $\mu_k=-\lambda^p_k= k^2\pi^2 +2c_k k\pi + 2id_k k\pi +O(1)$, so that
\begin{align}\label{value_aux_xi-1}
	1+\frac{m_1^2}{\mu_k}&= 1+ \frac{k^4\pi^4 + 4c_k k^3\pi^3 + 4i d_k k^3\pi^3 + O(k^2)}{k^2\pi^2 +2c_k k\pi + 2id_k k\pi +O(1)}\\ \notag 
	&= 1+ k^2\pi^2 +  4id_k k\pi +O(k), \quad \forall k\geq k_0 \text{ large enough} , \\  \label{value_aux_xi-2}
	1+\frac{m_2^2}{\mu_k} &= 1+ \frac{\big(-k^2\pi^2 -2c_k k\pi  -2id_k k\pi\big) +\big(ik\pi +O(1)\big)}{k^2\pi^2+2c_k k\pi +2id_k k\pi +O(1)} \\ \notag 
	&=\frac{i}{k\pi} +O(k^{-2}), \qquad \qquad \forall k\geq k_0 \text{ large enough} ,  \\
	\label{value_aux_xi-3}
	1+\frac{m_3^2}{\mu_k}  &= -\frac{i}{k\pi} + O(k^{-2}), \qquad \qquad  \forall  k\geq k_0   \text{ large enough}.
\end{align}
Using the quantities \eqref{value_aux_xi-1}, \eqref{value_aux_xi-2} and \eqref{value_aux_xi-3} in  the expression \eqref{expression_xi}, we obtain 
\begin{multline}  \label{eigen-xi_lambda}
{\xi_{\lambda^p_k}(x)}= \left(1+ k^2\pi^2 +O(k)\right)\left(e^{ - i(k\pi +c_k)- \frac{1}{2}+d_k +O(k^{-1})}- e^{i(k\pi +c_k)-\frac{1}{2}-d_k+O(k^{-1})} \right) \\  
\quad\quad \times e^{x\left(-k^2\pi^2 -2c_k k\pi - 2 i d_k k\pi  + O(1) \right)}  \\ 
+\left(\frac{i}{k\pi} +O(\frac{1}{k^2})\right) \left( e^{i(k\pi +c_k)+O(k^{-1})-\frac{1}{2}-d_k} -  e^{-k^2\pi^2 -2c_kk\pi -2id_k k\pi + O(1)} \right) e^{x \left(-i(k\pi +c_k)-\frac{1}{2}+d_k +O(k^{-1}) \right)} \\ 
-\left(\frac{i}{k\pi} +O(\frac{1}{k^2})\right)  \left( e^{-k^2\pi^2 -2c_k k\pi-2id_k k\pi  + O(1)}- e^{ - i(k\pi +c_k)-\frac{1}{2}+d_k+O(k^{-1})}\right) e^{x\left(i(k\pi +c_k)-\frac{1}{2}-d_k +O(k^{-1}) \right)}.
\end{multline}
\begin{remark}\label{remark-asymptic-coeffi}
To the function $\xi_{\lambda^p_k}$, given by  \eqref{eigen-xi_lambda}, let us use the condition $\xi_{\lambda^p_k}(0)= \xi_{\lambda^p_k}(1)$, so that one can further deduce that the coefficient $\big(e^{ - i(k\pi +c_k)- \frac{1}{2}+d_k +O(k^{-1})}- e^{i(k\pi +c_k)-\frac{1}{2}-d_k+O(k^{-1})} \big)$ of \\ $e^{x\left( -k^2\pi^2-2c_k k\pi -2id_k k\pi +O(1) \right)}$ in the expressions of $\eta_{\lambda_k^p}$ and $\xi_{\lambda^p_k}$ satisfies the following asymptotic behavior: $$\left(e^{ - i(k\pi +c_k)- \frac{1}{2}+d_k +O(k^{-1})}- e^{i(k\pi +c_k)-\frac{1}{2}-d_k+O(k^{-1})} \right) = O\Big(\frac{1}{k^3}\Big), \quad \text{for large }k\geq k_0.$$

The above fact together with \eqref{eigen-xi_lambda} and \eqref{component_eta_lambda} gives the required expressions for the eigenfunctions as mentioned in \eqref{eigen-1-xi} and \eqref{eigen-1-eta}.
\end{remark}  

\medskip 

\paragraph{\bf Set of eigenfunctions associated with ${\lambda^h_k}$}
For the set of eigenvalues $\{\lambda^h_k\}_{|k|\geq k_0}$ \eqref{eigen-hyperbolic}, we denote the eigenfunctions by $\Phi_{\lambda^h_k}$, where we shall use the notation
\begin{align}\label{eigenfunc-2nd}
\Phi_{\lambda^h_k}(x) = \begin{pmatrix}\xi_{\lambda^h_k}(x) \\ \eta_{\lambda^h_k}(x)\end{pmatrix}, \quad \forall |k|\geq k_0. 
\end{align}

\smallskip 

{\bf Computing ${\eta_{\lambda^h_k}}$.} 
In this case, recall that  $\tilde\mu_k=-\lambda^h_k = 1+\alpha_{1,k} + i (2k \pi +\alpha_{2,k})$, for all $|k|\geq k_0$.
Let us compute $\tilde \mu_k^{1/2}$.  Assume  $\tilde \mu_k^{1/2}= p_k + iq_k$, $p_k , q_k \in \mathbb R$ and $\tilde \mu_k = a_k+ib_k,$  $a_k, b_k \in \mathbb R$, so that 
\begin{align*}
	(p_k + iq_k)^2 = (p^2_k-q^2_k) + i2p_kq_k=  a_k +ib_k,
\end{align*}
From the fact that $q_k= \frac{b_k}{2p_k}$, we have 
\begin{align*}
	p^2_k-q^2_k= a_k \Longrightarrow 4p_k^4 -4a_k p_k^2- b^2_k=0,
\end{align*}
and that yields 
\begin{align*}
	p_k =\left(\frac{\sqrt{a^2_k+b^2_k} + a_k}{2} \right)^{\frac{1}{2}} , \quad  q_k =   \left(\frac{\sqrt{a^2_k+b^2_k} - a_k}{2} \right)^{\frac{1}{2}}.
\end{align*}
Now, 
\begin{align*}
	\sqrt{a_k^2+b_k^2} = \left[(1+\alpha_{1,k})^2 + (2k\pi + \alpha_{2,k})^2\right]^{\frac{1}{2}} = \left[4k^2\pi^2 + O(k)\right]^{\frac{1}{2}} = 2|k\pi| + O(1), \ \ \forall |k|\geq k_0.
\end{align*}
Thus, it follows that 
\begin{align}
	p_k = \sqrt{|k\pi|} + O(|k|^{-\frac{1}{2}}) , \quad 	q_k = \pm \sqrt{|k\pi|} + O(|k|^{-\frac{1}{2}}), \quad \forall |k|\geq k_0.
\end{align}
We get 
\begin{align}\label{values_mu_1/2}
\tilde\mu^{1/2}_k = \sqrt{|k\pi|} + i \sgn(k) \sqrt{|k\pi|} + {O(|k|^{-\frac{1}{2}})}, \quad \forall |k| \geq k_0>0,
\end{align}
(the sign function $\sgn$ has been defined by \eqref{sign-func}). 

Then using the characteristics roots  $m_1, m_2,m_3$, given by \eqref{charac_roots}, we get that
\begin{align}\label{charac_roots-exp}
\begin{cases}
m_1 = - \alpha_{1,k}- i(2k\pi+\alpha_{2,k})+ O(|k|^{-1}), \quad \quad \qquad  \forall |k| \geq k_0 \text{ large enough}, \\
m_2 = -\frac{1}{2} + \sgn(k) \sqrt{|k\pi|} - i\sqrt{|k\pi|} + O(|k|^{-\frac{1}{2}}), \quad \forall |k| \geq k_0 \text{ large enough}, \\
m_3=- \frac{1}{2} - \sgn(k) \sqrt{|k\pi|} +i\sqrt{|k\pi|}+O(|k|^{-\frac{1}{2}}), \quad \forall |k| \geq k_0 \text{ large enough}.
\end{cases}
\end{align}

Using above information,  we now write the expression of $\eta_{\lambda^h_k}(x)$ (we take the formulation after dividing by $k\pi e^{\sqrt{|k\pi|}+\frac{1}{\sqrt{\mod{k}}}}$), given  by
\begin{multline}\label{eigen-eta-gamma}
{\eta_{\lambda^h_k}(x)}=\frac{1}{k\pi e^{ \sqrt{|k\pi|} + \frac{1}{\sqrt{|k|}} } } \left(e^{\sgn(k) \sqrt{|k\pi|} - \frac{1}{2} -i\sqrt{|k\pi|}+O(|k|^{-\frac{1}{2}})}  - e^{-\sgn(k)\sqrt{|k\pi|} - \frac{1}{2} +i\sqrt{|k\pi|} + O(|k|^{-\frac{1}{2}})} \right)\\
\times  e^{-x\left(\alpha_{1,k} + i(2k\pi +\alpha_{2,k}) +O(|k|^{-1})\right)} \\
+\frac{1}{k\pi e^{ \sqrt{|k\pi|} + \frac{1}{\sqrt{|k|}} } }\left(e^{-\sgn(k)\sqrt{|k\pi|} - \frac{1}{2} +i\sqrt{|k\pi|} + O(|k|^{-\frac{1}{2}})} - e^{-\alpha_{1,k} - i(2k\pi +\alpha_{2,k}) + O(|k|^{-1})}   \right)\\
\times e^{x \left(\sgn(k)\sqrt{|k\pi|} - \frac{1}{2} -i\sqrt{|k\pi|} + O(|k|^{-\frac{1}{2}}) \right)}  \\
+\frac{1}{k\pi e^{ \sqrt{|k\pi|} + \frac{1}{\sqrt{|k|}} } } \left(e^{-\alpha_{1,k} - i(2k\pi +\alpha_{2,k}) + O(|k|^{-1})}  - e^{\sgn(k)\sqrt{|k\pi|} - \frac{1}{2} -i\sqrt{|k\pi|} + O(|k|^{-\frac{1}{2}})}  \right)\\
\times  e^{x\left(-\sgn(k)\sqrt{|k\pi|} - \frac{1}{2} +i\sqrt{|k\pi|} + O(|k|^{-\frac{1}{2}})  \right)},
\end{multline}
for all $x\in (0,1)$ and for all $|k|\geq k_0$.

\smallskip 
 
{\bf Computing ${\xi_{\lambda^h_k}}$.} By using the values of $m_1, m_2, m_3$ from \eqref{charac_roots-exp}, we calculate the following quantities for all $|k|\geq k_0$ large enough
\begin{align*}
m_1^2&=(- \alpha_{1,k}- i(2k\pi+\alpha_{2,k})+ O(|k|^{-1}))^2\\
&= -4k^2\pi^2+4ik\pi\alpha_{1,k}+O(k), \\
m_2^2&=(-\frac{1}{2} + \sgn(k) \sqrt{|k\pi|} - i\sqrt{|k\pi|} + O(|k|^{-\frac{1}{2}}))^2\\
&= -\sgn(k) \sqrt{|k\pi|} - 2i\sgn(k) |k\pi| + i  \sqrt{|k\pi|} + O(1),\\
m_3^2&=(- \frac{1}{2} - \sgn(k) \sqrt{|k\pi|} +i\sqrt{|k\pi|}+O(|k|^{-\frac{1}{2}}))^2\\
&=\sgn(k) \sqrt{|k\pi|} -2i \sgn(k) |k\pi| - i \sqrt{|k\pi|} + O(1).
\end{align*}
Next, we compute the following: for all $|k|\geq k_0$ large enough, 
\begin{align}\label{value_aux_gamma-1}
	 1+\frac{m_1^2}{\tilde \mu_k} &= 1+ \frac{-4k^2\pi^2+4ik\pi\alpha_{1,k}+O(k)}{1+\alpha_{1,k} + i (2k\pi +\alpha_{2,k})}\\ \notag 	
		&=1+\frac{(-4k^2\pi^2+4ik\pi\alpha_{1,k}+O(k))(1+\alpha_{1,k}-i(2k\pi+\alpha_{2,k}))}{(1+\alpha_{1,k})^2+(2k\pi+\alpha_{2,k})^2}\\ \notag 
		&=1+\frac{-(1+\alpha_{1,k})4k^2\pi^2+8ik^3\pi^3+O(k^2)}{4k^2\pi^2+O(k)} \\ \notag 
	&=-\alpha_{1,k}+2ik\pi+O(1),
\end{align}
\begin{align}\label{value_aux_gamma-2} 
1+\frac{m_2^2}{\tilde \mu_k}&= 1+ \frac{-\sgn(k) \sqrt{|k\pi|} - 2i\sgn(k) |k\pi| + i  \sqrt{|k\pi|} + O(1)}{1+\alpha_{1,k} + i (2k\pi +\alpha_{2,k})} \\ \notag 
& = 1+ \frac{\Big(-\sgn(k) \sqrt{|k\pi|} - 2i\sgn(k) |k\pi| + i  \sqrt{|k\pi|} + O(1)\Big)\Big(1+\alpha_{1,k} - i (2k\pi+\alpha_{2,k})\Big)}{(1+\alpha_{1,k})^2+(2k\pi+\alpha_{2,k})^2}\\ \notag
&=1+ \frac{-4 k^2\pi^2+2(k\pi)^{3/2}+2i(k\pi)^{3/2}+ O(k)}{4k^2\pi^2 +O(k)}\\ \notag
&=\sgn(k)\frac{1}{2\sqrt{|k\pi|}} + \frac{i}{2\sqrt{|k\pi|}} + O\Big(\frac{1}{|k|}\Big),    
\end{align} 
\begin{align}\label{value_aux_gamma-3} 
1+\frac{m_3^2}{\tilde \mu_k} & =	 1+ \frac{\Big(\sgn(k) \sqrt{|k\pi|} -2i \sgn(k) |k\pi| - i \sqrt{|k\pi|} + O(1) \Big)\Big(1+\alpha_{1,k} - i (2k\pi +\alpha_{2,k})\Big)}{(1+\alpha_{1,k})^2+(2k\pi+\alpha_{2,k})^2}\\ \notag
&= 1+ \frac{-4k^2\pi^2-2(k\pi)^{3/2} - 2i(k\pi)^{3/2} +O(k) }{4k^2\pi^2+O(k)} \\ \notag
&=-\sgn(k)\frac{1}{2\sqrt{|k\pi|}} - \frac{i}{2\sqrt{|k\pi|}} + O\Big(\frac{1}{|k|}\Big) .
\end{align}
Using the quantities \eqref{value_aux_gamma-1}, \eqref{value_aux_gamma-2} and \eqref{value_aux_gamma-3} in  the expression \eqref{expression_xi}, we obtain  the component $\xi_{\lambda^h_k}(x)$, for all $|k|\geq k_0$ (upon  a division by $k\pi e^{\sqrt{|k\pi|}+\frac{1}{\sqrt{|k|}} }$), 
\begin{multline}\label{eigen-xi-gamma}
{\xi_{\lambda^h_k}(x)} =  \left(e^{\sgn(k) \sqrt{|k\pi|} - \frac{1}{2} -i\sqrt{|k\pi|} + O(|k|^{-\frac{1}{2}})}  - e^{-\sgn(k) \sqrt{|k\pi|} - \frac{1}{2} +i\sqrt{|k\pi|} + O(|k|^{-\frac{1}{2}})} \right) \\
\times \frac{(-\alpha_{1,k}+2ik\pi + O(1))}{k\pi e^{ \sqrt{|k\pi|} + \frac{1}{\sqrt{|k|}} } } \times e^{-x\left(\alpha_{1,k} + i(2k\pi +\alpha_{2,k}) +O(|k|^{-1})\right)} \\
+ \bigg(e^{-\sgn(k)\sqrt{|k\pi|} - \frac{1}{2} +i\sqrt{|k\pi|} + O(|k|^{-\frac{1}{2}})}  - e^{-\alpha_{1,k} - i(2k\pi +\alpha_{2,k}) + O(|k|^{-1})}   \bigg) \\
\times \frac{1}{k\pi e^{ \sqrt{|k\pi|} + \frac{1}{\sqrt{|k|}} } }\left(\sgn(k) \frac{1}{2\sqrt{|k\pi|}} + \frac{i}{2\sqrt{|k\pi|}} + O\Big(\frac{1}{|k|}\Big)\right)	\times e^{x \left(\sgn(k) \sqrt{|k\pi|} - \frac{1}{2} -i\sqrt{|k\pi|} + O(|k|^{-\frac{1}{2}})\right)} \\
+\bigg(e^{-\alpha_{1,k} -i(2k\pi +\alpha_{2,k}) + O(|k|^{-1})}  - e^{\sgn(k) \sqrt{|k\pi|} - \frac{1}{2} -i\sqrt{|k\pi|} + O(|k|^{-\frac{1}{2}})} \bigg) \\
\times \frac{1}{k\pi e^{ \sqrt{|k\pi|} + \frac{1}{\sqrt{|k|}} } } \left( -\sgn(k)\frac{1}{2\sqrt{|k\pi|}} - \frac{i}{2\sqrt{|k\pi|}} + O\Big(\frac{1}{|k|}\Big) \right) \times e^{x\left(-\sgn(k) \sqrt{|k\pi|} - \frac{1}{2} +i\sqrt{|k\pi|} + O(|k|^{-\frac{1}{2}}) \right)},
\end{multline}

\smallskip 

We can now prove last part of Lemma \ref{lemma_eigenvalue_eigenfunc}.

\subsection{Proof of Lemma \ref{lemma_eigenvalue_eigenfunc}} \label{Proof-lemma-3.2} 
We have already proved the existence of eigenvalues  $\{\lambda^p_k\}_{k\geq k_0}$ (parabolic part) and $\{\lambda^h_k\}_{|k|\geq k_0}$ (hyperbolic part) by \eqref{eigen-parabolic} and \eqref{eigen-hyperbolic} respectively, which is the first part of Lemma \ref{lemma_eigenvalue_eigenfunc}.

It lefts to show the asymptotic properties of the sequences $\{c_k\}_{k\geq k_0}$, $\{d_k\}_{k\geq k_0}$ and $\{\alpha_{1,k}\}_{|k|\geq k_0}$, $\{\alpha_{1,k}\}_{|k|\geq k_0}$. 

\smallskip 

\begin{itemize} 
\item Let us use the form of $\mu_k$ (i.e., of $-\lambda^p_k$) in the eigenvalue equation \eqref{main_eigen_equ}. Then, for large $k$, it is easy to observe that 
\begin{align*}
F(\mu_k)&=2\sin(k\pi+c_k+id_k)+O(k^{-1})\\
&=2(-1)^k\sin(c_k+id_k)+O(k^{-1}). 
\end{align*}
But $\mu_k$ is a root of $F$ and thus 
\begin{align}\label{prop_sin_c_k_d_k}
\sin(c_k+id_k) =  O(k^{-1}), \quad \text{for large $k\geq k_0$}.
\end{align}
Now, since $\mod{\sin(c_k+id_k)}^2=\sin^2(c_k)+\sinh^2(d_k)$, we can write
\begin{equation*}
\sin^2(c_k),\ \sinh^2(d_k)\leq\frac{C}{k^2},\ \ \forall k\geq k_0 \text{ large}.
\end{equation*}
Therefore, $\displaystyle \mod{c_k}^2,\ \mod{d_k}^2\leq\frac{C}{k^2},\ \ \forall k\geq k_0$, that is to say, 
\begin{equation*}   
c_k,\ d_k = O(k^{-1}), \quad \text{for large $k\geq k_0$}.
\end{equation*}
\item For the hyperbolic part of the eigenvalues $\{\lambda^h_k\}_{|k|\geq k_0}$, we
 can obtain using the property $\xi_{\lambda^h_k}(0)=\xi_{\lambda^h_k}(1)$ ($\xi_{\lambda^h_k}$ is defined by \eqref{eigen-xi-gamma}) that
\begin{align*}
\left(1-e^{-\alpha_{1,k}-i2k\pi - i\alpha_{2,k}+O(\mod{k}^{-1})}\right)+O(\mod{k}^{-1})=0,
\end{align*}
that is,
\begin{align}\label{prop_alpha_k_beta_k}
e^{-\alpha_{1,k}-i\alpha_{2,k}} = 1+O(|k|^{-1}), \quad \text{for large $|k|\geq k_0$}.
\end{align}
that is, there exists a $C>0$ such that
\begin{equation*}
\mod{e^{-\alpha_{1,k}-i\alpha_{2,k}}}\leq\left(1+\frac{C}{\mod{k}}\right),\ \ \forall \mod{k}\geq k_0 \text{ large}. 
\end{equation*}
As a consequence,
\begin{equation*}
e^{-\alpha_{1,k}-i\alpha_{2,k}}\to1,\ \ \text{as }\mod{k}\to+\infty.
\end{equation*}
But both $\alpha_{1,k}$ and $\{\alpha_{2,k}\}$ therefore 
\begin{equation}
\alpha_{1,k},\ \alpha_{2,k}\to0,\ \ \text{as }\mod{k}\to\infty.
\end{equation}
Since $\mod{e^{-\alpha_{1,k}-i\alpha_{2,k}}}=e^{-\alpha_{1,k}}$, we have $\displaystyle \mod{\alpha_{1,k}}\leq\frac{C}{\mod{k}}$,  $\forall\mod{k}\geq k_0$ large and that is
\begin{equation*}
\alpha_{1,k} = O(k^{-1}), \quad \text{for large $|k|\geq k_0$}.
\end{equation*}
Using the above result, we get
\begin{equation*}
e^{-i\alpha_{2,k}}  = 1+O(k^{-1}), \quad \text{for large $|k|\geq k_0$}.
\end{equation*}
But, one has $\mod{e^{-i\alpha_{2,k}}-1}=2|\sin(\alpha_{2,k}/2)|$ and   therefore, 
\begin{equation*}
\mod{\alpha_{2,k}}\leq\frac{C}{\mod{k}}, \quad \text{for large $|k|\geq k_0$}.
\end{equation*}
that is,  $\alpha_{2,k} = O(|k|^{-1})$.

Thus, the proof of Lemma \ref{lemma_eigenvalue_eigenfunc} is complete.
\end{itemize}

\subsection{Proof of Lemma \ref{lemma_bound_eigenfunc}} \label{Proof-Lemma-3.4}
We shall give the sketch of the estimates for $\xi_{\lambda^p_k}$ for $k\geq k_0$ and $\xi_{\lambda^h_k}$ for $|k|\geq k_0$. The formula we use to find the $H^{-s}_{per}$ norm is given by \eqref{norm-H_(-s)}. We present the proof for $0<s<1$. In a similar way, one can prove the estimates for $s\geq 1$. 

\smallskip 

{\em Parabolic Part.} Recall the component $\xi_{\lambda^p_k}$, given by \eqref{eigen-xi_lambda}. In this case, the Fourier coefficients $c_m$, $m\in \mb Z$, are given by 
\begin{align*}
c_m&=\int_{0}^{1}\xi_{\lambda^p_k}(x)e^{2im\pi x}dx\\
&=\left(1+ k^2\pi^2+O(k)\right) O(k^{-3}) \times \frac{e^{\left(-k^2\pi^2 -2c_k k\pi - 2 i d_k k\pi +2im\pi + O(1) \right)}-1}{-k^2\pi^2 -2c_k k\pi - 2 i d_k k\pi+2im\pi  + O(1)} \\  
&+(\frac{i}{k\pi} +O(k^{-2})) \left( e^{i(k\pi +c_k)+O(k^{-1})-\frac{1}{2}-d_k} -  e^{-k^2\pi^2 -2c_kk\pi -2id_k k\pi + O(1)} \right) \\
&\qquad \times \frac{e^{ \left(-i(k\pi +c_k)+2im\pi-\frac{1}{2}+d_k +O(k^{-1}) \right)}-1}{-i(k\pi +c_k)+2im\pi-\frac{1}{2}+d_k +O(k^{-1})} \\
&	+ (-\frac{i}{k\pi} +O(k^{-2}))  \left( e^{-k^2\pi^2 -2c_kk\pi-2id_k k\pi  + O(1)}- e^{ - i(k\pi +c_k)-\frac{1}{2}+d_k+O(k^{-1})}\right) \\
&\qquad \times \frac{e^{\left(i(k\pi +c_k)+2im\pi-\frac{1}{2}-d_k +O(k^{-1}) \right)}-1}{i(k\pi +c_k)+2im\pi-\frac{1}{2}-d_k  +O(k^{-1})}.
\end{align*}
One can observe that  the above Fourier coefficients satisfies the following estimate:
\begin{align*}
\mod{c_m}^2&\leq C\left[\frac{1}{{k}^2}\frac{e^{-k^2\pi^2}+1}{k^4\pi^4+4m^2\pi^2}+\frac{1}{{k}^2}\frac{1+e^{-k^2\pi^2}}{(k\pi-2m\pi)^2 + \frac{1}{4}}	+\frac{1}{{k}^2}\frac{1+e^{-k^2\pi^2}}{(k\pi+2m\pi)^2+\frac{1}{4}}\right], \ \ \forall m\in \mb Z,
\end{align*}
thanks to the properties of the sequences $(c_k)_{k\geq k_0}$ and $(d_k)_{k\geq k_0}$. 

Now, we have from the definition of $H^{-s}_{per}$ norm (see \eqref{norm-H_(-s)}) that 
\begin{align}\label{sum-norm}
&	\norm{\xi_{\lambda^p_k}}^2_{ H^{-s}_{\text{per}}(0,1)}=\sum_{m\in\mathbb{Z}}(1+4\pi^2m^2)^{-s}\mod{c_m}^2 \\  \notag  
&\leq \frac{C}{k^2}\left(1+e^{-k^2\pi^2}\right) \sum_{m\in \mb Z} \frac{(1+4\pi^2m^2)^{-s}}{k^4\pi^4+4m^2\pi^2} \\ \notag 
&\quad +\frac{C}{k^2}\left(1+e^{-k^2\pi^2}\right)\left(\sum_{m\in\mathbb{Z}}\frac{(1+4\pi^2m^2)^{-s}}{(k\pi-2m\pi)^2+\frac{1}{4}}+ \sum_{m\in\mathbb{Z}}\frac{(1+4\pi^2m^2)^{-s}}{(k\pi+2m\pi)^2+\frac{1}{4}}\right),
\end{align}
for some constant $C>0$ that can be chosen largely but independent with respect to $k$.

Observe that, the order of $k$ in the first series of the right hand side of \eqref{sum-norm} is lower than the second and third ones. Thus, we just focus on computing the second sum  
appearing in the right hand side of \eqref{sum-norm}. The others can be estimated in the same way. We see
\begin{align*}
\sum_{m\in \mb Z}\frac{(1+4\pi^2m^2)^{-s}}{(k\pi-2m\pi)^2+\frac14}&\leq 	\sum_{-k<2m<k }\frac{(1+4\pi^2m^2)^{-s}}{(k\pi-2m\pi)^2+\frac14} + 	\sum_{|2m|\geq k}\frac{(4\pi^2m^2)^{-s}}{(k\pi-2m\pi)^2+\frac14}\\
&\leq C(k\pi)^{-2s}\sum_{-k<2m< k}\frac{(\frac{1+4\pi^2m^2}{k^2\pi^2})^{-s}}{(k\pi-2m\pi)^2+\frac14}+C(k\pi)^{-2s}\sum_{|2m|\geq k}\frac{1}{(k\pi-2m\pi)^2+\frac14} \\
&\leq C k^{-2s}.
\end{align*}
Using the above result in \eqref{sum-norm}, we get 
\begin{equation}
\norm{\xi_{\lambda^p_k}}_{ H^{-s}_{\text{per}}(0,1)}\leq C k^{-s-1}, \quad \text{for large } k\geq k_0.
\end{equation}

\smallskip 

{\em Hyperbolic Part.} By following similar approach as above, one can find that 
\begin{equation*}
\norm{\xi_{\lambda^h_k}}_{{H}^{-s}_{\text{per}}(0,1)}\leq C\mod{k}^{-s}, \quad \text{for large } |k|\geq k_0,
\end{equation*}
for some $C>0$, independent in $k$.

\smallskip

\section{Further remarks and conclusion}\label{Conclusion}
In the present work, we have proved the boundary null-controllability of our linearized 1D compressible Navier-Stokes system with a control acting only on density through a Dirichlet condition. The space for the initial conditions we can consider is $H^s_{per}(0,1)\times L^2(0,1)$ for any $s>1/2$ which is a larger space (less regularity assumption)  in the context of controllability of linearized compressible Navier-Stokes type systems. The  spectral analysis and the method of moments give the benefit to deal with this larger space. In fact, the sharp upper bounds of the eigenfunctions and  lower bounds of the observation terms play the lead role to obtain such a null-controllability result. 

On the other hand, when a Dirichlet control is acting on the velocity component (under the assumption that $\rho(t,0)=\rho(t,1)$ on the boundary), we just get the controllability in a subspace of finite codimension $\mathcal H \subset H^{s}_{per}(0,1)\times L^2(0,1)$ for $s>1/2$ using the Ingham-type inequality. This restriction is because it is hard to prove that $\B^*_u\Phi_\lambda \neq 0$ for the finitely many lower frequencies of eigenmodes.

\medskip 

Let us  make some final remarks related to our work. 

\begin{itemize}
	\item {\bf Backward uniqueness.} 
Let us consider the following system  without any control input,
 \begin{align}\label{backward}
	\begin{cases}
		\rho_t+\rho_x +u_{x} =0   &\text{in } (0,T)\times (0,1),\\
		u_{t} -u_{xx} +u_x +\rho_x =0 &\text{in } (0,T)\times (0,1),\\
		\rho(t,0)=\rho(t,1)  &\text{for } t\in (0,T) , \\
		u(t,0)=0, \ \ u(t,1)=0  &\text{for } t\in (0,T) ,\\
		\rho(0,x)=\rho_0(x),\ \ u(0,x)=u_0(x)  &\text{for } x\in(0,1),
	\end{cases}
\end{align}
where the density $\rho$ takes same values on the boundary points with respect to time and we have homogeneous Dirichlet boundary conditions for the velocity $u$. 

  Then,	if the solution $(\rho,u)$ to the system \eqref{backward}  satisfies 
  $$\rho(T,\cdot)=u(T,\cdot)=0 \quad \text{in} \ \ (0,1),$$ 
  then we will have 
  $$\rho_0=u_0=0, \ \ \text{in } (0,1), \quad \text{i.e.,} \ \ \rho(t,x)=u(t,x)=0 \ \ \text{in} \ \ (0,T)\times(0,1). $$
This can be proved from the fact that $(A,D(A))$ defines a strongly continuous semigroup in $L^2(0,1)\times L^2(0,1)$ and the set of eigenfunctions $\mathcal E(A)$ of $A$ forms a Riesz basis for the space $L^2(0,1)\times L^2(0,1)$ (see \Cref{Remark-eigen-A}).

\medskip 

The backward uniqueness property plays an important role in the context of unique continuation and controllability. There are several systems including linearized compressible Navier-Stokes system where the backward uniqueness property is helpful to deduce many important results on approximate controllability; see for instance \cite{Chowdhury15, Teresa-Zuazua, Zuazua-3}.

In this regard, we mention that the backward uniqueness is well-known for the cases when the associated operator forms a $\mathcal C^0$-group (hyperbolic case), for instance the system  
\begin{align*}
	\begin{dcases}
		\rho_{t}+\rho_{x}=0\quad\mbox{in}\quad (0,T)\times(0,1),\\
		\rho(t,0)=\rho(t,1),\quad t\in (0,T), \\
		\rho(0,x)=\rho_{0}(x),\quad x\in (0,1),
	\end{dcases}
\end{align*}
or an analytic semigroup (parabolic case), for instance the system 
\begin{align*}
	\begin{cases}
		u_{t}-u_{xx} =0\quad\mbox{in}\quad (0,T)\times(0,1),\\
		u(t,0)=u(t,1)=0,\quad\,t\in (0,T), \\
		u(0,x)=u_{0}(x),\quad x\in (0,1).
	\end{cases}
\end{align*} 

But in our case, we note that our system \eqref{backward} is of mixed nature (coupling between parabolic and hyperbolic
components) and this is why the backward uniqueness question is interesting from the mathematical point of view.  In fact, for coupled parabolic-hyperbolic system, the backward uniqueness property is a delicate issue; see for instance \cite{Las-Tri-Ren, Ava-Tri-1, Ava-Tri-2}.  This difficulty can be realized  in our case also, since the spectral analysis and showing the {\em Riesz Basis property} of the eigenfunctions are really involved parts of this article.

\medskip 

\item {\bf Growth bound of the semigroup and a stability result when $(\rho_0, u_0)\in \dot L^2(0,1)\times L^2(0,1)$.}  Denote 
\begin{align*}
	\dot L^2(0,1): = \Big\{ \phi \in L^2(0,1) : \int_0^1 \phi =0    \Big\}.
\end{align*}
We shall point out some stability result associated with the system \eqref{backward}  (that is, without any control) when the initial data $(\rho_0, u_0)\in \dot L^2(0,1)\times L^2(0,1)$.
 
In this case, the operator $A$ with its formal expression  \eqref{op_A} has the domain 
\begin{align}\label{domain-cal-D}
		\mathcal D(A)= \Big\{\Phi=(\xi, \eta) \in \dot {H}^1(0,1)\times H^2(0,1) : \xi(0)=\xi(1)  , \ \eta(0)=\eta(1)=0  \Big\},
\end{align}
where $\dot H^1(0,1)$ contains all the functions in  $H^1(0,1)$ with mean zero. 
Similarly, $A^*$ has its formal expression as \eqref{op_A*} with the same domain $\mathcal D(A^*) = \mathcal D(A)$ as of \eqref{domain-cal-D}. 

It is enough to obtain the growth bound of the semigroup $\{S^*(t)\}_{t\geq 0}$ generated by $(A^*, \mathcal D(A^*))$ in $L^2(0,1)\times L^2(0,1)$. Then, using the fact $\|S(t)\|=\|S^*(t)\|$ we can deduce the growth of the semigroup $\{S(t)\}_{t\geq 0}$ generated by $(A,\mathcal D(A))$ (in $L^2(0,1)\times L^2(0,1)$).

\smallskip 

We first ensure that $\lambda=0$ cannot be an eigenvalue of $A^*$ (or $A$) with the domain \eqref{domain-cal-D}. If yes, then the associated eigenfunction will be $(1,0)$, but this is not possible since $(1,0)\notin \mathcal D(A^*)$. Also,  observe that the first component of the eigenfunction of  $A^*$ (or $A$) corresponding to any  eigenvalue has mean zero (in the light of \Cref{Remark-integral-zero}). As a consequence, in this case we can prove that the set of eigenfunctions of $A^*$ (or $A$) with the domain given by \eqref{domain-cal-D} forms a Riesz basis for $\dot L^2(0,1)\times L^2(0,1)$ (using \Cref{Thm_Bao}).  So,  $(A^*, \mathcal D(A^*))$ (or $(A, \mathcal D(A))$) is indeed a Riesz-spectral operator since there is no accumulation point of the set of eigenvalues of $A^*$ (or $A$), see \cite[Chapter 3]{Curtain-Zwart}. 
 
Now in one hand,  since $\lambda \neq 0$,  all the eigenvalues of $A^*$ with domain \eqref{domain-cal-D} have negative real parts (see \eqref{real_part}), i.e.,  
$$ \Re(\lambda) < 0, \quad \forall \lambda \in \sigma(A^*).$$
On the other hand,
thanks to \Cref{lemma_eigenvalue_eigenfunc}, the set of parabolic and hyperbolic branches of the eigenvalues of $A^*$ with domain \eqref{domain-cal-D} have the following asymptotics properties:
\begin{align*}
	\lambda^p_k &= -k^2\pi^2 + O(1), \qquad \text{for large } k \geq k_0, \\
	\lambda^h_k &= -1 -i2k\pi + O(|k|^{-1}), \quad \text{for large } |k| \geq k_0.
\end{align*}
 Thus, there exists some $\omega_0 \in [-1,0)$ such that
\begin{align*}
\omega_0 = \sup \big\{ \Re(\lambda) : \lambda \in \sigma(A) \big\}<0. 
\end{align*}
 Now recall that $(A^*, \mathcal D(A^*))$ is a Riesz-spectral operator and so  the semigroup $\{S^*(t)\}_{t\geq 0}$ generated by $(A^*, \mathcal D(A^*))$ has the following growth 
\begin{align*}
	\|S^*(t)\|\leq C e^{\omega_0 t}, \quad \forall t\geq 0.
\end{align*}
But, $\|S(t)\|=\|S^*(t)\|$ and therefore 
\begin{align*}
	\|S(t)\|\leq C e^{\omega_0 t}, \quad \forall t\geq 0.
\end{align*}
with $-1\leq \omega_0<0$, 
which gives the exponential stability of the system \eqref{backward} with initial data $(\rho_0,u_0)\in \dot L^2(0,1)\times L^2(0,1)$.

\medskip 

\item {\bf A Dirichlet-Dirichlet system with control on velocity.}  Recall that,  when we   considered a Dirichlet boundary control on velocity, then we have the assumption $\rho(t,0)=\rho(t,1)$ for the density part. It would be really interesting to deal with the full Dirichlet case when a control $q$ acts on the velocity, that is the following system 
\begin{align}\label{full-dirichlet}
	\begin{dcases}
		\rho_t+\rho_x +u_{x} =0   &\text{in } (0,T)\times (0,1),\\
		u_{t} -u_{xx} +u_x +\rho_x =0 &\text{in } (0,T)\times (0,1),\\
		\rho(t,0)= 0    &\text{for } t\in (0,T) , \\
		u(t,0)=0, \  u(t,1)= q(t) &\text{for } t\in (0,T) ,\\
		\rho(0,x)=\rho_0(x),\  u(0,x)=u_0(x)  &\text{for } x\in(0,1).
	\end{dcases}
\end{align}
    This is really a challenging issue to handle. As per our observation, the  spectral analysis of the associated adjoint operator is beyond  comprehension. This can be studied as a future work.   

\end{itemize}

\appendix 
\section{Proof of the well-posedness results} \label{Appendix-well-posed}

\medskip 

This section is devoted to  prove the well-posedness of the solution to our control system \eqref{lcnse3}. More precisely, we shall prove  \Cref{lemma_wellposedness} and  \Cref{Thm-existnce-control-sol}.

\subsection{Existence of  semigroup: proof of Lemma \ref{lemma_wellposedness}} \label{Appendix-A1}
The proof is divided  into several parts. Recall the operator $(A, D(A))$ given  by \eqref{op_A}--\eqref{domain_A} and denote $\mathbf Z = L^2(0,1)\times L^2(0,1)$ over the field $\mb{C}$. 

\smallskip 

\textbf{Part 1.} {\em The operator $A$ is dissipative.}  We check that, all $\mathbf{U}=(\rho,u)\in {D}(A)$
\begin{align*}
&\Re\ip{A\mathbf{U}}{\mathbf{U}}_{\mathbf{Z}}=\Re\ip{\vector{-\rho_x-u_x}{-\rho_x+u_{xx}-u_x}}{\vector{\rho}{u}}_{\mathbf Z} \\
&=\Re\left(-\int_{0}^{1}\bar{\rho}\rho_xdx-\int_{0}^{1}\bar{\rho}u_xdx-\int_{0}^{1}\rho_x\bar{u}dx+\int_{0}^{1}\bar{u}u_{xx}dx-\int_{0}^{1}\bar{u}u_xdx\right)\\
&=-\frac{1}{2}\int_{0}^{1}\frac{d}{dx}(\mod{\rho}^2)dx-\int_{0}^{1}\bar{u}_xu_xdx-\frac{1}{2}\int_{0}^{1}\frac{d}{dx}(\mod{u}^2)dx\\
&=-\int_{0}^{1}\mod{u_x}^2dx\leq0,
\end{align*}
\textbf{Part 2.} {\em The operator $A$ is maximal.} This is equivalent to the following. For any $\lambda>0$ and any $\vector{f}{g}\in\mathbf{Z}$ we can find a $\vector{\rho}{u}\in {D}(A)$ such that
\begin{equation}
(\lambda I-A)\vector{\rho}{u}=\vector{f}{g}
\end{equation}
that is
\begin{align*}
\lambda\rho+\rho_x+u_x=f,\\
\lambda u+\rho_x-u_{xx}+u_x=g.
\end{align*}
Let $\epsilon>0$. Instead of solving the above problem, we will solve the following regularized problem
\begin{equation}\label{reg_prb}
\begin{aligned}
\lambda\rho+\rho_x+u_x-\epsilon\rho_{xx}=f,\\
\lambda u+\rho_x+u_x-u_{xx}=g,
\end{aligned}
\end{equation}
with the following boundary conditions
\begin{equation*}
\rho(0)=\rho(1),\ \ \rho_x(0)=\rho_x(1), \ \ u(0)=0,\ \ u(1)=0.
\end{equation*}
We now proceed through the following steps.

\textit{Step 1.} We consider the space $V$, given by
\begin{equation*}
V=\left\{(\rho,u)\in H^1(0,1)\times H^1(0,1) \ : \ \rho(0)=\rho(1),\ \ u(0)=0,\ \ u(1)=0\right\}.
\end{equation*}
Using Lax-Milgram theorem, we first prove that the system \eqref{reg_prb} has a unique solution in $V$. Define the operator $B:V\times V\to\mathbb{C}$ by
\begin{align*}
B\left(\vector{\rho}{u},\vector{\sigma}{v}\right)&=\epsilon\int_{0}^{1}\rho_x\bar{\sigma}_xdx+\int_{0}^{1}u_x\bar{\sigma}dx+\int_{0}^{1}\rho_x\bar{\sigma}dx+\lambda\int_{0}^{1}\rho\bar{\sigma}dx\\
&\quad+\int_{0}^{1}u_x\bar{v}_xdx+\int_{0}^{1}u_x\bar{v}dx+\int_{0}^{1}\rho_x\bar{v}dx+\lambda\int_{0}^{1}u\bar{v}dx,
\end{align*}
for all $\vector{\rho}{u},\vector{\sigma}{v}\in V$. Then, one can show that $B$ is continuous and coercive.
Thus, by Lax-Milgram theorem, for every $\epsilon>0$, there exists a unique solution $(\rho^{\epsilon},u^{\epsilon})\in V$ such that
\begin{equation*}
B\left(\vector{\rho^{\epsilon}}{u^{\epsilon}},\vector{\sigma}{v}\right)=F\left(\vector{\sigma}{v}\right), \quad \forall \vector{\sigma}{v}\in V,
\end{equation*}
where $F:V\to\mathbb{C}$ is the linear functional given by
\begin{equation*}
F\left(\vector{\sigma}{v}\right):=\int_{0}^{1}f\bar{\sigma}dx+\int_{0}^{1}g\bar{v}dx.
\end{equation*}

\textit{Step 2.}
Now, observe that 
\begin{equation*}
\Re\left(B\left(\vector{\rho^{\epsilon}}{u^{\epsilon}},\vector{\rho^{\epsilon}}{u^{\epsilon}}\right)\right)\leq\int_{0}^{1}\mod{f\overline{\rho^{\epsilon}}}+\int_{0}^{1}\mod{g\overline{u^{\epsilon}}}\leq\frac{1}{2}\int_{0}^{1}\left(\mod{f}^2+\mod{\overline{\rho^{\epsilon}}}^2\right)+\frac{1}{2}\int_{0}^{1}\left(\mod{g}^2+\mod{\overline{u^{\epsilon}}}^2\right),
\end{equation*}
which yields 
\begin{equation*}
\epsilon\int_{0}^{1}\mod{\rho^{\epsilon}_x}^2+\frac{\lambda}{2}\int_{0}^{1}\mod{\rho^{\epsilon}}^2+\int_{0}^{1}\mod{u^{\epsilon}_x}^2+\frac{\lambda}{2}\int_{0}^{1}\mod{u^{\epsilon}}^2\leq\frac{1}{2}\int_{0}^{1}\mod{f}^2+\frac{1}{2}\int_{0}^{1}\mod{g}^2
\end{equation*}
This shows that $(u^{\epsilon})_{\epsilon\geq0}$ is bounded  in $H^1(0,1)$, $(\rho^{\epsilon})_{\epsilon\geq0}$ is bounded in $L^2(0,1)$ and $(\sqrt{\epsilon}\rho^{\epsilon}_x)_{\epsilon\geq0}$ is bounded in $L^2(0,1)$. Since the spaces $H^1(0,1)$ and $L^2(0,1)$ are reflexive, there exist subsequences, still denoted by $(u^{\epsilon})_{\epsilon\geq0}$, $(\rho^{\epsilon})_{\epsilon\geq 0}$,
and functions $\rho\in L^2(0,1)$ and $u\in H^1(0,1)$, such that
\begin{equation*}
u^{\epsilon}\rightharpoonup u \text{ in } H^1(0,1),\text{ and } \rho^{\epsilon}\rightharpoonup\rho \text{ in }L^2(0,1).
\end{equation*}
Furthermore, we have
\begin{equation*}
\int_{0}^{1}\mod{\epsilon\rho^{\epsilon}_x}^2=\epsilon\int_{0}^{1}\mod{\sqrt{\epsilon}\rho^{\epsilon}}^2\to0,\text{ as }\epsilon\to0.
\end{equation*}
Now, since $B\left(\vector{\rho^{\epsilon}}{u^{\epsilon}},\vector{\sigma}{v}\right)=F\left(\vector{\sigma}{v}\right)$, for all $\vector{\sigma}{v}\in V$, we may take $\vector{\sigma}{0}\in V$, so that we obtain
\begin{equation}\label{identity1}
\epsilon\int_{0}^{1}\rho^{\epsilon}_x\bar{\sigma}_x+\int_{0}^{1}u^{\epsilon}_x\bar{\sigma}+\int_{0}^{1}\rho^{\epsilon}_x\bar{\sigma}+\lambda\int_{0}^{1}\rho^{\epsilon}\bar{\sigma}=\int_{0}^{1}f\bar{\sigma}.
\end{equation}
Similarly, by taking  $\vector{0}{v}\in V$, we get 
\begin{equation}\label{identity2}
\int_{0}^{1}u^{\epsilon}_x\bar{v}_x+\int_{0}^{1}u^{\epsilon}_x\bar{v}+\int_{0}^{1}\rho^{\epsilon}_x\bar{v}+\lambda\int_{0}^{1}u^{\epsilon}\bar{v}=\int_{0}^{1}g\bar{v}.
\end{equation}
Integrating by parts, we get from equation \eqref{identity1} that,
\begin{equation*}
\epsilon\int_{0}^{1}\rho^{\epsilon}_x\bar{\sigma}_x+\int_{0}^{1}u^{\epsilon}_x\bar{\sigma}-\int_{0}^{1}\rho^{\epsilon}\bar{\sigma}_x+\lambda\int_{0}^{1}\rho^{\epsilon}\bar{\sigma}=\int_{0}^{1}f\bar{\sigma}.
\end{equation*}
Then, passing to the limit $\epsilon\to0$, we obtain
\begin{equation*}
\int_{0}^{1}u_x\bar{\sigma}-\int_{0}^{1}\rho\bar{\sigma}_x+\lambda\int_{0}^{1}\rho\bar{\sigma}=\int_{0}^{1}f\bar{\sigma},
\end{equation*}
and the above relation is true  $\forall \sigma\in \mathcal C_c^{\infty}(0,1)$. As a consequence, 
\begin{equation}
u_x+\rho_x+\lambda\rho=f,
\end{equation}
in the sense of distribution and therefore $\rho_x=f-u_x-\lambda \rho\in L^2(0,1)$; in other words,  $\rho\in H^1(0,1)$.

\smallskip 

\textit{Step 3.} We now show $u(0)=u(1)=0$. Since the inclusion map $i:H^1(0,1)\to \mathcal C^0([0,1])$ is compact and $u^{\epsilon}\rightharpoonup u$ in $H^1(0,1)$, we obtain
\begin{equation*}
u^{\epsilon}\to u     \ \ \text{ in }  \mathcal C^0([0,1]).
\end{equation*}
Thus,  $(u^{\epsilon}(0),u^{\epsilon}(1)) \to (u(0),u(1))$. Since $u^{\epsilon}(0)=u^{\epsilon}(1)=0$ for all $\epsilon>0$, we have $$u(0)=u(1)=0.$$

Similarly from the identity \eqref{identity2}, one can deduce  that 
\begin{equation}\label{u_eqn}
-u_{xx}+u_x+\rho_x+\lambda u=g,
\end{equation}
in the sense of distribution and therefore $u_{xx}\in L^2(0,1)$, that is $u\in H^2(0,1)$. 

We now show $\rho(0)=\rho(1)$. Recall that, $u_x+\rho_x+\lambda \rho=f$ and therefore
\begin{equation*}
\int_{0}^{1}u_x\bar{\sigma}+\int_{0}^{1}\rho_x\bar{\sigma}+\lambda\int_{0}^{1}\rho\bar{\sigma}=\int_{0}^{1}f\bar{\sigma}.
\end{equation*}
Integrating by parts, we get
\begin{equation}\label{compare-1}
\int_{0}^{1}u_x\bar{\sigma}-\int_{0}^{1}\rho\bar{\sigma}_x+\rho\bar{\sigma}|_{0}^{1}+\lambda\int_{0}^{1}\rho\bar{\sigma}=\int_{0}^{1}f\bar{\sigma}.
\end{equation}
From \eqref{identity1}, we deduce
\begin{equation}
\epsilon\int_{0}^{1}\rho^{\epsilon}_x\bar{\sigma}_x+\int_{0}^{1}u^{\epsilon}_x\bar{\sigma}-\int_{0}^{1}\rho^{\epsilon}\bar{\sigma}_x+\lambda\int_{0}^{1}\rho^{\epsilon}\bar{\sigma}=\int_{0}^{1}f\bar{\sigma}.
\end{equation}
Taking $\epsilon\to0$, we get
\begin{equation}\label{compare-2}
\int_{0}^{1}u_x\bar{\sigma}-\int_{0}^{1}\rho\bar{\sigma}_x+\lambda\int_{0}^{1}\rho\bar{\sigma}=\int_{0}^{1}f\bar{\sigma}.
\end{equation}
Comparing \eqref{compare-1} and \eqref{compare-2}, one has $\rho(0)\bar{\sigma}(0)=\rho(1)\bar{\sigma}(1)$. But $\sigma(0)=\sigma(1)$, and thus $$\rho(0)=\rho(1).$$ So, we get $\vector{\rho}{u}\in {D}(A)$.  Hence, the operator $A$ is maximal.

\smallskip 
 
\subsection{Solution by transposition: proof of Theorem \ref{Thm-existnce-control-sol}}\label{Appendix_A2}
In this section, we are going to proof the existence of solution to our control problem \eqref{lcnse3}, more precisely Theorem \ref{Thm-existnce-control-sol}. We omit the proof for \Cref{Thm-existnce-control-sol-vel},  when a control acts on the velocity part.

\smallskip 

\paragraph{\bf Step 1}
We first consider system \eqref{lcnse3} with zero initial data and nonhomogeneous boundary conditions, that is, 
\begin{align}\label{new_bdry}
\begin{dcases}
\rho_t+\rho_x +u_{x} =0   &\text{in } (0,T)\times (0,1),\\
u_{t} -u_{xx} +u_x +\rho_x =0 &\text{in } (0,T)\times (0,1),\\
\rho(t,0)=\rho(t,1)+ p(t)      &\text{for } t\in (0,T) , \\
u(t,0)=0, \ \ u(t,1)=0  &\text{for } t\in (0,T) ,\\
\rho(0,x)= u(0,x)= 0  &\text{for } x\in(0,1), 
\end{dcases}
\end{align}
with $p\in L^2(0,T)$. 

We now prove the existence of solution to the new  system \eqref{new_bdry}. 
\begin{theorem}\label{Thm-step-1}
For a given $p\in L^2(0,T)$, the system \eqref{new_bdry} has a unique solution $(\tilde{\rho},\tilde{u})$ belonging to the space  $L^2(0,T;L^2(0,1))\times L^2(0,T;L^2(0,1))$ in the sense of transposition. Moreover, the operator: $$p\mapsto(\tilde{\rho},\tilde{u}),$$ is linear and continuous from $L^2(0,T)$ into $L^2(0,T;L^2(0,1))\times L^2(0,T;L^2(0,1))$. 
\end{theorem}
\begin{proof}
\textit{Existence:} Let us define a map $\Lambda_1:L^2(0,T;L^2(0,1))\times L^2(0,T;L^2(0,1))\rightarrow L^2(0,T)$, 
\begin{equation}\label{map_lambda_1}
\Lambda_1(f,g)=\sigma(t,1),
\end{equation}
where $(\sigma,v)$ is the unique solution to the adjoint system \eqref{lcnse_adjoint} with given  source term $(f,g)$. The map $\Lambda_1$ is well-defined because of the hidden regularity as mentioned in Appendix \ref{Appendix-D}, Corollary \ref{hidden_reg_adj}.

Now, thanks to  \Cref{Prop-Adjoint}, the map
\begin{equation*}
(f,g)\mapsto(\sigma,v)
\end{equation*}
is linear and continuous from $L^2(0,T;L^2(0,1))\times L^2(0,T;L^2(0,1))$ to $L^2(0,T;L^2(0,1))\times L^2(0,T;H^1_{0}(0,1))$, which implies that the map $\Lambda_1$ given by \eqref{map_lambda_1} is  linear and continuous (Corollary \ref{hidden_reg_adj}). 
	
So, we can define the adjoint to $\Lambda_1$ as follows
\begin{equation}
\Lambda_1^*:L^2(0,T)\to L^2(0,T;L^2(0,1))\times L^2(0,T;L^2(0,1)),
\end{equation}
which is also linear and continuous. 

Let us denote $\Lambda_1^*(p) = (\tilde{\rho},\tilde{u})$. Then, for $(\tilde{\rho},\tilde{u})$,  we have	
\begin{equation*}
\begin{aligned}
\int_{0}^{T}\int_{0}^{1}\tilde \rho(t, x)f(t, x)dxdt+\int_{0}^{T}\int_{0}^{1}\tilde u(t,x) g(t,x) dxdt & =\ip{\Lambda_1^*p}{(f,g)}  \\
&=\ip{p}{\Lambda_1(f,g)}\\
&=\int_{0}^{T} p(t)\sigma(t,1)  dt,
\end{aligned}
\end{equation*}
for every $(f,g)$ in $L^2(0,T;L^2(0,1)) \times L^2(0,T; L^2(0,1))$. Hence for any $p\in L^2(0,T)$, $(\tilde{\rho},\tilde{u})$ is the solution to the system \eqref{new_bdry} in the sense of transposition and 
\begin{equation}\label{esti-bdry-nonhomo}
\begin{aligned}
\|(\tilde{\rho},\tilde{u})\|_{L^2(L^2)\times L^2(L^2)} &=\|\Lambda_1^*(p)\|_{L^2(L^2)\times L^2(L^2)}\\
&\leq \|\Lambda_1^*\| \ \|p\|_{L^2(0,T)}.
\end{aligned}
\end{equation}	

\smallskip 

\textit{Uniqueness:} If $p=0$ on $(0,T)$, we have
\begin{equation*}
\int_{0}^{T}\int_{0}^{1}\rho(t,x)f(t,x)dxdt+\int_{0}^{T}\int_{0}^{1}u(t,x)g(t,x) dxdt = 0,
\end{equation*}
for all $(f,g)\in L^2(0,T;L^2(0,1))\times L^2(0,T;L^2(0,1))$, which gives  $(\rho,u)=(0,0)$ and therefore the solution to the system \eqref{new_bdry} is unique.
\end{proof}

\smallskip 

\paragraph{\bf Step 2}
We now consider the system \eqref{lcnse3} with non-zero initial data and homogeneous boundary conditions and check the existence, uniqueness of solution. The system reads as 
\begin{align}\label{new_initial}
\begin{dcases}
\rho_t+\rho_x +u_{x} =0   &\text{in } (0,T)\times (0,1),\\
u_{t} -u_{xx} +u_x +\rho_x =0 &\text{in } (0,T)\times (0,1),\\
\rho(t,0)=\rho(t,1)     &\text{for } t\in (0,T) , \\
u(t,0)=0, \ \ u(t,1)=0  &\text{for } t\in (0,T) ,\\
\rho(0,x)= \rho_0(x), \  u(0,x)= u_0(x)  &\text{for } x\in(0,1), 
\end{dcases}
\end{align}
with $(\rho_0, u_0)\in L^2(0,1)\times L^2(0,1)$. 
\begin{theorem}\label{Thm-step-2}
For any $(\rho_0, u_0)\in L^2(0,1)\times L^2(0,1)$, the system \eqref{new_initial} has a unique solution $(\check{\rho},\check{u})$ belonging to the space  $L^2(0,T;L^2(0,1))\times L^2(0,T;L^2(0,1))$ in the sense of transposition. Moreover, the operator: $$(\rho_0, u_0) \mapsto(\check{\rho},\check{u}),$$ is linear and continuous from $L^2(0,1) \times L^2(0,1)$ into $L^2(0,T;L^2(0,1))\times L^2(0,T;L^2(0,1))$. 
\end{theorem}
\begin{proof}
\textit{Existence:} Let us define a map $\Lambda_2:L^2(0,T;L^2(0,1))\times L^2(0,T;L^2(0,1))\rightarrow L^2(0,1)\times L^2(0,1)$, 
\begin{equation}\label{map_lambda_2}
\Lambda_2(f,g)=(\sigma(0,\cdot),v(0,\cdot)),
\end{equation}
where $(\sigma,v)$ is the unique solution to the adjoint system \eqref{lcnse_adjoint} with given  source term $(f,g)$.
	
Now, thanks to Proposition \ref{Prop-Adjoint}, the map
\begin{equation*}
(f,g)\mapsto(\sigma,v)
\end{equation*}
is linear and continuous from $L^2(0,T;L^2(0,1))\times L^2(0,T;L^2(0,1))$ to $\mc{C}([0,T];L^2(0,1))\times [\mc{C}([0,T];L^2(0,1))\cap L^2(0,T;H^1_{0}(0,1))]$, which implies that the map $\Lambda_2$ given by \eqref{map_lambda_2} is  linear and continuous. 
	
So, we can define the adjoint to $\Lambda_2$ as follows
\begin{equation}
\Lambda_2^*:L^2(0,1)\times L^2(0,1)\to L^2(0,T;L^2(0,1))\times L^2(0,T;L^2(0,1)),
\end{equation}
which is also linear and continuous. 
	
Let us denote $\Lambda_2^*(\rho_0,u_0) = (\check{\rho},\check{u})$. Then, for $(\check{\rho},\check{u})$,  we have	
\begin{equation*}
\begin{aligned}
\int_{0}^{T}\int_{0}^{1}\check \rho(t, x)f(t, x)dxdt+\int_{0}^{T}\int_{0}^{1}\check u(t,x) g(t,x) dxdt & =\ip{\Lambda_2^*(\rho_0,u_0)}{(f,g)}  \\
&=\ip{(\rho_0,u_0)}{\Lambda_2(f,g)}\\
&=\ip{(\rho_0,u_0)}{(\sigma(0,\cdot),v(0,\cdot))},
\end{aligned}
\end{equation*}
for every $(f,g)$ in $L^2(0,T;L^2(0,1)) \times L^2(0,T; L^2(0,1))$. Hence for any $(\rho_0,u_0)\in L^2(0,1)\times L^2(0,1)$, $(\check{\rho},\check{u})$ is the solution to the system \eqref{new_bdry} and 
\begin{equation}\label{esti-bdry-nonhomo_2}
\begin{aligned}
\|(\check{\rho},\check{u})\|_{L^2(L^2)\times L^2(L^2)} &=\|\Lambda_2^*(\rho_0,u_0)\|_{L^2(L^2)\times L^2(L^2)}\\
&\leq \|\Lambda_2^*\| \ \|(\rho_0,u_0)\|_{L^2(0,1)\times L^2(0,1)}.
\end{aligned}
\end{equation}	
	
\smallskip 
	
\textit{Uniqueness:} Let the system \eqref{new_initial} has two solutions $(\rho_1,u_1)$ and $(\rho_2,u_2)$. Introduce $$(\rho,u)=(\rho_1,u_1)-(\rho_2,u_2).$$
Then one can show that the only possibility is $(\rho,u)=(0,0)$, using the initial and boundary conditions: $\rho(0,x)=u(0,x)=0$ for all $x\in (0,1)$ and $\rho(t,0)=\rho(t,1)$, $u(t,0)=u(t,1)=0$ for all $t\in (0,T)$.
\end{proof}

\begin{proof}[\bf Proof of Theorem \ref{Thm-existnce-control-sol}]
We now recall the system \eqref{lcnse3} with given boundary data $p\in L^2(0,T)$ and initial data $(\rho_0, u_0)\in L^2(0,1)\times L^2(0,1)$. Then, thanks to  Theorem \ref{Thm-step-1} \& \ref{Thm-step-2}, 
\begin{align*}
(\rho, u) : =(\tilde \rho, \tilde u) + (\check \rho, \check u),
\end{align*}  
is the unique solution to \eqref{lcnse3}. 

It remains to prove the continuity estimate of the solution $(\rho,u)$. Let $H:L^2(0,1)\times L^2(0,1)\times L^2(0,T)\to L^2(0,T;L^2(0,1)\times L^2(0,T;L^2(0,1)))$ be defined by
\begin{align}
H(\rho_0,u_0,p)=(\rho,u).
\end{align}
Then $H$ is linear.  Furthermore, using \eqref{esti-bdry-nonhomo} and \eqref{esti-bdry-nonhomo_2}, we get  
\begin{align*}
&\norm{H(\rho_0,u_0,p)}_{L^2(0,T;L^2(0,1))\times L^2(0,T;L^2(0,1))}=\norm{(\tilde \rho, \tilde u) +(\check{\rho},\check{u})}_{L^2(0,T;L^2(0,1))\times L^2(0,T;L^2(0,1))}\\
&\leq\norm{\Lambda^*_1}\norm{p}_{L^2(0,T)}+ \norm{\Lambda^*_2} \|(\rho_0 ,u_0)\|_{L^2(0,1)\times L^2(0,1)}
\\
&\leq C\left(\norm{p}_{L^2(0,T)}+\norm{\rho_0}_{L^2(0,1)}+\norm{u_0}_{L^2(0,1)}\right).
\end{align*}
Finally, the required regularity result \eqref{regularity}--\eqref{continuity_estimate} can be obtained by applying the usual regularity of parabolic equation (with homogeneous boundary data) and then using that, the regularity of transport part follows immediately. 

The proof is complete.
\end{proof}

\section{On the compactness of the resolvent to the adjoint operator}\label{Appendix-A3}

\medskip 

In this section, we are going to prove the part (i) of  Proposition \ref{Prop_spectral}. 

For any $\lambda  \notin \sigma(A^*)$,  denote the resolvent operator associated to $A^*$ by $R(\lambda, A^*): = (\lambda I-A^*)^{-1}$ (where $\sigma(A^*)$ is the spectrum of $A^*$ defined by \eqref{spectrum}). 

Let $\{Y_n\}_{n}=\{(f_n,g_n)\}_{n}$ be a bounded sequence in $\mathbf Z:= L^2(0,1)\times L^2(0,1)$. Our claim is to prove  that for any $\lambda>0$ the sequence $\big\{R(\lambda;A^*)Y_n\big\}_{n}$ contains a convergent subsequence. Let $X_n=(\sigma_n,u_n) = R(\lambda; A^*) Y_n \in {D}(A^*)$, that is
\begin{equation}\label{resolvant_equation}
(\lambda I-A^*)X_n=Y_n.
\end{equation}
More explicitly,
\begin{align}\label{explicit_resolvant_equation}
\begin{dcases}
\lambda\sigma_n-(\sigma_n)_x-(u_n)_x=f_n  \quad &\text{in } (0,1),\\
\lambda u_n-(\sigma_n)_x-(u_n)_x-(u_n)_{xx}=g_n \quad &\text{in } (0,1),\\
\sigma_n(0) = \sigma_n(1), \ \ u_n(0) =u_n(1) =0. 
\end{dcases}
\end{align}
Taking inner product with $X_n$ in the equation \eqref{resolvant_equation}, we get
\begin{equation*}
\lambda\ip{X_n}{X_n}_{\mathbf Z}-\ip{A^*X_n}{X_n}_{\mathbf Z}=\ip{X_n}{Y_n}_{\mathbf Z}.
\end{equation*}
Considering only the real parts, we see
\begin{equation*}
\lambda\norm{X_n}_{\mathbf Z}^2-\Re(\ip{A^*X_n}{X_n}_{\mathbf Z})=\Re(\ip{X_n}{Y_n}_{\mathbf Z}).
\end{equation*}
Now, recall that the operator $A^*$ is dissipative, i.e.,  $\Re(\ip{A^*X_n}{X_n}_{\mathbf Z})\leq0$; in what follows, we have
\begin{equation*}
\lambda\norm{X_n}_{\mathbf Z}^2 \leq \Re(\ip{X_n}{Y_n}_{\mathbf Z})\leq\mod{\ip{X_n}{Y_n}_{\mathbf Z}}\leq\frac{\lambda}{2}\norm{X_n}_{\mathbf Z}^2+\frac{1}{2\lambda}\norm{Y_n}_{\mathbf Z}^2.
\end{equation*}
In other words,
\begin{equation*}
\norm{X_n}_{\mathbf Z}^2\leq\frac{1}{\lambda^2}\norm{Y_n}_{\mathbf Z}^2.
\end{equation*}
Thus, the sequence $\{X_n\}_{n}$ is bounded in $\mathbf Z$. We now prove that $\{X_n\}_{n}$ is in fact bounded in $H^1(0,1)\times H^1_0(0,1)$. Multiplying the second equation of \eqref{explicit_resolvant_equation} by $u_n$, we get
\begin{equation*}
\lambda\int_{0}^{1}\mod{u_n}^2dx-\int_{0}^{1}(\sigma_n)_x\bar{u}_ndx-\int_{0}^{1}(u_n)_{xx}\bar{u}_ndx=\int_{0}^{1}g_n\bar{u}_ndx.
\end{equation*}
Performing an integration by parts, we obtain
\begin{equation*}
\lambda\int_{0}^{1}\mod{u_n}^2dx+\int_{0}^{1}\sigma_n(\bar{u}_n)_xdx+\int_{0}^{1}\mod{(u_n)_x}^2dx=\int_{0}^{1}g_n\bar{u}_n dx,
\end{equation*}
from which, it follows that
\begin{align*}
\lambda\int_{0}^{1}\mod{u_n}^2dx+\int_{0}^{1}\mod{(u_n)_x}^2dx&=\Re\left(\int_{0}^{1}g_n\bar{u}_ndx\right)-\Re\left(\int_{0}^{1}\sigma_n(\bar{u}_n)_xdx\right)\\
&\leq\mod{\int_{0}^{1}g_n\bar{u}_ndx}+\mod{\int_{0}^{1}\sigma_n(\bar{u}_n)_xdx}\\
&\leq\frac{1}{2\lambda}\int_{0}^{1}\mod{g_n}^2dx+\frac{\lambda}{2}\int_{0}^{1}\mod{u_n}^2dx+\frac{1}{2}\int_{0}^{1}\mod{\sigma_n}^2dx+\frac{1}{2}\int_{0}^{1}\mod{(u_n)_x}^2dx .
\end{align*}
After simplification, we have
\begin{equation*}
\frac{\lambda}{2}\int_{0}^{1}\mod{u_n}^2dx+\frac{1}{2}\int_{0}^{1}\mod{(u_n)_x}^2dx\leq\frac{1}{2\lambda}\int_{0}^{1}\mod{g_n}^2dx+\frac{1}{2}\int_{0}^{1}\mod{\sigma_n}^2dx ,
\end{equation*}
that is, the sequence $\{u_n\}_{n}$ is bounded in $H^1_0(0,1)$. Then, the first equation of \eqref{explicit_resolvant_equation}  gives
\begin{equation*}
(\sigma_n)_x=\lambda\sigma_n-(u_n)_x-f_n,
\end{equation*}
which shows that the sequence $\{(\sigma_n)_x\}_{n}$ is bounded in $L^2(0,1)$.  

So, we have proved that  $\{X_n\}_{n}$ is a bounded sequence in $H^1(0,1)\times H^1_0(0,1)$ (which is compactly embedded in $\mathbf Z$) and therefore, $\{X_n\}_{n}$ is relatively compact in $\mathbf Z$. 

This completes the proof.

\medskip 

\section{Solution to the moments problem of mixed  parabolic-hyperbolic type} \label{Appendix-Han}
In this section, we recall an important result by S. W. Hansen \cite{Hansen94} and adapting his idea we solve our joint moments problem \eqref{Moments-parabolic}--\eqref{Moments-hyperbolic} appearing in \Cref{Section-Moments}.  The author made the following assumptions in his work. 
\begin{hypo}\label{assump-han}
Let $\{\lambda_k\}_{k\in \mathbb N^*}$ and $\{\gamma_k\}_{k\in \mathbb Z}$ be two sequences in $\mathbb C$ with the following properties:
\begin{itemize}
\item[(i)] for all $k, j\in \mathbb Z$,  $\gamma_k\neq\gamma_j$ unless $j=k$,
\item[(ii)] $\gamma_k = \beta + c k\pi i + \nu_k$  for all   $k\in \mb Z$, 
\end{itemize}
where $\beta \in \mathbb C$, $c>0$ and $\{\nu_k\}_{k\in \mb Z}\in\ell^2$.  

Also, there exists positive constants $A_0, B_0 , \delta, \epsilon$ and $0\leq\theta<\pi/2$ for which $\{\lambda_k\}_{k\in \mathbb N^*}$ satisfies
\begin{itemize}
\item[(i)] $\mod{\arg(-\lambda_k)}\leq\theta$ for all $k\in\mb{N}^*$,
\item[(ii)] $\mod{\lambda_k- \lambda_j}\geq\delta\mod{k^2-j^2}$ for all $k\neq j$, $k,j\in \mb N^*$,
\item[(iii)] $\epsilon(A_0+ B_0 k^2) \leq \mod{\lambda_k}\leq A_0+ B_0 k^2$ for all  $k\in\mb{N}^*$. 
\end{itemize}
We also assume that	the families are disjoint, i.e., $$ \left\{\gamma_k, \ k\in \mathbb{Z}\right\} \cap \left\{\lambda_k, \ k\in \mathbb{N}^*\right\}=\emptyset. $$ 
\end{hypo}
Then, he introduced the following spaces: for any $0\leq a <b$,
\begin{align*}
W_{[a,b]} &= \text{closed span } \{e^{\gamma_k t}\}_{k\in \mb Z} \text{ in } L^2(a,b), \\
E_{[a,b]} &= \text{closed span } \{e^{-\lambda_k t}\}_{k\in \mb N^*} \text{ in } L^2(a,b).
\end{align*}
With these, the author has proved the following results.
\begin{theorem}\label{thm-han}
Assume that the Hypothesis \ref{assump-han} holds true. Then, for each $T>2/c$, where $c$ is defined as in Hypothesis \ref{assump-han}, the spaces $W_{[0,T]}$ and $E_{[0,T]}$ are uniformly separated. This does not hold for $T\leq 2/c$.
\end{theorem}
The proof relies upon the several results but the main one is the following. We denote $t_c=2/c$.
\begin{lemma}\label{lem-han}
For any $a \in \mb R$, $W_{[a,a+t_c]} = L^2(a, a + t_c)$. Furthermore, for $T \geq t_c$, $\{e^{\gamma_k t}\}_{k\in \mb Z}$ forms a Riesz basis for each of the spaces $W_{[a,a+T]}$.
\end{lemma}
We refer \cite{Hansen94} for the proofs of Theorem \ref{thm-han} and Lemma \ref{lem-han}. 

\smallskip 

Let us write the following set of moments problem,
\begin{align}
\label{moment-para-han}	p_k & = \int_0^T  e^{\lambda_k t} u(t) dt , \qquad  k\in \mb N^*, \\
\label{moment-hyp-han} h_k & = \int_0^T  e^{\gamma_k t} u(t) dt , \qquad  k\in \mb Z.
\end{align}
The space of all sequences $\{p_k\}_{k\in \mb N^*} \cup \{h_k\}_{k\in \mb Z}$ for which there exists a $f\in L^2(0,T)$ that solves the set of equations \eqref{moment-para-han}--\eqref{moment-hyp-han} is called the moment space. 

Now, we recall the following results from the same paper which relate Theorem \ref{thm-han} to the  moments problem \eqref{moment-para-han}--\eqref{moment-hyp-han}. 
\begin{proposition}\label{Prop-sol-moment-hyper}
Let $\{h_k\}_{k\in \mb Z} \in \ell^2$. Then, for any $T\geq t_c$, there exists $f\in W_{[0,T]}$, which solves the moment problem \eqref{moment-hyp-han}. Moreover,  any $\widetilde f \in L^2(0,T)$ given by $\widetilde f= f + \widehat f$ with $\widehat f\in  W^{\perp}_{[0,T]}$ also solves \eqref{moment-hyp-han}.
\end{proposition}
The proof follows as a consequence of Lemma \ref{lem-han}.
\begin{proposition}\label{Prop-sol-moment-para}
Assume that for any $r>0$, the sequence $\{p_k\}_{k\in \mb N^*}$ satisfies 
\begin{align}\label{assump-p_k}
|p_k|e^{r k} \to 0 \quad \text{as } k\to +\infty.
\end{align}
Then, for  given any $\tau>0$, there exists  $g\in  E_{[0,\tau]}$, which solves the moment problem \eqref{moment-para-han}. Moreover, any $\widetilde g \in L^2(0,\tau)$ given by $\widetilde g=g+\widehat g$ with $\widehat g \in E^{\perp}_{[0,\tau]}$ also solves  \eqref{moment-para-han}.
\end{proposition}
The proof of the above proposition is standard. It relies on the existence of bi-orthogonal family in the space $E_{[0,\tau]}$ to the family of exponentials $\{e^{\lambda_k t}\}_{k\in \mb N^*}$; see \cite{Hansen91} for a proof. 

\smallskip

Let us now present the main theorem that tells the solvability of the mixed moment problems \eqref{moment-para-han}--\eqref{moment-hyp-han}. 
\begin{theorem}\label{Thm-main_Han}
Let any $T> t_c$ be given. Then, under  Hypothesis \ref{assump-han}, given any sequence $\{p_k\}_{k\in \mb N^*}$ satisfying \eqref{assump-p_k} and any $\{h_k\}_{k\in \mb Z}\in \ell^2$, there exists a function $f\in L^2(0,T)$ that  simultaneously solves the set of moments problem \eqref{moment-para-han}--\eqref{moment-hyp-han}. This does not hold for $T\leq t_c$.
\end{theorem} 
The proof of above theorem can be found in  \cite[Theorem 4.11]{Hansen94}. For the sake of completeness, we give the proof below. 
\begin{proof}
For $T\leq t_c$, the set of moments problem \eqref{moment-para-han}--\eqref{moment-hyp-han} does not necessarily have a solution. Thus, we start with $T>t_c$. By \Cref{thm-han}, we have $E:=E_{[0,T]}$ and $W:=W_{[0,T]}$ are uniformly separated. Thus the space $V:=E+W$ is closed in $L^2(0,T)$ with its norm $\|\cdot\|_{V}:=\|\cdot\|_{L^2(0,T)}$  and so $V:=E\oplus W$. Moreover, the orthogonal complements $E^\perp$ and $W^\perp$ of $E$ and $W$ (resp.) in $V$ are also uniformly separated using a result by T. Kato \cite[Chap. 4, \S 4]{Kato} and therefore, $V=E^\perp \oplus W^\perp$. From this, one can show that the restrictions $P_{E}|_{W^\perp}$ and $P_W|_{E^\perp}$ are isomorphisms, where $P_E$ and $P_W$ are the orthogonal projections respectively onto $E$ and $W$ in $V$. By Propositions \ref{Prop-sol-moment-para} and \ref{Prop-sol-moment-hyper}, there exists functions $f_1\in E$ and $f_2\in W$ which solve the equations \eqref{moment-para-han} and \eqref{moment-hyp-han} respectively. Set, 
\begin{align*}
f= (P_E|_{W^\perp})^{-1} f_1 + (P_W|_{E^\perp})^{-1}f_2,
\end{align*} 
which simultaneously solves the equations \eqref{moment-para-han}--\eqref{moment-hyp-han} and moreover $f\in L^2(0,T)$. 
\end{proof}

\smallskip

\section{A hidden regularity result}\label{Appendix-D}

\smallskip

Consider the following system
\begin{align}\label{trans_para}
\begin{dcases}
\rho_t+\rho_x +u_{x} =0   &\text{in } (0,T)\times (0,1),\\
u_{t} -u_{xx} +u_x +\rho_x =0 &\text{in } (0,T)\times (0,1),\\
\rho(t,0)=\rho(t,1)+ p(t)      &\text{for } t\in (0,T) , \\
u(t,0)=0,  \ u(t,1)=0  &\text{for } t\in (0,T) ,\\
\rho(0,x)= \rho_0(x), \ u(0,x)= u_0(x)  &\text{for } x\in(0,1), 
\end{dcases}
\end{align}
where $(\rho_0, u_0)\in L^2(0,1)\times L^2(0,1)$ and $p\in  L^2(0,T)$ are given data. Then, one has the following result.

\begin{lemma}\label{lemma-hidden}
For any $(\rho_0, u_0)\in L^2(0,1)\times L^2(0,1)$ and $p\in  L^2(0,T)$, the density component $\rho$ to the system \eqref{trans_para} satisfies $\rho(t,1)\in L^2(0,T)$.
\end{lemma}
\begin{proof} The proof is split into two steps. First, recall \Cref{Thm-existnce-control-sol} so that one has
\begin{align*}
(\rho, u) \in \mathcal C^0 ([0,T]; L^2(0,1)) \times  [ \mathcal C^0([0,T];L^2(0,1)) \cap L^2(0,T;H^1_0(0,1)) ].
\end{align*}  
	
\smallskip 
	
{\bf Step 1.} Let us take the initial state $\rho_0\in H^1_{\{0\}}(0,1)$ (i.e., $\rho_0\in H^1(0,1)$ with $\rho_0(0)=0$), $u_0\in H^1_0(0,1)$ and the boundary data $p\in  H^1_{\{0\}}(0,T)$. Then one can prove that the solution $(\rho,u)$ to system \eqref{trans_para} lies in the space $[H^1(0,T;L^2(0,1))\cap L^2(0,T;H^1(0,1))]\times [L^2(0,T;H^2(0,1)\cap H^1_0(0,1))\cap H^1(0,T;L^2(0,1))]$, see for instance \cite{Chowdhury13}. Therefore, $u_x\in L^2(0,T;H^1(0,1))$ and so the integration by parts are justified. Multiplying the first equation of \eqref{trans_para} by $x\rho$, we get
\begin{equation*}
\int_{0}^{T}\int_{0}^{1}x\rho\rho_tdxdt+\int_{0}^{T}\int_{0}^{1}x\rho\rho_xdxdt+\int_{0}^{T}\int_{0}^{1}x\rho u_xdxdt=0.
\end{equation*}
Integrating by parts and using the boundary conditions, we obtain
\begin{equation}\label{xrho}
\frac{1}{2}\int_{0}^{1}x(\rho^2(T,x)-\rho_0^2(x))dx+\frac{1}{2}\int_{0}^{T}\rho^2(t,1)dt-\frac{1}{2}\int_{0}^{T}\int_{0}^{1}\rho^2dxdt+\int_{0}^{T}\int_{0}^{1}x\rho u_xdxdt=0.
\end{equation}
Therefore
\begin{align*}
\int_{0}^{T}\rho^2(t,1)dt&=-\int_{0}^{1}x(\rho^2(T,x)-\rho_0^2(x))dx+\int_{0}^{T}\int_{0}^{1}\rho^2dxdt-2\int_{0}^{T}\int_{0}^{1}x\rho u_xdxdt\\
&\leq 2\int_{0}^{T}\int_{0}^{1}\rho^2dxdt+\int_{0}^{T}\int_{0}^{1}u_x^2dxdt+\int_{0}^{1}\rho_0^2(x)dx.
\end{align*}
Using the continuity estimate \eqref{continuity_estimate}, we obtain
\begin{equation}\label{esti-Step-1} 
\int_{0}^{T}\rho^2(t,1)dt\leq C\left(\int_{0}^{1}\rho_0^2(x)dx+\int_{0}^{1}u_0^2(x)dx+\int_{0}^{T}p^2(t)dt\right).
\end{equation}	

\smallskip 

{\bf Step 2.} Let $(\rho_0,u_0)\in L^2(0,1)\times L^2(0,1)$ and $p\in  L^2(0,T)$. By density, there exists sequences $\rho_0^n\in H^1_{\{0\}}(0,1)$, $u_0^n\in H^1_0(0,1)$ and $p^n \in  H^1_{\{0\}}(0,T)$ such that $\rho_0^n \to \rho$, $u_0^n\to u_0$ in $L^2(0,1)$ and $p^n\to p$ in $L^2(0,T)$. Let $(\rho^n,u^n)$ be the solution to \eqref{trans_para} corresponding to the initial state $(\rho_0^n,u_0^n)$ and boundary data $p^n$. Using \eqref{esti-Step-1} from Step 1, we have 
\begin{equation*}
\int_{0}^{T}(\rho^n)^2(t,1)dt\leq C\left(\int_{0}^{1}(\rho_0^n)^2(x)dx+\int_{0}^{1}(u_0^n)^2(x)dx+\int_{0}^{T}(p^n)^2(t)dt\right).
\end{equation*}
We  first observe that 
\begin{align*}
	\int_{0}^{1}(\rho_0^n)^2(x)dx+\int_{0}^{1}(u_0^n)^2(x)dx+\int_{0}^{T}(p^n)^2(t)dt \to \int_{0}^{1}\rho_0^2(x)dx+\int_{0}^{1}u_0^2(x)dx+\int_{0}^{T}p^2(t)dt,
\end{align*}
as $n\to +\infty$. Therefore, the sequence $\displaystyle \Big(\int_0^T (\rho^n)^2(t,1)dt\Big)_{n}$ is indeed a Cauchy sequence and hence convergent. 
 Then, by the uniqueness of solution to \eqref{trans_para}, we have  $\displaystyle \lim_{n \to +\infty}\int_{0}^{T}(\rho^n)^2(t,1)dt = \int_0^T \rho^2(t,1) dt$,  which yields 
\begin{equation*}
\int_{0}^{T}\rho^2(t,1)dt\leq C\left(\int_{0}^{1}\rho_0^2(x)dx+\int_{0}^{1}u_0^2(x)dx+\int_{0}^{T}p^2(t)dt\right).
\end{equation*}
This concludes the  proof of the lemma.  
\end{proof}

Let us now consider the following system
\begin{equation}\label{adj_sys}
\begin{cases}
-\sigma_t - \sigma_x - v_{x}=f   &\text{in } (0,T)\times (0,1),\\
-v_{t} - v_{xx} - v_x - \sigma_x = g &\text{in } (0,T)\times (0,1),\\
\sigma(t,0)=\sigma(t,1)     &\text{for } t\in(0,T), \\
v(t,0)=v(t,1)=0            &\text{for } t\in(0,T), \\
\sigma(T,x)=0,\ \ v(T,x)=0  &\text{for } x\in(0,1),
\end{cases}
\end{equation}
with $f,g\in L^2(0,T;L^2(0,1))$. We can similarly conclude the following result.
\begin{corollary}\label{hidden_reg_adj}
For any $f,g\in L^2(0,T;L^2(0,1))$, the solution component $\sigma$ to the adjoint system \eqref{adj_sys} satisfies the following estimate.
\begin{equation}
\norm{\sigma(t,1)}_{L^2(0,T)}\leq C\left(\norm{f}_{L^2(0,T;L^2(0,1))}+\norm{g}_{L^2(0,T;L^2(0,1))}\right).
\end{equation}
\end{corollary}
\smallskip 

\section*{Acknowledgements}
The work of the first author is partially supported by the French Government research program ``Investissements d’Avenir" through the IDEX-ISITE initiative 16-IDEX-0001 (CAP 20-25). He also thanks to the Department of Mathematics \& Statistics, IISER Kolkata for the hospitality and support during his visit in Spring semester, 2021.   S. Chowdhury and R. Dutta acknowledge the DST-INSPIRE grants  (ref. no. DST/INSPIRE/04/2014/002199 \& DST/INSPIRE/04/2015/002388 resp.) for the financial supports. The work of J. Kumbhakar is supported by the Prime Minister's Research Fellowship (ref. no. 41-1/2018-TS-1/PMRF), Government of India.

\bigskip 

\bibliographystyle{plain} 
\bibliography{references}

\begin{thebibliography}{10}

\bibitem{Amosova11}
E.~V. Amosova.
\newblock Exact local controllability for equations of the dynamics of a
  viscous gas.
\newblock {\em Differ. Uravn.}, 47(12):1754--1772, 2011.

\bibitem{Arfaoui11}
H.~Arfaoui, F.~Ben~Belgacem, H.~El~Fekih, and J.-P. Raymond.
\newblock Boundary stabilizability of the linearized viscous {S}aint-{V}enant
  system.
\newblock {\em Discrete Contin. Dyn. Syst. Ser. B}, 15(3):491--511, 2011.

\bibitem{Ava-Tri-1}
G.~Avalos and R.~Triggiani.
\newblock Backward uniqueness of the s.c. semigroup arising in
  parabolic-hyperbolic fluid-structure interaction.
\newblock {\em J. Differential Equations}, 245(3):737--761, 2008.

\bibitem{Ava-Tri-2}
G.~Avalos and R.~Triggiani.
\newblock Backwards uniqueness of the {$C_0$}-semigroup associated with a
  parabolic-hyperbolic {S}tokes-{L}am\'{e} partial differential equation
  system.
\newblock {\em Trans. Amer. Math. Soc.}, 362(7):3535--3561, 2010.

\bibitem{Badra16}
M.~Badra, S.~Ervedoza, and S.~Guerrero.
\newblock Local controllability to trajectories for non-homogeneous
  incompressible {N}avier-{S}tokes equations.
\newblock {\em Ann. Inst. H. Poincar\'{e} Anal. Non Lin\'{e}aire},
  33(2):529--574, 2016.

\bibitem{Beauchard20}
K.~Beauchard, A.~Koenig, and K.~Le~Balc'h.
\newblock Null-controllability of linear parabolic transport systems.
\newblock {\em J. \'{E}c. polytech. Math.}, 7:743--802, 2020.

\bibitem{Cerpa-Crepeau}
E.~Cerpa and E.~Cr\'{e}peau.
\newblock Rapid exponential stabilization for a linear {K}orteweg-de {V}ries
  equation.
\newblock {\em Discrete Contin. Dyn. Syst. Ser. B}, 11(3):655--668, 2009.

\bibitem{Chaves14}
F.~W. Chaves-Silva, L.~Rosier, and E.~Zuazua.
\newblock Null controllability of a system of viscoelasticity with a moving
  control.
\newblock {\em J. Math. Pures Appl. (9)}, 101(2):198--222, 2014.

\bibitem{Chowdhury15}
S.~Chowdhury.
\newblock Approximate controllability for linearized compressible barotropic
  {N}avier-{S}tokes system in one and two dimensions.
\newblock {\em J. Math. Anal. Appl.}, 422(2):1034--1057, 2015.

\bibitem{Chowdhury21}
S.~Chowdhury, R.~Dutta, and S.~Majumdar.
\newblock Boundary stabilizability of the linearized compressible
  {N}avier-{S}tokes system in one dimension by backstepping approach.
\newblock {\em SIAM J. Control Optim.}, 59(3):2147--2173, 2021.

\bibitem{Chowdhury15a}
S.~Chowdhury, D.~Maity, M.~Ramaswamy, and J.-P. Raymond.
\newblock Local stabilization of the compressible {N}avier-{S}tokes system,
  around null velocity, in one dimension.
\newblock {\em J. Differential Equations}, 259(1):371--407, 2015.

\bibitem{Chowdhury15b}
S.~Chowdhury and D.~Mitra.
\newblock Null controllability of the linearized compressible {N}avier-{S}tokes
  equations using moment method.
\newblock {\em J. Evol. Equ.}, 15(2):331--360, 2015.

\bibitem{Chowdhury14}
S.~Chowdhury, D.~Mitra, M.~Ramaswamy, and M.~Renardy.
\newblock Null controllability of the linearized compressible {N}avier {S}tokes
  system in one dimension.
\newblock {\em J. Differential Equations}, 257(10):3813--3849, 2014.

\bibitem{Chowdhury13}
S.~Chowdhury and M.~Ramaswamy.
\newblock Optimal control of linearized compressible {N}avier-{S}tokes
  equations.
\newblock {\em ESAIM Control Optim. Calc. Var.}, 19(2):587--615, 2013.

\bibitem{Chowdhury12}
S.~Chowdhury, M.~Ramaswamy, and J.-P. Raymond.
\newblock Controllability and stabilizability of the linearized compressible
  {N}avier-{S}tokes system in one dimension.
\newblock {\em SIAM J. Control Optim.}, 50(5):2959--2987, 2012.

\bibitem{Coron07}
J.-M. Coron.
\newblock {\em Control and nonlinearity}, volume 136 of {\em Mathematical
  Surveys and Monographs}.
\newblock American Mathematical Society, Providence, RI, 2007.

\bibitem{Coron16}
J.-M. Coron, L.~Hu, and G.~Olive.
\newblock Stabilization and controllability of first-order integro-differential
  hyperbolic equations.
\newblock {\em J. Funct. Anal.}, 271(12):3554--3587, 2016.

\bibitem{Coron14}
J.-M. Coron and P.~Lissy.
\newblock Local null controllability of the three-dimensional {N}avier-{S}tokes
  system with a distributed control having two vanishing components.
\newblock {\em Invent. Math.}, 198(3):833--880, 2014.

\bibitem{Curtain-Zwart}
R.~Curtain and H.~Zwart.
\newblock {\em Introduction to infinite-dimensional systems theory}, volume~71
  of {\em Texts in Applied Mathematics}.
\newblock Springer, New York, [2020] \copyright 2020.
\newblock A state-space approach.

\bibitem{Teresa-Zuazua}
L.~de~Teresa and E.~Zuazua.
\newblock Controllability of the linear system of thermoelastic plates.
\newblock {\em Adv. Differential Equations}, 1(3):369--402, 1996.

\bibitem{Edward06}
J.~Edward.
\newblock Ingham-type inequalities for complex frequencies and applications to
  control theory.
\newblock {\em J. Math. Anal. Appl.}, 324(2):941--954, 2006.

\bibitem{Ervedo16}
S.~Ervedoza, O.~Glass, and S.~Guerrero.
\newblock Local exact controllability for the two- and three-dimensional
  compressible {N}avier-{S}tokes equations.
\newblock {\em Comm. Partial Differential Equations}, 41(11):1660--1691, 2016.

\bibitem{Ervedoza12}
S.~Ervedoza, O.~Glass, S.~Guerrero, and J.-P. Puel.
\newblock Local exact controllability for the one-dimensional compressible
  {N}avier-{S}tokes equation.
\newblock {\em Arch. Ration. Mech. Anal.}, 206(1):189--238, 2012.

\bibitem{Ervedoza18}
S.~Ervedoza and M.~Savel.
\newblock Local boundary controllability to trajectories for the 1{D}
  compressible {N}avier {S}tokes equations.
\newblock {\em ESAIM Control Optim. Calc. Var.}, 24(1):211--235, 2018.

\bibitem{Fernandez04}
E.~Fern\'{a}ndez-Cara, S.~Guerrero, O.~Yu. Imanuvilov, and J.-P. Puel.
\newblock Local exact controllability of the {N}avier-{S}tokes system.
\newblock {\em J. Math. Pures Appl. (9)}, 83(12):1501--1542, 2004.

\bibitem{Girinon08}
V.~Girinon.
\newblock {\em Quelques problèmes aux limites pour les équations de
  Navier-Stokes}.
\newblock 2008.
\newblock DOCTORAT de l’Université Paul Sabatier Toulouse III.

\bibitem{Bao1}
B.-Z. Guo.
\newblock Riesz basis approach to the stabilization of a flexible beam with a
  tip mass.
\newblock {\em SIAM J. Control Optim.}, 39(6):1736--1747, 2001.

\bibitem{Hansen91}
S.~W. Hansen.
\newblock Bounds on functions biorthogonal to sets of complex exponentials;
  control of damped elastic systems.
\newblock {\em J. Math. Anal. Appl.}, 158(2):487--508, 1991.

\bibitem{Hansen94}
S.~W. Hansen.
\newblock Boundary control of a one-dimensional linear thermoelastic rod.
\newblock {\em SIAM J. Control Optim.}, 32(4):1052--1074, 1994.

\bibitem{Ingham}
A.~E. Ingham.
\newblock Some trigonometrical inequalities with applications to the theory of
  series.
\newblock {\em Math. Z.}, 41(1):367--379, 1936.

\bibitem{Kato}
T.~Kato.
\newblock {\em Perturbation theory for linear operators}.
\newblock Classics in Mathematics. Springer-Verlag, Berlin, 1995.
\newblock Reprint of the 1980 edition.

\bibitem{Komornik-Ingham}
V.~Komornik and G.~Tenenbaum.
\newblock An {I}ngham-{M}\"{u}ntz type theorem and simultaneous observation
  problems.
\newblock {\em Evol. Equ. Control Theory}, 4(3):297--314, 2015.

\bibitem{Las-Tri-Ren}
I.~Lasiecka, M.~Renardy, and R.~Triggiani.
\newblock Backward uniqueness for thermoelastic plates with rotational forces.
\newblock {\em Semigroup Forum}, 62(2):217--242, 2001.

\bibitem{Lebeau98}
G.~Lebeau and E.~Zuazua.
\newblock Null-controllability of a system of linear thermoelasticity.
\newblock {\em Arch. Rational Mech. Anal.}, 141(4):297--329, 1998.

\bibitem{Montes99}
A.~L{\'o}pez.
\newblock {\em Control y perturbaciones singulares de sistemas
  parab{\'o}licos}.
\newblock PhD thesis, Universidad Complutense de Madrid, 1999.

\bibitem{Debayan15}
D.~Maity.
\newblock Some controllability results for linearized compressible
  {N}avier-{S}tokes system.
\newblock {\em ESAIM Control Optim. Calc. Var.}, 21(4):1002--1028, 2015.

\bibitem{Martin13}
P.~Martin, L.~Rosier, and P.~Rouchon.
\newblock Null controllability of the structurally damped wave equation with
  moving control.
\newblock {\em SIAM J. Control Optim.}, 51(1):660--684, 2013.

\bibitem{Micu04}
S.~Micu and E.~Zuazua.
\newblock An introduction to the controllability of partial differential
  equations.
\newblock 2004.

\bibitem{Mitra15}
D.~Mitra, M.~Ramaswamy, and J.-P. Raymond.
\newblock Largest space for the stabilizability of the linearized compressible
  {N}avier-{S}tokes system in one dimension.
\newblock {\em Math. Control Relat. Fields}, 5(2):259--290, 2015.

\bibitem{Mitra17}
D.~Mitra, M.~Ramaswamy, and J.-P. Raymond.
\newblock Local stabilization of compressible {N}avier-{S}tokes equations in
  one dimension around non-zero velocity.
\newblock {\em Adv. Differential Equations}, 22(9-10):693--736, 2017.

\bibitem{Rosier97}
L.~Rosier.
\newblock Exact boundary controllability for the {K}orteweg-de {V}ries equation
  on a bounded domain.
\newblock {\em ESAIM Control Optim. Calc. Var.}, 2:33--55, 1997.

\bibitem{Young}
R.~M. Young.
\newblock {\em An introduction to nonharmonic {F}ourier series}.
\newblock Academic Press, Inc., San Diego, CA, first edition, 2001.

\bibitem{Zhang-Zuazua-1}
X.~Zhang and E.~Zuazua.
\newblock Control, observation and polynomial decay for a coupled heat-wave
  system.
\newblock {\em C. R. Math. Acad. Sci. Paris}, 336(10):823--828, 2003.

\bibitem{Zhang-Zuazua-2}
X.~Zhang and E.~Zuazua.
\newblock Polynomial decay and control of a {$1-d$} hyperbolic-parabolic
  coupled system.
\newblock {\em J. Differential Equations}, 204(2):380--438, 2004.

\bibitem{Zuazua-3}
E.~Zuazua.
\newblock Controllability of the linear system of thermoelasticity.
\newblock {\em J. Math. Pures Appl. (9)}, 74(4):291--315, 1995.

\end{thebibliography}
\end{document}